\documentclass[10pt,a4paper,reqno]{amsart}
\usepackage{mystyle}

\def\firstauth#1{}
\def\second#1{}
\def\third#1{}

\begin{document}

\title{Colour-biased Hamilton cycles in dense graphs and random~graphs}

\author[Behague]{Natalie Behague}
\author[Chakraborti]{Debsoumya Chakraborti}
\author[León]{Jared León}

\email{\{natalie.behague, debsoumya.chakraborti, jared.leon\}@warwick.ac.uk}
\address{Mathematics Institute, University of Warwick, Coventry, UK}

\thanks{NB and DC were supported by the European Research Council (ERC) under the European Union Horizon 2020 research and innovation programme (grant agreement No.\ 947978). JL is supported by the Warwick Mathematics Institute Centre for Doctoral Training.}

\begin{abstract}
A classical result of Dirac says that every $n$-vertex graph with minimum degree at least $\frac{n}{2}$ contains a Hamilton cycle. A `discrepancy' version of Dirac's theorem was shown by Balogh--Csaba--Jing--Pluh\'ar, Freschi--Hyde--Lada--Treglown, and Gishboliner--Krivelevich--Michaeli as follows. Every $r$-colouring of the edge set of every $n$-vertex graph with minimum degree at least $(\frac{1}{2} + \frac{1}{2r} + o(1))n$ contains a Hamilton cycle where one of the colours appears at least $(1+o(1))\frac{n}{r}$ times. In this paper, we generalize this result by asymptotically determining the maximum possible value $f_{r,\alpha}(n)$ for every $\alpha \in [\frac{1}{2}, 1]$ such that every $r$-colouring of the edge set of every $n$-vertex graph with minimum degree at least $\alpha n$ contains a Hamilton cycle where one of the colours appears at least $f_{r,\alpha}(n)$ times. In particular, we show that $f_{r,\alpha}(n) = (1-o(1)) \min\set{(2\alpha - 1)n, \frac{2\alpha n}{r}, \frac{2n}{r+1}}$ for every $\alpha\in [\frac{1}{2} + \frac{1}{2r}, 1]$. (It is known that $f_{r,\alpha}(n) = \frac{n}{r}$ for every $\alpha\in [\frac{1}{2},\frac{1}{2} + \frac{1}{2r}]$ whenever $2r$ divides $n$.)

A graph $H$ is called an $\alpha$-residual subgraph of a graph $G$ if $d_H(v)\ge \alpha d_{G[V(H)]}(v)$ for every $v\in V(H)$. Extending Dirac's theorem in the setting of random graphs, Lee and Sudakov showed the following. The Erd\H{o}s--R\'enyi random graph $G(n,p)$, with $p$ above the Hamiltonicity threshold, typically has the property that every $(\frac{1}{2} +o(1))$-residual spanning subgraph contains a Hamilton cycle. Motivated by this, we prove the following random version of our `discrepancy' result. The random graph $G \sim G(n,p)$, with $p$ above the Hamiltonicity threshold, typically satisfies that every $r$-colouring of the edge set of every $\alpha$-residual spanning subgraph of $G$ contains a Hamilton cycle where one of the colours appears at least $f_{r,\alpha}(n)$ times. This is asymptotically optimal and strengthens a result of Gishboliner, Krivelevich, and Michaeli. We also obtain a more general hitting time result. 
\end{abstract}

\maketitle

\section{Introduction}
Determining whether a graph contains a Hamilton cycle (i.e., a cycle that visits every vertex of the graph) is one of the classical NP-complete problems listed in Karp's seminal paper~\cite{Karp1972reducibility}. Therefore, it has been an important direction of research in graph theory to study natural conditions that ensure Hamiltonicity. An early and influential such result was obtained in 1952 by Dirac~\cite{Dirac1952some} who showed that every $n$-vertex graph with minimum degree at least $\frac{n}{2}$ contains a Hamilton cycle. This minimum degree condition is best possible. Inspired by Dirac's theorem, numerous results have appeared in the literature throughout the last few decades (see, e.g., \cite{Bondy1995basic,Bottcher2009proof,Cuckler2009entropy,Cuckler2009hamilton,Krivelevich2014robust}).

\subsection{`Discrepancy' results on Hamilton cycles}
\emph{Discrepancy} is a fundamental notion in mathematics that studies the `irregularity' in a distribution. This notion has found applications in diverse areas such as ergodic theory, number theory, statistics, geometry, and computer science.
In combinatorial discrepancy theory, the central question is to study, for any vertex-colouring of a given hypergraph $\mathcal{H}$ with two colours, existence of an unbalanced hyperedge (i.e., a hyperedge with significantly more vertices in one colour than the other).
In this paper, we study such questions in the context of graphs, where the vertex set of $\mathcal{H}$ is the edge set of some given graph $G$ and the hyperedges of $\mathcal{H}$ correspond to subgraphs of $G$ with a given type. These questions have a rich and influential history dating back to Erd\H{o}s~\cite{Erdos1995discrepancy,Erdos1971imbalances}.

We will consider colourings with two or more colours. An \defin{$r$-edge-colouring} of a graph $G$ refers to an assignment of $r$ colours to the edge set of $G$.
In this paper, we are interested in the following problem. Given graphs $G$ and $H$, determine the largest $t$ such that any $r$-edge-colouring of $G$ contains a copy of $H$ with at least $t$ edges of the same colour. This problem has been recently studied in \cite{Balogh2020on,Freschi2021a,Gishboliner2022discrepancies} when $H$ is a Hamilton cycle or a spanning tree. For Hamilton cycles, the following result is established.
\begin{theorem}[\cite{Balogh2020on,Freschi2021a,Gishboliner2022discrepancies}] \label{thm:colour biased Ham cycle with min degree condition}
    Let $r\ge 2$.
    For all $\epsilon >0$, there exists $\delta >0$ such that for every $r$-edge-colouring of every $n$-vertex graph with minimum degree at least $\left(\frac{1}{2} + \frac{1}{2r} + \epsilon\right)n$, there is a Hamilton cycle with one colour appearing at least $\left(\frac{1}{r} + \delta\right)n$ times. 
\end{theorem} 

This was shown to be best possible by constructing $n$-vertex graphs with minimum degree at least $\left(\frac{1}{2} + \frac{1}{2r}\right)n$ with an $r$-edge-colouring of it such that there is no Hamilton cycle where some colour appears more than $\frac{n}{r}$ times. In this context, there is a recent stability result~\cite{Chen2025optimal}. In this paper, we prove the following result, where we assume a more general condition on the minimum degree.
\begin{restatable}{theorem}{deterministic}\label{thm:deterministic main}
  Let $r\ge 2$ and $\alpha\in [\frac{1}{2} + \frac{1}{2r}, 1]$. For all $\epsilon >0$, there exists $n_\eps$ such that the following holds for every $n\ge n_\eps$. 
  For every $r$-edge-colouring of every $n$-vertex graph with minimum degree at least $\alpha n$, there is a Hamilton cycle with one colour appearing at least $(1-\eps) \min\set{(2\alpha - 1)n, \frac{2\alpha n}{r}, \frac{2n}{r+1}}$ times.
\end{restatable}
The above result recovers \Cref{thm:colour biased Ham cycle with min degree condition}. We also show that \Cref{thm:deterministic main} is best possible up to the $\eps$ error term. We provide constructions supporting this in the next section. We remark that the analogous results for perfect matchings are easy to derive for all the results mentioned in the introduction. Here, we state the one analogous to \Cref{thm:deterministic main}. 
\begin{corollary}\label{cor:main for perfect matching}
  Let $r\ge 2$ and $\alpha\in [\frac{1}{2} + \frac{1}{2r}, 1]$.  For all $\epsilon >0$, there exists $n_\eps$ such that the following holds for every even $n\ge n_\eps$. 
  For every $r$-edge-colouring of every $n$-vertex graph with minimum degree at least $\alpha n$, there is a perfect matching with one colour appearing at least $(1-\eps) \min\set{(2\alpha - 1)\frac{n}{2}, \frac{\alpha n}{r}, \frac{n}{r+1}}$ times.
\end{corollary}
\begin{proof}
    By applying \Cref{thm:deterministic main}, there exists a Hamilton cycle with one colour appearing at least $k:= (1-\eps) \min\set{(2\alpha - 1)n, \frac{2\alpha n}{r}, \frac{2n}{r+1}}$ times. Fix such a Hamilton cycle and say, without loss of generality, that $k$ of its edges are blue. Now, consider the two perfect matchings obtained by taking alternate edges of the cycle. One of these perfect matchings contains at least half of the blue edges, i.e., $\frac{k}{2}$ blue edges. This finishes the proof.
\end{proof}

\subsection{Resilience of Hamilton cycles in random graphs}
Studying Hamiltonicity has been an active area of research in the field of random graphs. In this paper, we will focus on the Erd\H{o}s--R\'enyi random graph $G(n,p)$ on $n$ vertices where every possible edge is present independently with probability $p$.
The study of Hamiltonicity in the random graph $G(n,p)$ was pioneered by the breakthrough work of P\'osa~\cite{Posa1976Hamiltonian}, who introduced the versatile \emph{P\'osa's rotation-extension technique}, and was subsequently advanced by Kor\v{s}unov~\cite{Korsunov1977solution}. Their results were sharpened by Bollob\'as~\cite{Bollobas1984the} and Koml\'os--Szemer\'edi~\cite{Komlos1983limit} who proved that the threshold for $G(n,p)$ to contain a Hamilton cycle is at $p=\frac{\log{n} + \log\log{n}}{n}$.
A striking result in this area says that the `bottleneck' for the random graph to contain a Hamilton cycle is having minimum degree $2$. This concept is made precise in the hitting time results of Ajtai--Koml\'os--Szemer\'edi~\cite{Ajtai1985First} and Bollob\'as~\cite{Bollobas1984the}. For more on this, see~\Cref{subsec:hitting time}. For a nice introduction to this area along with shorter proofs of many of these results, see the survey~\cite{Krivelevich2015Hamiltonicity}. For other related problems and a historical overview of this area, we suggest the readers to consult the survey~\cite{Frieze2019hamilton}.

One way to view Dirac's theorem is that if one starts with a complete graph on $n$ vertices and deletes fewer than half of the edges incident to every vertex, then the residual graph contains a Hamilton cycle. Such a phenomenon was also shown to be true for random graphs above the Hamiltonicity threshold by Lee and Sudakov~\cite{Lee2012Dirac}. Sudakov and Vu~\cite{Sudakov2008local} initiated the systematic study of such phenomena in random graphs using the terminology of `resilience' of a given graph property. For some results on this topic, see, e.g.~\cite{Allen2020the,Alon2010increasing,Balogh2011local,Bottcher2013almost,Conlon2016combinatorial,Haxell1995turan,Haxell1996turan,Huang2012bandwidth,Lee2012Dirac,Schacht2016extremal}. 

There are two main types of resilience that are studied: (i) global resilience and (ii) local resilience. In this paper, we focus on the latter and say a graph $G$ is \defin{$\alpha$-resilient} with respect to a given property $\mathcal{P}$ if for every subgraph $H$ of $G$ satisfying $d_H(v)\le \alpha d_G(v)$ for each $v\in V(G)$, the graph $G-H$ has the property $\mathcal{P}$. In our paper, the property we are interested in is Hamiltonicity. Lee and Sudakov~\cite{Lee2012Dirac} showed that the random graph $G(n,p)$ with $p$ above the Hamiltonicity threshold is $(1/2-o(1))$-resiliently Hamiltonian w.h.p.\footnote{We say an event $A_n$ occurs \defin{with high probability} (w.h.p., in short) if $\lim_{n\rightarrow \infty} \Prob{A_n} = 1$.}, which is best possible up to the $o(1)$ term. Here and everywhere else in this paper, $o(1)$ (respectively $\omega(1)$) denotes a function of $n$ that goes to zero (respectively infinity) as $n$ tends to infinity. 
The above result can be stated in the following equivalent form. 
Following literature, a subgraph $H$ of a graph $G$ is called \defin{$\alpha$-residual} if $d_H(v) \ge \alpha d_{G[V(H)]}(v)$ for all vertices $v\in V(H)$. A subgraph $H$ of a graph $G$ is called \defin{spanning} if $V(H)=V(G)$.

\begin{theorem}[\cite{Lee2012Dirac, Montgomery2019resilient,Nenadov2019Resilience}] \label{thm:Dirac theorem in random graphs}
    Let $\epsilon >0$. If $p =  p(n) \ge \frac{\log{n} + \log\log{n} + \omega(1)}{n}$ and $G \sim G(n,p)$, then \whp every $\left(\frac{1}{2} + \epsilon\right)$-residual spanning subgraph of $G$ contains a Hamilton cycle. 
\end{theorem}

We remark that a hitting time version of the above result was obtained independently by Montgomery~\cite{Montgomery2019resilient} and Nenadov--Steger--Truji\'c~\cite{Nenadov2019Resilience} -- see Section~\ref{subsec:hitting time} for more detail on this and other hitting time results. As noted in the concluding remarks of~\cite{Krivelevich2014robust}, we have the following easy consequence of \Cref{thm:Dirac theorem in random graphs}. For $0\le p\le 1$ and a graph $G$, we denote by $G(p)$ the random subgraph of $G$ obtained from keeping every edge of $G$ independently with probability $p$.

\begin{corollary}\label{cor:Lee-Sudakov}
    For all $\epsilon >0$, there exists $C >0$ such that if $p =  p(n) \ge \frac{C\log{n}}{n}$ and $G$ is an $n$-vertex graph with minimum degree at least $\left(\frac{1}{2} + \epsilon\right)n$, then \whp the random graph $G(p)$ contains a Hamilton cycle. 
\end{corollary}

Krivelevich, Lee, and Sudakov~\cite{Krivelevich2014robust} showed that $\epsilon$ can be taken to be zero in the above result.

\subsection{Resilience meets discrepancy}
Gishboliner, Krivelevich, and Michaeli~\cite{Gishboliner2021colourbiased} considered the above problem of finding colour-biased Hamilton cycles in random graphs and showed the following. 

\begin{theorem}[\cite{Gishboliner2021colourbiased}] \label{thm:colour biased Ham cycle in random graph}
    Let $r\ge 2$ and $\eps >0$. 
    If $p =  p(n) \ge \frac{\log{n} + \log\log{n} + \omega(1)}{n}$ and $G \sim G(n,p)$, then \whp for every $r$-edge-colouring of $G$, there is a Hamilton cycle with one colour appearing at least $\left(\frac{2}{r+1} - \eps\right)n$ times. 
\end{theorem} 
This was also shown to be best possible up to the $\eps$ term -- for a matching construction, see the next section. We combine the preceding lines of research to obtain the following stronger result. 
\begin{restatable}{theorem}{main} \label{thm:main}
    Let $r\ge 2$, $\alpha\in [\frac{1}{2} + \frac{1}{2r}, 1]$, and $\epsilon >0$. If $p =  p(n) \ge \frac{\log{n} + \log\log{n} + \omega(1)}{n}$ and $G \sim G(n,p)$, then \whp for every $r$-edge-colouring of every $\alpha$-residual spanning subgraph of $G$, there is a Hamilton cycle with one colour appearing at least $(1-\eps) \min\set{(2\alpha - 1)n, \frac{2\alpha n}{r}, \frac{2n}{r+1}}$ times. 
\end{restatable} 

\Cref{thm:deterministic main} is the $p=1$ case of our above result. 
In the same way that \Cref{cor:Lee-Sudakov} is derived from \Cref{thm:Dirac theorem in random graphs}, we obtain the following as a consequence of \Cref{thm:main}. 

\begin{restatable}{recorollary}{corGp}
  \label{cor:G(p)}
  Let $r\ge 2$ and $\alpha\in [\frac{1}{2} + \frac{1}{2r}, 1]$. 
    For all $\epsilon >0$, there exists $C >0$ such that if $p =  p(n) \ge \frac{C\log{n}}{n}$ and $G$ is an $n$-vertex graph with minimum degree at least $\alpha n$, then \whp for every $r$-edge-colouring of the random graph $G(p)$, there is a Hamilton cycle with one colour appearing at least $(1-\eps) \min\set{(2\alpha - 1)n, \frac{2\alpha n}{r}, \frac{2n}{r+1}}$ times. 
\end{restatable}

In the above result as well as in \Cref{thm:main}, both the range of edge probability $p$ and the `colour bias' we obtain are asymptotically best possible. We will show this optimality in \Cref{sec:optimality}.

\subsection{Hitting time result}\label{subsec:hitting time}

We adapt the techniques we used to prove \Cref{thm:main} in order to prove a similar statement in the random graph process, leading to a more precise version of the theorem. Before stating that, we start with some basic definitions and classical hitting time results on Hamiltonicity.

The~\defin{$n$-vertex random graph process} is a sequence of graphs $\set{G_{n,M}}_{M \geq 0} = G_0, G_1, \dots, G_{\binom{n}{2}}$, each with vertex set $V$ of size $n$, such that $G_0$ contains no edges and, for each $i \in [\binom{n}{2}]$, $G_i$ is formed by adding an edge to $G_{i - 1}$ uniformly at random from the non-adjacent pairs of vertices of $G_{i - 1}$. We say that an event \defin{holds in almost every random graph process} if \whp the event holds in every graph of the random graph process $\set{G_{n,M}}_{M \geq 0}$.

Clearly, we can also form the above random graph process by giving an ordering of the pairs of vertices uniformly at random, and then adding the edges to the empty graph with vertex set $V$ in this order. All sets of $M$ edges are equally likely to appear as the first $M$ edges in this ordering; thus for each fixed $M$, $G_{n,M}$ is a random graph on $n$ vertices with $M$ edges chosen uniformly at random.

As briefly mentioned earlier, Ajtai--Koml\'os--Szemer\'edi~\cite{Ajtai1985First} and Bollob\'as~\cite{Bollobas1984the} showed that the following happens in almost every random graph process: the edge that increases the minimum degree from $1$ in $G_{n,M-1}$ to $2$ in $G_{n,M}$, for some $M$, also creates a Hamilton cycle in $G_{n,M}$. It turns out that this obstacle not only makes the graph Hamiltonian, but also makes it resiliently Hamiltonian as shown by Montgomery~\cite{Montgomery2019resilient} and Nenadov--Steger--Truji\'c~\cite{Nenadov2019Resilience}. They proved the following result, which is a hitting time version of~\Cref{thm:Dirac theorem in random graphs}. The minimum degree of a graph $G$ is denoted by $\delta(G)$.

\begin{theorem}[\cite{Montgomery2019resilient,Nenadov2019Resilience}] 
\label{thm:hittimg time montgomery}
Let $\eps >0$. Then, in almost every random graph process $\set{G_{n,M}}_{M \geq 0}$, if $\delta(G_{n,M}) \geq 2$, then every $(\frac{1}{2}+\eps)$-residual spanning subgraph of $G_{n,M}$ contains a Hamilton cycle.
\end{theorem}

The hitting time version of \Cref{thm:main} is as follows.
\begin{restatable}{theorem}{hittingtime}
  \label{thm:hitting-time-main}
  Let $r \geq 2$, $\alpha \in [\frac{1}{2} + \frac{1}{2r}, 1]$, and $\eps > 0$. Then, in almost every $n$-vertex random graph process $\set{G_{n,M}}_{M \geq 0}$, the following is true for each $0 \leq M \leq \binom{n}{2}$. If $\delta(G_{n,M}) \geq 2$, then for every $r$-edge-colouring of every $\alpha$-residual spanning subgraph of $G_{n,M}$, there is a Hamilton cycle with one colour appearing at least $(1 - \eps) \min\set{(2\alpha - 1)n, \frac{2\alpha n}{r}, \frac{2n}{r+1}}$ times.
\end{restatable}

\subsection{Ramsey-type results on matchings}\label{sec:ramsey matchings}

A key ingredient of our proofs of the main results is a Ramsey-type statement about finding large monochromatic matchings in edge-coloured graphs with a given minimum degree. Before stating it, we briefly mention some known results along this line. 

A classical result of Cockayne and Lorimer \cite{Cockayne1975the} implies that every $r$-edge-colouring of the complete graph $K_n$ contains a monochromatic matching of size at least $\frac{n-1}{r+1}$. Gishboliner, Krivelevich, and Michaeli~\cite{Gishboliner2021colourbiased} extended this to almost complete host graphs and used that to prove their \Cref{thm:colour biased Ham cycle in random graph}. We strengthen these results and obtain the following pleasing statement.

\begin{theorem}
  \label{thm:intro large monochromatic matching}
  Let $G$ be an $r$-edge-coloured $n$-vertex graph with minimum degree $d$. Then $G$ contains a monochromatic matching with at least $\min\left\{\frac{d}{r}, \frac{n-1}{r+1}\right\}$ edges. 
\end{theorem}
This is best possible as shown by matching constructions in \Cref{sec:optimality}. A particularly nice corollary of \Cref{thm:intro large monochromatic matching} is that every $r$-edge-coloured $n$-vertex graph with minimum degree at least $\frac{r(n-1)}{r+1}$ contains a monochromatic matching of size at least $\frac{n-1}{r+1}$. This significantly improves the previously mentioned result of Cockayne and Lorimer \cite{Cockayne1975the}. On a different note, a related `Ramsey-Tur\'an' variant of this result is established in a recent preprint \cite{Keevash2025a}.

\begin{organization*}
The rest of this paper is organized as follows. The next section contains constructions showing the tightness of our main results and outlines some of the main proof ideas. This outline also briefly (but in more detail than here) describes the contents of the subsequent sections. In \Cref{sec:prop-binom-rand}, we collect a few useful probabilistic tools and results about random graphs. In \Cref{sec:mono matchings}, we prove \Cref{thm:intro large monochromatic matching}. This is then used in \Cref{sec:lonG_{n,M}ono_paths} to prove a more complicated Ramsey-type statement about finding large monochromatic linear forests in edge-coloured dense graphs, which is then used to prove \Cref{thm:deterministic main}. We then prove more auxiliary results in \Cref{sec:clean-up,sec:resilience of Hamiltonicity}, which are finally combined in \Cref{sec:wrap up} to prove \Cref{thm:main,cor:G(p),thm:hitting-time-main}.
\end{organization*}

\begin{notation*}
  We write $[n]$ to denote the set $\{1,\dots,n\}$. For a set $X$ and an integer $k$, denote by $\binom{X}{k}$ the set of all subsets of $X$ of size $k$. 
  For three numbers $\alpha,\beta,\gamma$, we write $\alpha=\beta\pm \gamma$ to mean that $\beta-\gamma\le \alpha \le \beta+\gamma$. 
  Throughout this paper, we omit the rounding signs whenever they are not crucial to our arguments.

  We use standard graph theoretic notation. Consider a graph $G$. We denote its
  vertex set by $V(G)$ and its edge set by $E(G)$. The number of edges in $G$ is denoted by $e(G)$, which is also referred to as \defin{size} of $G$. When $H$ is a subgraph of $G$, we write $H\subseteq G$.
  For $A\subseteq V(G)$, we denote by $G[A]$ the subgraph of $G$ induced by $A$, and denote by $e_G(A)$ the number of edges in the graph $G[A]$. For $A\subseteq V(G)$, we denote by $G-A$ the graph $G[V(G)\setminus A]$.
  For $A,B\subseteq V(G)$, we denote by $e_G(A,B)$ the number of pairs $(a,b)\in A\times B$ such that $ab\in E(G)$. Note that $e_G(A,A) = 2e_G(A)$ for $A\subseteq V(G)$.
  For a vertex $v\in V(G)$, we denote by $N_G(v)$ the neighbourhood of $v$ in $G$, and denote by $d_G(v)$ the degree of $v$. 
  For a vertex $v\in V(G)$ and a set of vertices $A \subseteq V(G)$, let $d_G(v, A) = |N_G(v) \cap A|$. 
  For $A\subseteq V(G)$, we write $N_G(A) = (\bigcup_{v\in A} N_G(v))\setminus A$. We often omit the subscripts when the graph $G$ is clear from the context.
  For a set $A$, we say $G$ \defin{covers} $A$ if $A\subseteq V(G)$, and we say $G$ \defin{spans} $A$ if $A=V(G)$.
  For two graphs $G$ and $H$ on the same vertex set $V$, we write $G + H$ or $G - H$ to denote the graph with vertex set $V$ and with edge set $E(G)\cup E(H)$ or $E(G)\setminus E(H)$, respectively. If $H$ is a graph with a single edge $e$, we write $G + e$ or $G - e$ to mean $G + H$ or $G - H$, respectively.
\end{notation*}

\bigskip \section{Optimality and proof ideas of the main results}\label{sec:optimality}
\subsection{Optimality}

In this section, we show the tightness of our main results (i.e., \Cref{thm:deterministic main,cor:main for perfect matching,thm:main,thm:intro large monochromatic matching}). There are three competing constructions for these results, and each of them is optimal for certain regimes. Note that~\Cref{construction:small_alpha} is a generalization of the construction used to show tightness of \Cref{thm:colour biased Ham cycle with min degree condition} in \cite{Freschi2021a,Gishboliner2022discrepancies}, and~\Cref{construction:large_alpha} is the construction used to show tightness of \Cref{thm:colour biased Ham cycle in random graph} in~\cite{Gishboliner2021colourbiased}.

\begin{con}\label{construction:small_alpha} Suppose that $\alpha n$ is an integer less than $n$. Let $F$ be a graph on $n$ vertices partitioned into $r$ sets $V_1, \dots, V_r$, where $|V_r| = \alpha n$ and  $|V_i| \in \left\{\left\lfloor\frac{(1 - \alpha) n}{r - 1}\right\rfloor, \left\lceil\frac{(1 - \alpha) n}{r - 1}\right\rceil\right\}$ for each $i \in [r-1]$ so that $\sum_{i=1}^{r-1} |V_i| = (1-\alpha)n$. Let the edges of $F$ be all pairs of vertices that intersect $V_r$. Note that the minimum degree of $F$ is $|V_r| = \alpha n$.
We colour the edges of $F$ with $r$-colours as follows. For every $i \in [r - 1]$, every edge that intersects set $V_i$ and $V_r$ has colour $i$. Every other edge, that is, every edge contained in $V_r$, has colour $r$. See \Cref{fig:construction_small_alpha} for an example.
\end{con}

\begin{con}\label{construction:med_alpha}
Suppose that $\alpha n$ is an integer less than $n$. Let $F$ be a graph on $n$ vertices partitioned into sets $V_1, \dots, V_{r + 1}$, where $|V_{r+1}| = (1 - \alpha)n$, and $|V_i| \in \left\{\left\lfloor\frac{\alpha n}{r}\right\rfloor, \left\lceil\frac{\alpha n}{r}\right\rceil\right\}$ for every $i \in [r]$ so that $\sum_{i=1}^{r} |V_i| = \alpha n$. Let the edges of $F$ be all pairs that intersect the set $\cup_{i \in [r]} V_i$. Note that the minimum degree of $F$ is $\sum_{i=1}^{r} |V_i| = \alpha n$.
We colour the edges of $F$ with $r$-colours as follows.  For each $i, j \in [r]$ with $i \leq j$, each edge incident to both sets $V_i$ and $V_{r + 1}$ has colour $i$, and each edge incident to both sets $V_i$ and $V_j$ (possibly the same set) has colour $i$. See \Cref{fig:construction_medium_alpha} for an example.
\end{con}

\begin{con}\label{construction:large_alpha}
Let $F$ be a complete graph on $n$ vertices partitioned into sets $V_1, \dots, V_{r + 1}$, each  of size $\left\lfloor\frac{n}{r + 1}\right\rfloor$ or $\left\lceil\frac{n}{r + 1}\right\rceil$, with $|V_{r + 1}| = \left\lceil\frac{n}{r + 1}\right\rceil$. We colour the edges of $F$ with $r$-colours as follows. For every $i \in [r]$, the colour of the edges inside part $V_i$ is $i$, and inside part $V_{r + 1}$ is $r$. For each $i, j \in [r + 1]$ with $i < j$, the colour of an edge incident to both parts $V_i$ and $V_j$ is $i$. See \Cref{fig:construction_large_alpha} for an example.
\end{con}

  \begin{figure}[ht]
    \centering
    \begin{subfigure}[t]{0.48\textwidth}
    \centering
    \includegraphics[scale=.5]{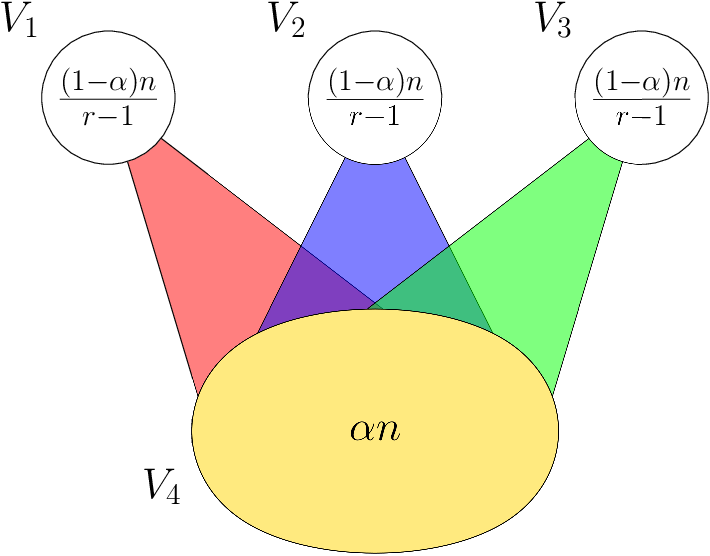}
    \caption{\Cref{construction:small_alpha}}
    \label{fig:construction_small_alpha}
\end{subfigure}
~
  \begin{subfigure}[t]{0.48\textwidth}
  \centering
    \includegraphics[scale=.5]{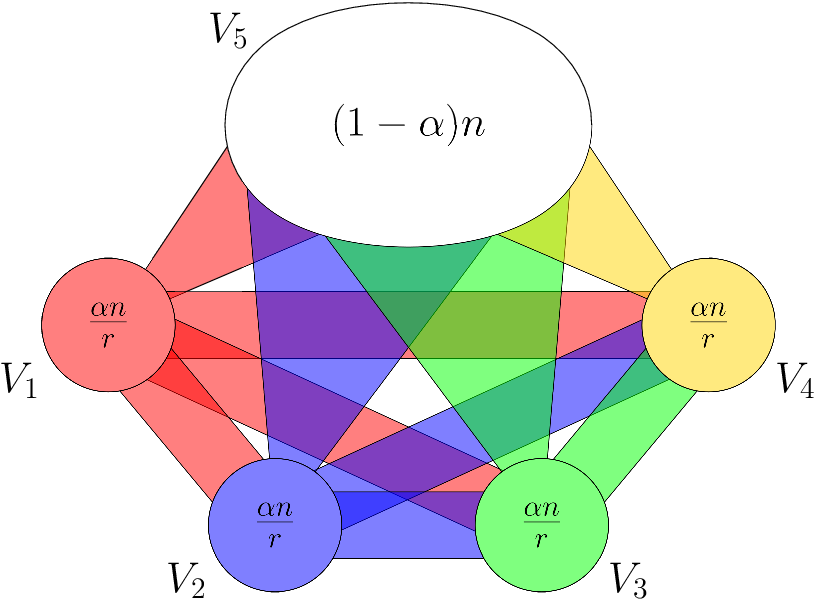}
    \caption{\Cref{construction:med_alpha}}
    \label{fig:construction_medium_alpha}
    \end{subfigure}

    \bigskip
    
    \begin{subfigure}[t]{\textwidth}
    \centering
    \includegraphics[scale=.5]{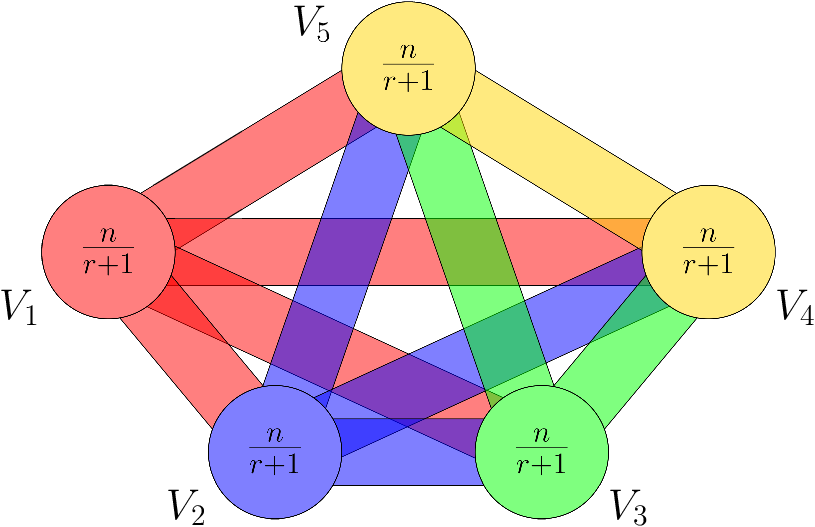}
    \caption{\Cref{construction:large_alpha}}
    \label{fig:construction_large_alpha}
\end{subfigure}
    \caption{Diagrams showing \Cref{construction:small_alpha}, \ref{construction:med_alpha} and~\ref{construction:large_alpha} with $r=4$ colours.}
\end{figure}

First we will use these constructions to show that \Cref{thm:intro large monochromatic matching} is best possible. We will show that 
\begin{proposition}
For all $0 \le d \le n-1$ there exists an $n$-vertex graph with minimum degree at least $d$ such that the size of a maximum monochromatic matching is at most $\lceil k \rceil$, where $k = \min\left\{ \frac{d}{r}, \frac{n-1}{r+1}\right\}$. 
\end{proposition}
\begin{proof}
First assume that $k = \frac{d}{r}$  and consider \Cref{construction:med_alpha} with $\alpha n = d$. The size of a monochromatic matching in $F$ is at most $\max_{i \in [r]}|V_i| = \left\lceil\frac{\alpha n}{r}\right\rceil = \left\lceil\frac{d}{r}\right\rceil$. Then suppose that $ k = \frac{n-1}{r+1}$ and consider \Cref{construction:large_alpha}. The size of a monochromatic matching is at most $m:=\max\left\{\left\lfloor\frac{|V_r| + |V_{r+1}|}{2}\right\rfloor, \max_{i \in [r-1]}|V_i|\right\}$. Clearly, $m \le \left\lceil\frac{n}{r+1}\right\rceil$. In fact, it is easy to check that either $\left\lceil\frac{n}{r+1}\right\rceil = \left\lceil\frac{n-1}{r+1}\right\rceil$, or that $n\equiv 1$ modulo ${r+1}$, in which case $m = \left\lfloor\frac{n}{r+1}\right\rfloor = \left\lceil\frac{n-1}{r+1}\right\rceil$ (because $|V_{r + 1}| = \left\lceil\frac{n}{r + 1}\right\rceil$ by construction). Either way, the size of a monochromatic matching is at most $\left\lceil\frac{n-1}{r+1}\right\rceil$.
\end{proof}

The following shows the asymptotic tightness of \Cref{thm:deterministic main,cor:main for perfect matching}.
\begin{proposition}
\label{prop:construction_deterministic_tight}
  Let $n,r \geq 2$ be integers, and $\alpha \in [\frac{1}{2} + \frac{1}{2r}, 1)$ be such that $\alpha n$ is an integer. Let $k = \min\left\{(2\alpha - 1)n, \frac{2\alpha n}{r}, \frac{2n}{r + 1}\right\}$. Then, there
  exists an $n$-vertex graph with minimum degree at least $\alpha n$, together with an $r$-edge-colouring, such that every Hamilton cycle contains at most $2\left\lceil \frac{k}{2} \right\rceil$ edges of the same colour, and, if $n$ is even, every perfect matching contains at most $\left\lceil \frac{k}{2} \right\rceil$ edges of the
  same colour.
\end{proposition}
\begin{proof}
We will split into three cases depending on whether $k = (2\alpha - 1)n$ or $k = \frac{2\alpha n}{r}$ or $k = \frac{2 n}{r+1}$.

\medskip
Assume first that $k = (2\alpha - 1)n$ and consider~\Cref{construction:small_alpha}. Let $\mathcal{C}$ be a Hamilton cycle in $F$. Clearly, for every vertex $v \in V_{i}$ with $i \in [r - 1]$, the cycle $\mathcal{C}$ contains exactly two edges incident to $v$, both of which have colour $i$. Furthermore, every edge of colour $i$ is incident to one such vertex. Hence, $\mathcal{C}$ contains exactly $2|V_i| \le 2\left\lceil\frac{(1 - \alpha)n}{r - 1}\right\rceil$ edges of colour $i$. Also, given that $\mathcal{C}$ contains $n$ edges, it contains $n - \sum_{i=1}^{r-1}2|V_i| = n - 2(1 - \alpha)n = (2\alpha - 1)n = k$ edges of colour $r$. Notice that $\frac{(1 - \alpha)}{r - 1} \le \frac{2\alpha - 1}{2}$ given that $\alpha \geq \frac{1}{2} + \frac{1}{2r}$ and so $2\left\lceil\frac{(1 - \alpha)n}{r - 1}\right\rceil \le 2\left\lceil\frac{(2\alpha - 1)n}{2}\right\rceil = 2\left\lceil \frac{k}{2} \right\rceil$. Hence, $\mathcal{C}$ has at most $2\lceil \frac{k}{2} \rceil$ edges of the same colour.  Similarly, if $\mathcal{M}$ is a perfect matching in $F$ then $\mathcal{M}$ contains exactly $|V_i|$ edges of colour $i$, and $\frac{n}{2} - \sum_{i=1}^{r-1}|V_i| = \frac{k}{2}$ edges of colour $r$. Thus $\mathcal{M}$ contains at most $\left\lceil \frac{k}{2} \right\rceil$ edges of the same colour.

\medskip
Suppose now that $k = \frac{2\alpha n}{r}$ and consider~\Cref{construction:med_alpha}. Let $\mathcal{C}$ and $\mathcal{M}$ be a Hamilton cycle and a perfect matching in $F$, respectively. Note that for every $i \in [r]$, all the edges of colour $i$ are incident to part $V_i$. For every vertex $v \in V_i$, the cycle $\mathcal{C}$ contains exactly two edges incident to $v$, and the matching $\mathcal{M}$ contains exactly one edge incident to $v$. Then $\mathcal{C}$ contains at most $2|V_i| \le 2\left\lceil\frac{\alpha n}{r}\right\rceil = 2\left\lceil \frac{k}{2}\right\rceil$ edges of colour $i$, and $\mathcal{M}$ contains at most $|V_i| \le \left\lceil \frac{k}{2}\right\rceil$ edges of colour $i$.

\medskip
Finally, suppose that $k = \frac{2n}{r + 1}$ and consider~\Cref{construction:large_alpha}. Let $\mathcal{C}$ and $\mathcal{M}$ be a Hamilton cycle and a perfect matching in $F$, respectively. For every $i \in [r - 1]$, every edge of colour $i$ is incident to part~$V_i$. For every vertex $v \in V_i$, the cycle $\mathcal{C}$ contains exactly two edges incident to $v$, and the matching $\mathcal{M}$ contains exactly one edge incident to $v$. Then for every colour $i\in [r-1]$, the cycle $\mathcal{C}$ contains at most $2|V_i| \le 2\left\lceil\frac{n}{r + 1}\right\rceil = 2\left\lceil \frac{k}{2} \right\rceil$ edges of colour $i$, and the matching $\mathcal{M}$ contains at most $|V_i| \le \left\lceil \frac{k}{2} \right\rceil$ edges of colour $i$. Also, all the edges of colour $r$ are contained in $V_r \cup V_{r + 1}$. Therefore, $\mathcal{C}$ contains at most $|V_r| + |V_{r + 1}| \le 2\left\lceil \frac{k}{2} \right\rceil$ edges of colour $r$, and $\mathcal{M}$ contains at most $\frac{|V_{r}| + |V_{r + 1}|}{2} \le  \left\lceil \frac{k}{2} \right\rceil$ edges of colour $r$.
\end{proof}

As a consequence of \Cref{prop:construction_deterministic_tight}, we obtain the following result, which shows the asymptotic tightness of \Cref{thm:main}.

\begin{proposition}
  Let $n,r \geq 2$ be integers, $\alpha \in [\frac{1}{2} + \frac{1}{2r}, 1)$,
  and $\eps > 0$.  Let also $k_{\alpha} = k_{\alpha}(n) = \min\set{(2\alpha - 1)n, \frac{2\alpha n}{r}, \frac{2n}{r + 1}}$,
  and $p = p(n) \geq \frac{\log{n} + \log\log{n} + \omega(1)}{n}$, and suppose $G \sim G(n, p)$. Then \whp there exists an $r$-edge-colouring of an $\alpha$-residual spanning subgraph of $G$ such that every Hamilton cycle has at most $(1 + \eps)k_{\alpha}$ edges of the same colour. 
\end{proposition}
\begin{proof}
Fix $r \geq 2$, $\alpha \in [\frac{1}{2} + \frac{1}{2r}, 1)$, $\eps > 0$, and let $n$ be a sufficiently large integer relative to $r,\alpha,\eps$. Let $\delta = \min\set{\frac{1-\alpha}{4\alpha}, \frac{(2\alpha - 1)\eps}{8\alpha}}$ and let $\alpha'$ be such that $(1 + \delta)\alpha \le \alpha' \le (1 + 2\delta)\alpha$ and $\alpha' n$ is an integer. Note that $\alpha'\in [\frac{1}{2} + \frac{1}{2r}, 1)$. Let $F$ be the $r$-edge-coloured graph on the same vertex set as $G$ obtained from applying \Cref{prop:construction_deterministic_tight} with $r$ and $\alpha'$ (in place of $\alpha$). Let $G'$ be the $r$-edge-coloured spanning subgraph of $G$ with edge set $E(G) \cap E(F)$, where the colouring of $G'$ is the restriction of the colouring of $F$ to the edges of $G'$. We claim that \whp $G'$ is an $\alpha$-residual spanning subgraph of $G$. Indeed, by Chernoff's bound (see, \Cref{lem:chernoff}), \whp the following hold for every vertex $v$ of $F$. 
\begin{align*}
    d_{G'}(v) &\geq \paren{1 - \frac{\delta}{3}} d_F(v) p, & &\text{and} & d_G(v) &\leq \paren{1 + \frac{\delta}{3}} np.
\end{align*}
We also have $d_F(v) \geq \alpha' n \geq (1 + \delta)\alpha n$. Combining these facts with $0\le \delta\le 1$, we have
\begin{equation*}
    d_{G'}(v) \geq \paren{1 - \frac{\delta}{3}} d_F(v) p \geq \paren{1 - \frac{\delta}{3}} (1 + \delta)\alpha
    n p \geq \paren{\frac{\paren{1 - \frac{\delta}{3}}\paren{1 + \delta}}{1 + \frac{\delta}{3}}} \alpha d_G(v) \geq \alpha d_G(v).
\end{equation*}
Denote $k_{\alpha'} = k_{\alpha'}(n) = \min\set{(2\alpha' - 1)n, \frac{2\alpha' n}{r}, \frac{2n}{r + 1}}$. By \Cref{prop:construction_deterministic_tight}, we know that every Hamilton cycle in $F$ has at most $2\left\lceil\frac{k_{\alpha'}}{2}\right\rceil$ edges of the same colour. We now show that $k_{\alpha'} \leq (1 + \frac{\eps}{2})k_{\alpha}$. Analyzing each term of the set being minimized in the definition of $k_{\alpha'}$, first we have $2\alpha' - 1 \leq 2(1 + 2\delta)\alpha - 1 = 2\alpha - 1 + 4\alpha\delta \le (1 + \frac{\eps}{2})(2\alpha - 1)$. Second, $\frac{2\alpha'}{r} \leq (1 + 2\delta) \frac{2\alpha}{r} \leq (1 + \frac{\eps}{2}) \frac{2\alpha}{r}$. Given that the remaining term does not depend on $\alpha'$, the conclusion holds.

Finally, given that $G'$ is a spanning subgraph of $F$, every Hamilton cycle in $G'$ is also a Hamilton cycle in $F$. Therefore, every Hamilton cycle in $G'$ has at most $2\left\lceil\frac{k_{\alpha'}}{2}\right\rceil \leq k_{\alpha'} + 2 \leq (1 + \frac{\eps}{2})k_{\alpha} + 2 \leq (1 + \eps)k_{\alpha}$ edges of the same colour, where the last inequality uses the fact that $n$ was chosen sufficiently large. This finishes the proof.
\end{proof}

\subsection{Outline of proof}\label{sec:proof ideas}
We now give a brief and simplified outline of the proofs of \Cref{thm:deterministic main,thm:main}. (Recall that \Cref{thm:deterministic main} is the $p=1$ case of \Cref{thm:main}.) For that, let $G \sim G(n,p)$ with $p$ above the Hamiltonicity threshold. Let $H$ be an $\alpha$-residual subgraph of $G$, and fix an arbitrary edge-colouring of $H$ with $r$ colours. Similar to \cite{Gishboliner2021colourbiased}, our core strategy is to first obtain a large monochromatic \defin{linear forest} (i.e., a forest where each connected component is a path) in $H$, and then extend it to a Hamilton cycle using techniques such as P\'osa's rotation-extension. However, due to the generality of our settings, a lot of complication arises in executing this idea. We divide our arguments into four main steps, and we refer the reader to \Cref{fig:outline of proof} for a schematic overview while following these steps.

\begin{figure}[ht]
    \centering
    \includegraphics[scale=1]{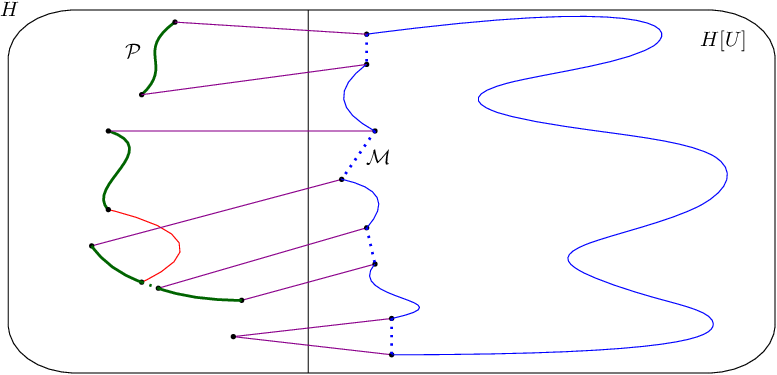}
    \caption{Outline of proof}
    \label{fig:outline of proof}
\end{figure}

\noindent {\bf Step~1.} First, we prove \Cref{thm:intro large monochromatic matching}, which is a Ramsey-type statement about finding large monochromatic matchings in edge-coloured graphs with a given minimum degree. This is done in \Cref{sec:mono matchings}. 

\medskip
\noindent {\bf Step~2.} We then find a large monochromatic linear forest $\mathcal{P}$ (i.e., the green paths in \Cref{fig:outline of proof}) in~$H$ with a constant number of paths and using at least $(1-o(1)) \min\left\{\frac{2\alpha n}{r}, \; \frac{2n}{r+1}\right\}$ edges in total (see, \Cref{cor:random graph large monochromatic linear forest}). To do so, we apply a `multicolour' variant of the sparse regularity lemma to $H$ to obtain an edge-coloured reduced graph. Then, using \Cref{thm:intro large monochromatic matching}, we find a large monochromatic matching in the reduced graph, each edge of which gives rise to a long monochromatic path in the original graph. This is inspired by the proof of~\Cref{thm:colour biased Ham cycle in random graph} by Gishboliner, Krivelevich, and Michaeli~\cite{Gishboliner2021colourbiased} and is covered in \Cref{sec:lonG_{n,M}ono_paths}. This step, together with a classical lemma of P\'osa (\Cref{lem:posa-force-hamilton}), already allows us to extend the large monochromatic forest $\mathcal{P}$ to a Hamilton cycle, thus proving \Cref{thm:deterministic main}. However, such arguments are not sufficient to prove \Cref{thm:main}, and we need a couple more steps to deal with it.

\medskip
\noindent {\bf Step~3.} To extend the large monochromatic linear forest $\mathcal{P}$ obtained in Step~2 to a Hamilton cycle, we next make minor alterations to `clean up' the graph $H$ (as shown by the addition of a red path and removal of a green edge in \Cref{fig:outline of proof}). We want to ensure that the graph induced on the unused vertices (i.e., not used in the linear forest) has sufficiently high degree. First, we cover the low-degree vertices of $H$ by adding a short path including each such vertex to the linear forest. We then also extend the paths to cover any unused vertex with a small degree to the set of unused vertices. In addition, we break up some of these paths (without removing too many of the monochromatic edges) to ensure that all endpoints of all paths have high degree to the set $U$ of as yet unused vertices. This last property is then used to extend the paths (using the purple edges in \Cref{fig:outline of proof}) to endpoints in a `well-behaved' subset of $U$. This well-behaviour and the high minimum degree of $H[U]$ will be useful in Step~4 to extend the `almost monochromatic' linear forest to a Hamilton cycle. Step~3 is executed in~\Cref{sec:clean-up}. 

\medskip
\noindent {\bf Step~4.} We now represent the paths of the `almost monochromatic' linear forest by a matching $\mathcal{M}$ where an edge is added between the two endpoints of a path in $U$ (see the blue dotted edges in \Cref{fig:outline of proof}). We next show that the graph $H[U]$ contains an expander that is `$\mathcal{M}$-respecting' in a natural sense. We then use P\'osa's rotation and extension technique and the idea of `booster' edges to find a Hamilton cycle on $U$ using the edges of $H[U]$ and the matching $\mathcal{M}$. In particular, the Hamilton cycle uses every edge of $\mathcal{M}$, so that replacing the edges of $\mathcal{M}$ by the corresponding paths gives a Hamilton cycle in $H$ with appropriate colour-bias. The ideas in this step are inspired by the proof of \Cref{thm:hittimg time montgomery} by Montgomery~\cite{Montgomery2019resilient}, with relatively minor adaptations to ensure that all steps are `$\mathcal{M}$-respecting' -- that is, if a vertex of $\mathcal{M}$ is used, then so is the corresponding edge. This is covered in \Cref{sec:resilience of Hamiltonicity}.

\medskip
See \Cref{fig:overview} for a diagram of these key steps of the proof. All these steps are combined in \Cref{sec:wrap up} to prove \Cref{thm:main,thm:hitting-time-main}. Before proceeding with our proofs, we devote the following section to collecting probabilistic tools that will be helpful throughout this paper.

\begin{figure}[h]
\centering
    \begin{tikzpicture}[
node distance = 5mm and 5mm,
  start chain = going below,
 block/.style = {draw, rounded corners, thick, align=center, minimum height=2em, inner sep=5pt,
                on chain},
 plate/.style={draw, shape=rectangle, rounded corners, dotted, align=left, text width=2cm, inner sep=4pt, label={[xshift=28mm,yshift=7mm]south west:#1}},
 FITout/.style = {draw, thick, rounded corners, dotted,inner sep=2mm, fit=#1}
  ]
    \node (main) [block] {\cref{thm:main,thm:hitting-time-main} \\
    (\Cref{sec:wrap up})};
    \node[draw=none] (dot1) [below= of main] {};
    \node (matching) [block] [below left=5mm and 4mm of dot1] {
    `Well-behaved' linear forest \\  represented by matching $\mathcal{M}$ in \\ set $U$ of as yet unused vertices
    };
    \node[draw=none] (dot2)  [below=12mm of matching] {};
    \node (cleanup) [block] [right=3mm of dot2]  {Clean-up \\
    (\cref{sec:clean-up})};
    \node (linearforest) [block] [left=3mm of dot2] {
    Large monochromatic \\ linear forest in $H$ \\(\cref{sec:lonG_{n,M}ono_paths})
    };
    \node (monomatching) [block] [below=of linearforest] {
   \Cref{thm:intro large monochromatic matching} \\ Large monochromatic\\ matching  in dense graphs 
    \\(\cref{sec:mono matchings})
    };
    \node (deterministic) [block] [left= 5 mm of matching] {
   \Cref{thm:deterministic main}
    \\(\cref{sec:lonG_{n,M}ono_paths})
    };
    \node (expander) [block] [below right =7mm and 4mm of dot1]  {$\exists$ expander \\ (\cref{sec:expander})};
    \node[draw=none] (dot)  [right=2.5mm of expander] {};
    \node (glue) [block] [right=2.5mm of dot] { expander $\Rightarrow$ \\ Ham.~cycle \\ (\cref{sec:join_paths_to_cycle})};
    
\node[FITout={(expander) (glue)}, 
      label={below:{In $H[U]$ and {`$\mathcal{M}$-respecting'}}}] (x1) {};
    
  \draw[->, thick,  to path={-- (\tikztotarget)}]
(monomatching) edge (linearforest);
  \draw[->, thick,  to path={-- (\tikztotarget)}]
(linearforest) edge (deterministic);
  \draw[->, thick]
 (dot2.center) -- (matching);
  \draw[-, thick]
   (matching) |- (dot1.center);
    \draw[-, thick]
   (dot.center) |- (dot1.center) ;
    \draw[->, thick]
 (dot1.center) -- (main);
  \draw[-, thick, to path={-- (\tikztotarget)}]
 (expander) edge (glue)
 (linearforest) edge (cleanup); 
 \draw[->, thick, dotted, to path={-- (\tikztotarget)}]
 (matching) edge (x1);
    \end{tikzpicture}
    \caption{Overview of the proof structure}
    \label{fig:overview}
\end{figure}
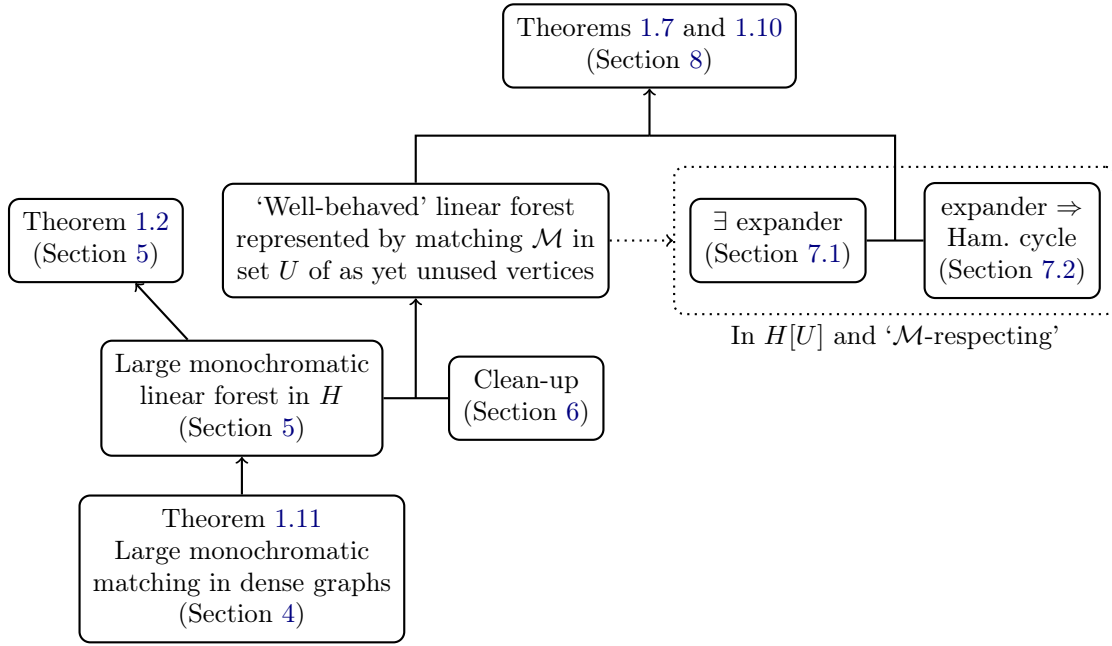

\bigskip \section{Probabilistic tools and properties of random graphs}
\label{sec:prop-binom-rand}
Our most fundamental tool is Chernoff's bound (see, e.g., \cite[Corollary~2.3]{janson2011random} or \cite[Lemma~2.2]{Montgomery2019resilient}).
\begin{lemma}[Chernoff's bound]
  \label{lem:chernoff}
  If $X$ is a binomial random variable, and $0 < \eps \leq \frac{3}{2}$, then
  \begin{equation*}
    \Prob{|X - \Expect X| \geq \eps \Expect X} \leq 2 \exp{\paren{-\frac{\eps^2 \Expect X}{3}}}. \qed
  \end{equation*}
\end{lemma}

The following variant of Chernoff's bound can be found in \cite[Proposition~2.3]{Montgomery2019resilient}.
\begin{proposition}
  \label{prop:weird-chernoff}
  Let $n \in \naturals$ and let $0 \leq p \leq 1$.
  Suppose $X_1, \dots, X_n$ are independent random variables, each
  equal to $1$ with probability $p$, and $0$
  otherwise. Suppose $\delta_i \in \set{1, 2}$, for each $i \in [n]$,
  and denote $X = \sum_{i \in [n]} \delta_i X_i$. Then, for each $0 <
  \eps < 1$, we have
  \begin{equation*}
    \Prob{|X - \Expect{X}| \geq \eps \Expect{X}} \leq 4
    \exp\paren{-\frac{\eps^2 \Expect{X}}{9}}.
  \end{equation*}
\end{proposition}

We can bound the degrees in $G(n,p)$ using the following pair of standard lemmas. 
\begin{lemma}[\cite{frieze2015introduction}]\label{lem:lower_bound_degrees}
    Let $G \sim G(n,p)$ with $p \ge \frac{\log{n} + \log\log{n} + \omega(1)}{n}$. Then \whp $d(v) \ge 2$ for all $v \in V(G)$. 
\end{lemma}
For the convenience of proving \Cref{thm:hitting-time-main}, we often specify some error terms on the probability in the following lemmas. How these error terms will help us is explained just before stating \Cref{lemma:model-switching}.
\begin{lemma}\label{lem:upper_bound_degrees}
    Let $G \sim G(n,p)$ with $p \ge \frac{\log{n}}{2n}$. Then with probability $1-o(n^{-3})$, we have $d(v) \le 10np$  for all $v \in V(G)$.
\end{lemma}
\begin{proof}
For every vertex $v\in V(G)$, we have
\[
\Prob{d(v)\ge 10np}\le \binom{n}{10np}p^{10np}\le \left(\frac{enp}{10np}\right)^{10np}\le \left(\frac{1}{e}\right)^{5\log n}=n^{-5}.
\]
This, together with a union bound over all vertices in $V(G)$, implies the lemma.
\end{proof}

The following pair of lemmas allows us to bound the number of edges in large sets and small sets, respectively, in random graphs.
\begin{lemma}[Lemma 2.4 of \cite{Montgomery2019resilient}]\label{lem:large-sets-pseudorandom-property-Gnp}
Let $0 < \lambda < 1$ and $G \sim G(n,p)$. Then with probability $1-o(n^{-3})$, for all $A,B\subseteq V(G)$ satisfying $p|A||B|\ge 100n/\lambda^2$, we have $e_G(A,B) = (1\pm \lambda)p|A||B|$.
\end{lemma}

\begin{lemma}\label{lem:small-sets-cannot-contain-too-many-edges}
For every $a,b >0$ with $a+b > 1$, there exists $C >0$ such that the following holds. Let $G \sim G(n,p)$ with $p\ge \frac{C}{n}$. Then with probability $1-o(n^{-3})$, for every set $S$ of size at most $n(np)^{-a}$, we have $e_G(S)\le |S|(np)^{b}$.
\end{lemma}
\begin{proof}
Choose $C$ sufficiently large compared to $a$ and $b$. Clearly, every set $S$ of size less than $(np)^{b}$ satisfies that $e_G(S)\le \binom{|S|}{2}\le |S|(np)^{b}$. By the union bound, the probability that a set $S$ with size at least  $(np)^{b}$ and at most $n(np)^{-a}$ violates the desired bound on $e(S)$ is at most 

\begin{align*}
\sum_{s=(np)^{b}}^{n(np)^{-a}} \binom{n}{s} \binom{\binom{s}{2}}{s(np)^{b}} &p^{s(np)^{b}} 
\le \sum_{s=(np)^{b}}^{n(np)^{-a}} \left(\frac{en}{s}\right)^s \left(\frac{es^2 p}{s(np)^{b}}\right)^{s(np)^{b}} \\
&= \sum_{s=(np)^{b}}^{n(np)^{-a}} \left(\frac{en}{s}\left(\frac{esp}{(np)^{b}}\right)^{(np)^{b}}\right)^s
\\
&\le \sum_{s=(np)^b}^{(\log{n})^2} \left(\frac{en}{(\log{n})^2}\left(\frac{ep(\log{n})^2}{(np)^{b}}\right)^{(np)^{b}}\right)^s + \sum_{s=(\log{n})^2}^{n(np)^{-a}} \left(e(np)^a\left(e(np)^{1-a-b}\right)^{(np)^{b}}\right)^s 
\\
&\le \sum_{s=(np)^{b}}^{(\log{n})^2} \left(\frac{1}{n}\right)^s + \sum_{s=(\log{n})^2}^{n(np)^{-a}} \left(\frac{1}{2}\right)^s
\quad = o(n^{-3}).
\end{align*}
In the above calculations, we frequently used the fact that $np\ge C$, where $C$ was chosen to be sufficiently large compared to $a$ and $b$. This finishes the proof.
\end{proof}

The next proposition bounds the likelihood of certain subgraphs appearing in $G(n,p)$.
\begin{proposition}[Proposition 2.5 of \cite{Montgomery2019resilient}]
  \label{prop:sparse-containment-Gnp}
  For each $0 < \delta < 1$, there exists $n_0$ such that, for each $n \geq n_0$ and $p \geq
  \frac{1}{n}$,
  \begin{equation*}
    \sum_{H \subseteq K_n, e(H) \leq \delta p n^2} \Prob{H \subseteq G(n, p)} \leq \exp{\paren{2\delta \log\paren{\frac{\e}{\delta}} \cdot pn^2}}.
  \end{equation*}
\end{proposition}

When proving a statement about the random graph process, it is often convenient to prove it in $G(n, p)$, with probability $1 - o(n^{-3})$, and then use the following standard model-switching auxiliary lemma (which can be found in, e.g., \cite{MR1864966}) to obtain the same conclusion in $G_{n,M}$ with probability $1 - o(n^{-2})$, which can then be transformed into a statement about the random graph process \whp by using a union bound.

\begin{lemma}[\cite{MR1864966}]
  \label{lemma:model-switching}
  Let $n \in \naturals$, $1 \leq M \leq \binom{n}{2}$ and $p = M /
  \binom{n}{2}$, and let $\cP$ be a graph property. Then
  \begin{equation*}
    \Prob{G_{n,M} \text{ has property } \cP} \leq 2n \cdot \Prob{G(n, p)
      \text{ has property } \cP}.
  \end{equation*}
\end{lemma}

\begin{lemma}[\cite{luczak1990equivalence}, see also Theorem 1.4 in \cite{frieze2015introduction}]\label{thm:equivalence_Gnp_Gnm}
    Let $\mathcal{P}$ be a graph property and let $p = p(n)$. Suppose that for every $M$ with $|M - \binom{n}{2}p| = n\sqrt{p(1-p)} \cdot \log n$, we have that \whp $G_{n,M}$ has property $\mathcal{P}$. Then \whp $G(n,p)$ has property $\mathcal{P}$.
\end{lemma}

It is well-known that in almost every random graph process $\set{G_{n,M}}$, the first graph with minimum degree $2$ has around $\frac{n(\log n + \log\log n)}{2}$ edges. This is captured in the following statement, which can be found in \cite[Theorems~2.2(ii) and~3.5]{MR1864966} and \cite[Lemma~6.2]{Montgomery2019resilient}.

\begin{lemma}[\cite{MR1864966}]
  \label{lem:few_edges_GnM_degree_less_2}
  If $M = \frac{n\paren{\log n + \log \log n - \omega(1)}}{2}$, then
  \whp $\delta(G_{n, M}) < 2$.
\end{lemma}

Bollob\'as~\cite{bollobas1984evolution} noted that when the graph $G_{n,M}$ first has minimum degree 2, there are no vertices with low degree that are close together. This was slightly generalized by Montgomery to the following result.
\begin{lemma}[Corollary 6.6 in \cite{Montgomery2019resilient}]\label{lem:small_deg_vertex_separation_in_G_{n,M}}
    In almost every random graph process $\{G_{n,M}\}$ the following holds: in each $G_{n,M}$ with $M \ge \frac{19n\log{n}}{40}$, if ${K \coloneqq \{v \in V(G_{n,M}) : d(v) \le \frac{M}{1000n}\}}$ then no two vertices in $K$ are within distance $5$ of each other.
\end{lemma}

We will also need to apply a version of \Cref{lem:small_deg_vertex_separation_in_G_{n,M}} to $G(n,p)$.
\begin{corollary}\label{cor:small_deg_vertices}
    Let $G \sim G(n,p)$ with $p \ge \frac{49\log{n}}{50n}$. Then \whp the following holds. If we let ${K \coloneqq \{v \in V(G) : d(v) \le \frac{np}{2500}\}}$, then no two vertices in $K$ are within distance $5$ of each other.
\end{corollary}
\begin{proof}
    This follows straightforwardly from \Cref{thm:equivalence_Gnp_Gnm,lem:small_deg_vertex_separation_in_G_{n,M}}.
\end{proof}

Our final prerequisite is the following local lemma.
\begin{lemma}[Lov\'asz local lemma, see~\cite{Alon2016the}]\label{lem:LLL} Let 
$\mathcal{B}$ be a finite set of events in a probability space, each occurring with probability at most $p$. Suppose that each event in $\mathcal{B}$ is independent from all but at most $d$ other events in $\mathcal{B}$. 
If $ep(d+1) < 1$ then there is a positive probability that no event in $\mathcal{B}$ occurs.
\end{lemma}

\bigskip \section{Large monochromatic matchings in dense graphs}\label{sec:mono matchings}

We prove \Cref{thm:intro large monochromatic matching} in this section. For this, we will prove the following equivalent result.
\begin{theorem}[Large monochromatic matching]
  \label{thm:large monochromatic matching}
  Let $G$ be an $r$-edge-coloured $n$-vertex graph with minimum degree $\delta(G) \leq \frac{rn}{r + 1}$. Then $G$ contains a monochromatic matching with at least $\frac{\delta(G)}{r}$ edges.
\end{theorem}
We first show how \Cref{thm:intro large monochromatic matching} follows from \Cref{thm:large monochromatic matching}.
\begin{proof}[Proof of \Cref{thm:intro large monochromatic matching} assuming \Cref{thm:large monochromatic matching}]
Let $G$ be an $r$-edge-coloured $n$-vertex graph with minimum degree $d$. If $d\le \frac{rn}{r+1}$, apply \Cref{thm:large monochromatic matching} on $G$ to obtain a monochromatic matching with at least $\frac{d}{r}\ge \min\set{\frac{d}{r},\frac{n-1}{r+1}}$ edges, as desired. Otherwise, $G$ has minimum degree $d > \frac{rn}{r+1}$. Then, first consider a spanning subgraph $G'$ of $G$ such that $\delta(G')=\lfloor\frac{rn}{r+1}\rfloor$. (Such a subgraph can be obtained by removing edges of $G$ one by one, and noting that removal of an edge can decrease the minimum degree by at most~$1$.) Now, we apply \Cref{thm:large monochromatic matching} on $G'$ to obtain a monochromatic matching with at least $\left\lceil\frac{\lfloor\frac{rn}{r+1}\rfloor}{r}\right\rceil \ge \frac{n-1}{r+1}$ edges. Thus, $G$ also contains a matching with at least $\frac{n-1}{r+1}\ge \min\set{\frac{d}{r},\frac{n-1}{r+1}}$ edges, as desired.
\end{proof}

We now proceed to prove \Cref{thm:large monochromatic matching}. For that, we will first introduce some shorthands and state some lemmas that will later facilitate our arguments. For a graph $G$, denote by $\mu(G)$ the matching number of $G$, denote by $c(G)$ the number of connected components of $G$, and denote by $\odd(G)$ the number of connected components of $G$ whose size is odd. 
\begin{theorem}[Tutte-Berge formula]
    Every graph $G$ satisfies 
    \begin{equation*}
    \mu(G) = \frac{|V(G)|}{2} - \frac{1}{2}\max_{U \in V(G)} \paren{\odd\paren{G - U} - |U|}.
    \end{equation*}
\end{theorem}
\begin{lemma}\label{lem:tutte}
  For every $n$-vertex graph $G$, there exists a set of vertices $U \subseteq V(G)$ of size $|U| \leq \mu(G)$ such
  that $c(G - U) \geq n - 2\mu(G) + |U|$.
\end{lemma}
\begin{proof}
By the Tutte-Berge formula, we have
\begin{equation*}
    \mu(G) = \frac{n}{2} - \frac{1}{2}\max_{U' \in V(G)} \paren{\odd\paren{G - U'} - |U'|} \geq \frac{n}{2} - \frac{1}{2}\max_{U' \in V(G)} \paren{c\paren{G - U'} - |U'|}.
\end{equation*}
Then $\max_{U' \in V(G)} \paren{c\paren{G - U'} - |U'|} \geq n - 2\mu(G)$. Let $U$ be the subset of $V(G)$ that maximizes the left-hand side of this expression. We then have that $c(G - U) \geq n - 2\mu(G) + |U|$. Also, clearly $n - |U| \geq c(G - U) \geq n - 2\mu(G) + |U|$, which implies that $|U| \leq \mu(G)$.
\end{proof}

\begin{lemma}\label{lem:up_bound_components_inverse_degree}
  Every graph $G$ satisfies
  \begin{equation*}
    c(G) \leq \sum_{v \in V(G)} \frac{1}{d_G(v) + 1}.
  \end{equation*}
\end{lemma}
\begin{proof}
For a vertex $v$ of $G$, denote by $C_v$ the size of the connected component of $G$ that contains $v$. Clearly $|C_v|\geq d_G(v)+1$. Therefore,
\begin{equation*}
    \sum_{v \in V(G)} \frac{1}{d_G(v) + 1} \geq \sum_{v \in V(G)} \frac{1}{|C_v|} = c(G).
\end{equation*}
\end{proof}

\begin{lemma}[Lemma 1 in Section 2 from~\cite{Gishboliner2021colourbiased}]
\label{lem:up_bound_edges_of_disconnected}
Every graph $G$ satisfies
\begin{equation*}
    e(G) \leq \binom{|V(G)| - c(G) + 1}{2}.
\end{equation*}
\end{lemma}

We are now ready to prove \Cref{thm:large monochromatic matching}.
\begin{proof}[Proof of \Cref{thm:large monochromatic matching}]
Let $G$ be an $r$-edge-coloured $n$-vertex graph with minimum degree $\delta(G) \leq \frac{rn}{r + 1}$. For each $i \in [r]$, let $G_i$ denote the graph with vertex set $V(G)$ and whose edges are all edges of $G$ with colour $i$. Thus, $E(G) = \bigcup_{i \in [r]} E(G_i)$. Let $\mu = \max_{i \in [r]}\mu(G_i)$ denote the size of the largest monochromatic matching of $G$.

Note that \Cref{thm:large monochromatic matching} is trivial if $\delta(G) = 0$ and so we may assume $\delta(G) \ge 1$. Suppose for contradiction that $\mu < \frac{\delta(G)}{r}$. Hence,
\begin{equation}\label{eq:up_bound_mu}
    \mu \leq \frac{\delta(G) - 1}{r}.
\end{equation}
In addition, if $\delta(G) \ge 1$ then clearly $\mu \ge 1$, which implies 
\begin{equation}\label{eq:bound_n}
      r +1 \le \delta(G) < n.
\end{equation}
By applying \Cref{lem:tutte} to the graphs $G_i$, with $i \in [r]$, there is a set $U_i \subseteq V(G)$ of size $|U_i| \leq \mu$ such that $c(G_i - U_i) \geq n - 2\mu + |U_i|$.
Let $U = V(G) \setminus \cup_{i \in [r]} U_i$. We have
\begin{equation}\label{eq:low_bound_U}
    |U| \geq n - \sum_{i \in [r]} |U_i|.
\end{equation}
We also have the following for every $i\in [r]$,
\begin{align}\label{eq:low_bound_CGiU}
    c(G_i[U]) \geq c(G_i - U_i) - c\paren{\paren{G_i - U_i} - U} 
    &\geq n - 2\mu + |U_i| - \paren{n - |U_i| - |U|} \nonumber \\
    &= |U| - 2\mu + 2|U_i|.
\end{align}
We next make a few claims.
\begin{claim}\label{claim:low_sum_ui}
    \begin{equation*}
      n - r\mu + 1 \le |U| \le 4\mu + 1.
    \end{equation*}
\end{claim}
\begin{claimproof}
Since $|U_i| \le \mu$ for every $i\in [r]$, we have $|U| \ge n -\sum_{i \in [r]} |U_i| \ge n - r\mu$. By \Cref{lem:up_bound_edges_of_disconnected}, for every $i\in [r]$, we have 
\begin{align*}
      2 e(G_i[U]) 
     \leq \paren{|U| - c(G_i[U]) + 1} \paren{|U| - c\paren{G_i[U]}} 
     &\stackrel{\eqref{eq:low_bound_CGiU}}{\leq} \paren{2\mu - 2|U_i| + 1} \paren{2\mu - 2|U_i|} \\
     &\leq 2\paren{2\mu + 1} \paren{\mu - |U_i|}. \\
\end{align*}
Using this, we also have that
\begin{align*}
      r\mu|U| \stackrel{\eqref{eq:up_bound_mu}}{<} \delta(G) |U| \leq \sum_{v \in U} d_G(v) &\leq |U| \paren{n - |U|} + \sum_{i \in [r]} 2e(G_i[U]) \\
      &\leq |U| \paren{n - |U|} + 2(2\mu + 1) \paren{r\mu - \sum_{i\in [r]} |U_i|} \\
      &\stackrel{\eqref{eq:low_bound_U}}{\leq} |U| \paren{n - |U|} + 2(2\mu + 1) \paren{r\mu - n + |U|}.
\end{align*}
Therefore, $|U| \paren{r\mu -n + |U|} < 2(2\mu + 1) \paren{r\mu - n + |U|}$. Since $r \mu - n + |U| \ge 0$, this implies that $|U| > n -  r \mu$ and $|U| < 2(2\mu + 1)$, which concludes the proof of the claim.
\end{claimproof}
  
\begin{claim}\label{claim:at_least_2_comp_in_GiU}
    For all $i \in [r]$, we have $c(G_i[U]) > 1$.
\end{claim}
\begin{claimproof}
We have
\begin{align*}
      c(G_i[U]) \stackrel{\eqref{eq:low_bound_CGiU}}{\geq} |U| - 2\mu + 2|U_i| \stackrel{\eqref{eq:low_bound_U}}{\geq} n - \sum_{j \in [r]} |U_j| - 2\mu + 2|U_i| 
      &\geq n - 2\mu - \sum_{j \in [r] \setminus \set{i}} |U_j| \\
      &\geq n - 2\mu - (r - 1)\mu \\ 
      &\stackrel{\eqref{eq:up_bound_mu}}{\geq} n - (r + 1)\paren{\frac{\delta(G) - 1}{r}} \\
      &\geq n - \frac{r + 1}{r}\paren{\frac{r}{r + 1}n - 1} \\
      &= \frac{r + 1}{r} > 1.
\end{align*}
\end{claimproof}
\begin{claim}
\label{claim:no_monochromatic_vertex_inside_U}
    For every vertex $v \in U$ and every $i \in [r]$, we have $d_{G_i[U]}(v) \le d_{G[U]} (v) - 2$.
\end{claim}
\begin{claimproof}
Let $t = \sum_{j \in [r]} \paren{\mu - |U_j|}$. Observe first that every vertex $v \in U$ satisfies
\begin{equation}
      \label{eq:low_deg_GU_v}
      d_{G[U]}(v) \geq \delta(G) - \sum_{j \in [r]}|U_j| \stackrel{\eqref{eq:up_bound_mu}}{\geq} 1 + \mu r - \sum_{j \in [r]} |U_j| = 1 + \sum_{j \in [r]} \paren{\mu - |U_j|} = 1 + t.
\end{equation}
For every vertex $v \in U$ and every $i \in [r]$, denote by $C_{i, u}$ the vertex-set of the connected component of $G_i[U]$ that contains $v$.

Suppose for a contradiction that there is some $v \in U$ and $i \in [r]$ such that $d_{G_i[U]}(v) \ge d_{G[U]}(v) - 1$. By \Cref{claim:at_least_2_comp_in_GiU}, the graph $G_i[U]$ has at least two components. Let $u$ be a vertex in $U \setminus C_{i, v}$, so that $C_{i, u} \neq C_{i, v}$.
Observe that
\begin{equation}
      \label{eq:low_C_iv}
      |C_{i, v}| \geq 1 + d_{G_i[U]}(v) \geq  d_{G[U]}(v)
      \stackrel{\Cref{eq:low_deg_GU_v}}{\geq} 1 + t.
\end{equation}
Furthermore, we have
\begin{equation}
      \label{eq:low_sum_C_ju}
      \sum_{j \in [r]} |C_{j, u}| \geq \sum_{j \in [r]} \paren{1 + d_{G_j[U]}(u)} = r + d_{G[U]}(u)
      \stackrel{\Cref{eq:low_deg_GU_v}}{\geq} r + t + 1.
\end{equation}
Then,
\begin{align*}
      \sum_{j \in [r]} c(G_j[U]) = c(G_i[U]) + \sum_{j \in [r] \setminus \set{i}} c(G_j[U]) &\leq 2 + \paren{|U| - |C_{i, v}| - |C_{i, u}|} + \sum_{j \in [r] \setminus \set{i}} \paren{1 + |U| - |C_{j, u}|} \\
      &= r + 1  + r|U| - |C_{i, v}| - \sum_{j \in [r]} |C_{j, u}| \\
      &\stackrel{\Cref{eq:low_C_iv},\Cref{eq:low_sum_C_ju}}{\leq} r|U| - 2t - 1.
\end{align*}
However, we also have that
\begin{equation*}
      \sum_{j \in [r]} c(G_j[U]) \stackrel{\Cref{eq:low_bound_CGiU}}{\geq} \sum_{j \in [r]} \paren{|U| - 2\mu + 2|U_j|} = r|U| - 2t,
\end{equation*}
a contradiction. This concludes the proof of \Cref{claim:no_monochromatic_vertex_inside_U}.
\end{claimproof}

\begin{claim}\label{claim:up_sum_ui}
\begin{equation*}
      2(n - r\mu) \le \paren{\frac{1}{3} + \frac{1}{|U| - n + r\mu}}|U|.
\end{equation*}
\end{claim}
\begin{claimproof}
We denote $d_i(v) = d_{G_i[U]}(v)$. By \Cref{lem:up_bound_components_inverse_degree}, we have
$c(G_i[U]) \leq \sum_{v \in U} \frac{1}{d_i(v) + 1}$ for every $i \in [r]$, and therefore
\newpage
\begin{equation*}
      \sum_{i \in [r]} c(G_i[U]) \leq \sum_{i \in [r]} \sum_{v \in U} \frac{1}{d_i(v) + 1} = \sum_{v \in U} \sum_{i\in [r]} \frac{1}{d_i(v) + 1}.
\end{equation*}
Observe that $\sum_{i \in [r]} d_i(v) = d_{G[U]}(v)$, and by \Cref{claim:no_monochromatic_vertex_inside_U}, we have $d_i(v) \leq d_{G[U]}(v) - 2$ for all $i \in [r]$. 
Note that for any numbers  $0 < x \leq y$, we have $\frac{1}{x + 1} + \frac{1}{y + 1} < \frac{1}{x} + \frac{1}{y + 2}$. This tells us that for $r$ variables $x_i \in \naturals\cup \{0\}$ with a fixed sum satisfying $x_i\le z$ for every $i\in [r]$, the expression $\sum_{i \in [r]} \frac{1}{x_i + 1}$ attains its maximum value when every $x_i$ except at most one equals $z$ or $0$. 
In particular, we conclude that the maximum of $\sum_{i \in [r]} \frac{1}{d_i(v) + 1}$ is at most the value of this expression when $d_i(v)$ are all zero except for one equal to $2$ and one equal to $d_{G[U]}(v) - 2$. Thus,
\begin{equation*}
      \sum_{i \in [r]} \frac{1}{d_i(v) + 1} \leq \frac{r - 2}{1} + \frac{1}{2+1} +
      \frac{1}{d_{G[U]}(v)-2 + 1} \leq r - 2 +\frac{1}{3} + \frac{1}{|U| - n + r\mu},
\end{equation*}
where we have used that $d_{G[U]}(v) - 1 \ge \delta(G) - 1 - (n - |U|) \ge r\mu - n + |U|$ by~\eqref{eq:up_bound_mu} (note that $|U| - n + r\mu > 0$ by~\Cref{claim:low_sum_ui}). In particular,
\begin{equation*}
      \sum_{i \in [r]} c(G_i[U]) \leq \sum_{v \in U} \paren{r - 2 + \frac{1}{3} + \frac{1}{|U| - n + r\mu}} = |U|
      \paren{r - 2+ \frac{1}{3} + \frac{1}{|U| - n + r\mu}}.
\end{equation*}
We also have that
\begin{equation*}
      \sum_{i \in [r]} c\paren{G_i[U]} \stackrel{~\eqref{eq:low_bound_CGiU}}{\geq} \sum_{i \in [r]}
      \paren{|U| - 2\mu + 2|U_i|} = r|U| - 2r\mu + 2\sum_{i \in [r]} |U_i|
      \stackrel{~\eqref{eq:low_bound_U}}{\geq}
      (r-2)|U| + 2(n-r\mu).
\end{equation*}
Combining the above inequalities, we obtain
\begin{equation*}
       (r-2)|U| + 2(n-r\mu) \leq |U| \paren{r - 2 + \frac{1}{3} + \frac{1}{|U| - n + r\mu}},
\end{equation*}
and canceling the $(r-2)|U|$ from both sides completes the proof of the claim.
\end{claimproof}

We will now show that the two inequalities of \Cref{claim:low_sum_ui} and \Cref{claim:up_sum_ui} cannot hold simultaneously. Let $x = n - r\mu$, noting that 
\begin{equation}\label{eq:bound_x}
      x \stackrel{\Cref{eq:up_bound_mu}}{\ge} n - \delta(G) + 1\ge \frac{n}{r+1} + 1 \stackrel{\Cref{eq:bound_n}}{>} 2.
\end{equation}
Let $y = |U| - n + r \mu$, noting that by  \Cref{claim:low_sum_ui}, 
\[
     1 \le y \le 4\mu + 1 - n + r\mu.
\]
By \eqref{eq:up_bound_mu}, we have $\mu \le \frac{\delta(G) - 1}{r} \le \frac{n}{r+1}- \frac{1}{r}$. Substituting this in, we have
\begin{equation}\label{eq:bounds_for_y}
    1 \le y \le (r+4)\left(\frac{n}{r+1} - \frac{1}{r}\right) + 1 - n < \frac{3n}{r+1}  \stackrel{\Cref{eq:bound_x}}{<} 3x.
\end{equation}
Now, the inequality of \Cref{claim:up_sum_ui} can be rewritten in terms of $x$ and $y$ as
  \[
  2x \le \left(\frac{1}{3} + \frac{1}{y}\right)(x+y),
  \]
  which we can rearrange to get
  \[
  0 \le y^2  - 5yx  + 3y + 3x = - (y-1)(3x-y) - 2y(x-2).
  \]
By \eqref{eq:bounds_for_y}, we see that $y - 1 \ge 0$ and $3x - y >0$. Moreover,  by \eqref{eq:bound_x} we also have $x - 2 > 0$. Therefore we conclude that $- (y-1)(3x-y) - 2y(x-2) < 0$, which is a contradiction. 
This finishes the proof of \Cref{thm:large monochromatic matching}.
\end{proof}

\bigskip \section{Large monochromatic linear forests and proof of \texorpdfstring{\Cref{thm:deterministic main}}{Theorem 1.2}}
\label{sec:lonG_{n,M}ono_paths}
In this section, we execute Step~2 as described in \Cref{sec:proof ideas}. Recall that a forest is called a \defin{linear forest} if each of its connected components is a path. The main result of this section is \Cref{thm:large monochromatic linear forest}, which finds a large monochromatic linear forest in an edge-coloured pseudorandom graph. Similar to \cite{Gishboliner2021colourbiased}, instead of random graphs, we work with more general pseudorandom graphs. 

We will now introduce some useful notation. Suppose $G$ is a graph. For two disjoint non-empty sets $V_1,V_2\subseteq V(G)$, the density of $(V_1,V_2)$ is defined as $d_G(V_1,V_2) = \frac{e_G(V_1,V_2)}{|V_1||V_2|}$. For $\gamma, p\in (0,1]$, the graph $G$ is called \defin{$(\gamma, p)$-pseudorandom} if for every disjoint non-empty sets $V_1,V_2\subseteq V(G)$ with $|V_1|,|V_2|\ge \gamma |V(G)|$, we have $|d_G(V_1,V_2)-p|\le \gamma p$. Define the density of $V$ as $d_G(V) = \frac{e_G(V)}{\binom{|V|}{2}}$. We also note the following well-known fact (for a proof, see e.g.~\cite[Section~3]{Gishboliner2021colourbiased}). 
\begin{lemma}[\cite{Gishboliner2021colourbiased}]
\label{lem:pseudorandomness on single set}
    If $G$ is a $(\gamma, p)$-pseudorandom graph, then every set $V\subseteq V(G)$ of size at least $2\gamma |V|$ satisfies $|d_G(V) -p| \le \gamma p$.
\end{lemma}

For convenience, we will use the usual hierarchy notation throughout this section. For example, we say $\alpha \ll \beta, \gamma$ to mean that there is a non-decreasing function $g:(0,1]^2 \rightarrow (0,1]$ such that $\alpha \le g(\beta,\gamma)$. When there are multiple numbers in a hierarchy, they are chosen from right to left. We are now ready to state the main result of this section.

\begin{theorem} \label{thm:large monochromatic linear forest}
  Suppose $0 < \frac{1}{n} \ll \frac{1}{C}, \gamma \ll \alpha, \eps, \frac{1}{r} \le 1$. Let $G$ be an $n$-vertex $(\gamma,p)$-pseudorandom graph for some $p\in (0,1]$, and $H$ be an $\alpha$-residual spanning subgraph of $G$. Then, for every $r$-edge-colouring of $H$, there is a monochromatic linear forest of size $(1-\eps)\min\set{\frac{2\alpha n}{r}, \frac{2n}{r+1}}$ consisting of at most $C$ paths.
\end{theorem}

We use the above result to prove \Cref{thm:deterministic main} at the end of this section. Before proving \Cref{thm:large monochromatic linear forest}, we first prove the following corollary, which will later be useful in the proof of \Cref{thm:main,thm:hitting-time-main}. Note that the range of probability $p$ is relaxed below compared to what we assume in \Cref{thm:main}. 

\begin{corollary} \label{cor:random graph large monochromatic linear forest}
Let $r$ be a positive integer, and $\alpha, \eps \in (0,1]$. Then there exists $C >0$ such that if $p =  p(n) \ge \frac{C}{n}$ and $G \sim G(n,p)$, then with probability $1-o(n^{-3}),$ for every $r$-edge-colouring of every $\alpha$-residual spanning subgraph of $G$, there is a monochromatic linear forest of size  $(1-\eps)\min\set{\frac{2\alpha n}{r}, \frac{2n}{r+1}}$ consisting of at most $C$ paths. 
\end{corollary}

To prove this, we will use the following standard lemma, which says that random graphs are pseudorandom with appropriate parameters.
\begin{lemma} \label{lem:pseudorandom}
Suppose $0 < \frac{1}{C} \ll \gamma \le 1$. Let $G\sim G(n,p)$ with $p\ge \frac{C}{n}$. Then with probability $1-o(n^{-3})$, the graph $G$ is $(\gamma,p)$-pseudorandom.
\end{lemma}
\begin{proof}
Every two disjoint sets $A,B\subseteq V(G)$ with $|A|,|B|\ge \gamma n$ satisfy 
\[
p|A||B|\ge p\gamma^2 n^2\ge C\gamma^2 n\ge \frac{100n}{\gamma^2},
\]
where the last inequality uses the fact that $\frac{1}{C} \ll \gamma$. Thus, by applying \Cref{lem:large-sets-pseudorandom-property-Gnp} with $\lambda = \gamma$, we obtain that the following happens with probability $1-o(n^{-3})$. Every disjoint $A,B\subseteq V(G)$ with $|A|,|B|\ge \gamma n$ satisfy $e_G(A,B) = (1\pm \gamma)p|A||B|$. This proves the lemma.
\end{proof}

\begin{proof}[Proof of \Cref{cor:random graph large monochromatic linear forest}]

Let $r, \alpha, \eps$ be given as in the statement of \Cref{cor:random graph large monochromatic linear forest}. Choose $C, \gamma$ such that $0< \frac{1}{C} \ll \gamma \ll \alpha, \eps$. If $p \ge \frac{C}{n}$, \Cref{lem:pseudorandom} implies that with probability $1-o(n^{-3})$, the graph $G$ is $(\gamma,p)$-pseudorandom. Then applying \Cref{thm:large monochromatic linear forest} gives the result. 
\end{proof}

We next focus on proving \Cref{thm:large monochromatic linear forest}, for which we use the following `multicolour' variant of the sparse regularity lemma. The same version was also used in~\cite{Gishboliner2021colourbiased}. 

We now introduce a few definitions for convenience. Consider a graph $G$. A pair $(V_1,V_2)$ of disjoint vertex-sets is called \defin{$(\delta, q)$-regular} if for every $V'_1\subseteq V_1$ and $V'_2\subseteq V_2$ with $|V'_i| \ge \delta |V_i|$ for all $i\in [2]$, we have $|d(V'_1,V'_2) - d(V_1,V_2)|\le \delta q$. A partition of a set is called an \defin{equipartition} if the sizes of any two parts differ by at most $1$. 
Let $G_1,\dots,G_r$ be graphs on the same vertex set $V$ of size $n$. An equipartition $\{V_1,\dots,V_t\}$ of $V$ is said to be \defin{$(\delta)$-regular} with respect to $G_1,\dots,G_r$ if for all but at most $\delta \binom{t}{2}$ pairs $(V_i,V_j)$ with $1\le i<j\le t$, the pair $(V_i,V_j)$ is $(\delta, q)$-regular in $G_\ell$ for every $\ell \in [r]$, where $q= \frac{\sum_{i\in [r]}e(G_i)}{\binom{n}{2}}$.

\begin{theorem} [Multicolour sparse regularity lemma~\cite{Scott2011szemeredi}] \label{thm:regularity}
Let $0 < \frac{1}{n} \ll \frac{1}{T} \ll \delta, \frac{1}{t_0}, \frac{1}{r} \le 1$. Then, for every collection $G_1,\dots,G_r$ of graphs on the same vertex set $V$, there is an equipartition of $V$ that is $(\delta)$-regular with respect to $(G_1,\dots,G_r)$, and has at least $t_0$ and at most $T$ parts.
\end{theorem}

\begin{lemma} [Lemma~4.4 in \cite{Ben2012the}] \label{lem:almost Hamilton path}
Let $n,k$ be positive integers, and let $G$ be a bipartite graph with sides $X$ and $Y$ of size $n$ each. Suppose that there is an edge between every pair of sets $X'\subseteq X$ and $Y'\subseteq Y$ with $|X'| = |Y'| = k$. Then, $G$ contains a path of length at least $2n - 4k$.
\end{lemma}

We are now ready to prove \Cref{thm:large monochromatic linear forest}.
\begin{proof}[Proof of \Cref{thm:large monochromatic linear forest}]
Let $q = e(H)/\binom{n}{2}$. By pseudorandomness of $G$, we have $e(G) = (1 \pm \gamma)p\binom{n}{2}$. By assumption, $q \binom{n}{2} = e(H)\le e(G) \le (1+\gamma) p \binom{n}{2}$. Thus, we have 
\begin{equation} \label{eq:p q relation}
q \le (1+\gamma)p.
\end{equation}
For every $c\in [r]$, let $H_c$ denote the subgraph of $H$ induced by the edges coloured with the $c$-th colour. 
Introduce new parameters $T,\delta,\eps'$ satisfying $\eps = 4\eps'$ and $\frac{1}{C}, \gamma \ll \frac{1}{T} \ll \delta \ll \eps', \alpha, \frac{1}{r}$. Thus, we have
\[
0< \frac{1}{n} \ll \frac{1}{T} \ll \delta, \frac{1}{r} \le 1.
\]
Since the above relation holds, we apply \Cref{thm:regularity} with $t_0 = \frac{1}{\delta}$ to obtain an $(\delta)$-regular equipartition $\{V_1,\ldots,V_t\}$ of the vertex set of $H$ with respect to $(H_1,\dots,H_r)$, where $t_0 \le t \le T$. Thus, we have
\begin{equation} \label{eq:t vs delta}
2\gamma \le \frac{1}{t} \le \delta.
\end{equation}
To avoid unimportant technicalities, we will assume that $t$ divides $n$ and $|V_i| = \frac{n}{t}$ for every $i\in [t]$.
For every $i\in [t]$, we have $|V_i| = \frac{n}{t} \ge 2\gamma n$, using \eqref{eq:t vs delta}. Thus, by pseudorandomness of $G$ and \Cref{lem:pseudorandomness on single set}, we have $d_G(V_i) = (1\pm \gamma)p$ for every $i\in [t]$, and we have $d_G(V_i,V_j) = (1\pm \gamma)p$ for every distinct $i, j\in [t]$. Let $V=V(H)$.
For every $i\in [t]$, we have 
\[
e_H(V_i,V)=\sum_{v\in V_i} d_H(v)
\ge \sum_{v\in V_i} \alpha d_G(v)
= \alpha e_G(V_i,V)
= \alpha \sum_{j\in [t]\setminus \{i\}} e_G(V_i,V_j)
\ge \alpha (1-\gamma)p \frac{(t-1)n^2}{t^2},
\]
and thus, noting that $e_H(V_i)\le e_G(V_i)$, we have
\begin{align}\label{eq:lower bound on density sum}
\sum_{j\in [t]\setminus \{i\}} d_H(V_i, V_j) 
= \frac{e_H(V_i,V\setminus V_i)}{(\frac{n}{t})^2} 
= \frac{e_H(V_i,V) - 2e_H(V_i)}{(\frac{n}{t})^2} 
&\ge \frac{\alpha (1-\gamma)p\frac{(t-1)n^2}{t^2} - 2(1+\gamma)p\binom{n/t}{2}}{(\frac{n}{t})^2} \nonumber
\\&\ge (1 - \eps')\alpha tp,
\end{align}
where the last step uses \eqref{eq:t vs delta} and $\gamma, \delta \ll \eps', \alpha$.
Consider the graph $R$ on the vertex set $[t]$ where $ij$ is an edge if and only if $d_H(V_i,V_j) \ge 2 r \delta q$. 
For every $i\in [t]$, noting that $d_H(V_i,V_j) \le d_G(V_i,V_j) \le (1+\gamma)p$ for every $j\in [t]\setminus \{i\}$, we have 
\begin{equation}\label{eq:upper bound on density sum}
\sum_{j\in [t]\setminus \{i\}} d_H(V_i, V_j) \le d_R(i) \cdot (1+\gamma)p + t \cdot 2 r \delta q,
\end{equation}
Combining \eqref{eq:p q relation}, \eqref{eq:lower bound on density sum}, and \eqref{eq:upper bound on density sum}, and using $\gamma, \delta \ll \eps', \alpha$, we have $d_R(i)\ge (1 - 2\eps')\alpha t$ for every $i\in [t]$.   

Consider the graph $R'$ on the vertex set $[t]$ where $ij$ is an edge if and only if $(V_i,V_j)$ is $(\delta,q)$-regular in $H_c$ for every $c\in [r]$. By definition, $e(R')\ge (1-\delta)\binom{t}{2}$. Using this, if we let $S = \{i\in [t]: d_{R'}(i)\ge (1-\sqrt{\delta})t\}$, then $|S|\ge (1-\sqrt{\delta})t$. 

Consider the graph $R''$ on the vertex set $S$ where $ij$ is an edge if and only if $ij$ is an edge in both $R[S]$ and $R'[S]$ (i.e., the subgraphs of $R$ and $R'$ induced by $S$). Note that for every $i\in S$, we have 
\begin{equation}\label{eq:lower bound on degree of R''}
d_{R''}(i)\ge d_R(i) + d_{R'}(i) - t - |[t]\setminus S| 
\ge (1 - 2\eps')\alpha t + (1 - \sqrt{\delta}) t - t - \sqrt{\delta} t
\ge (1 - 3\eps')\alpha t, 
\end{equation}
where the last step uses $\delta \ll \eps', \alpha$.
We now assign a colour from $[r]$ to every edge of $R''$ in the following way. For $ij \in E(R'')$, choose a colour $c\in [r]$ such that $d_{H_c}(V_i,V_j)\ge 2\delta q$ (such $c$ exists since $\sum_{c\in [r]} d_{H_c}(V_i,V_j) = d_H(V_i,V_j) \ge 2r\delta q$). Note that if the edge $ij \in E(R'')$ gets the colour $c$, then for every $V'_i\subseteq V_i$ and $V'_j\subseteq V_j$ with $|V'_i|,|V'_j|\ge \frac{\delta n}{t}$, we have $d_{H_c}(V'_i,V'_j) \ge d_{H_c}(V_i,V_j) - \delta q \ge \delta q > 0$, and thus by \Cref{lem:almost Hamilton path} the bipartite subgraph $H_c[V_i,V_j]$ contains a path of length at least $(2-4\delta)\frac{n}{t}$.

Now, using \eqref{eq:lower bound on degree of R''} and the fact that $|S|\ge (1-\sqrt{\delta})t$, we can apply \Cref{thm:intro large monochromatic matching} to find a monochromatic matching $\mathcal{M}\subseteq E(R'')$ of size $m:= \min\set{\frac{(1 - 3\eps')\alpha t}{r}, \frac{(1-\sqrt{\delta})t -1}{r+1}}$ in $R''$, say in colour $c\in [r]$. Thus, for every $ij \in E(\mathcal{M})$, the bipartite graph $H_c[V_i,V_j]$ contains a path of length at least $(2-4\delta)\frac{n}{t}$. Taking union of these paths gives us the desired monochromatic (in colour $c$) linear forest of size at least $m \cdot (2-4\delta)\frac{n}{t} \ge (1-\eps)\min\set{\frac{2\alpha n}{r}, \frac{2n}{r+1}}$, where we use $\delta\ll \eps' = \frac{\eps}{4}$. Clearly, the number of paths in this linear forest is exactly the size of the matching $\mathcal{M}$, which is at most $T\le C$. This finishes the proof of \Cref{thm:large monochromatic linear forest}.

\end{proof}

\subsection{Proof of \texorpdfstring{\Cref{thm:deterministic main}}{Theorem 1.2}}
\deterministic*
To prove this, we will use \Cref{thm:large monochromatic linear forest} and the following classical result by P\'osa~\cite{Posa1963ForceHamilton}.

\begin{lemma}[\cite{Posa1963ForceHamilton}]
  \label{lem:posa-force-hamilton}
  Let $t \geq 0$ and let $G$ be an $n$-vertex graph with minimum degree at
  least $\frac{n}{2} + t$. Let $E \subseteq E(G)$ be a set of edges that forms a linear forest and has size at most $2t$. Then there exists a Hamilton cycle in $G$ which uses all edges in $E$.
\end{lemma}

\begin{proof}[Proof of \Cref{thm:deterministic main}]
Let $r, \alpha, \eps$ be as in \Cref{thm:deterministic main}. We introduce a new parameter $\gamma$ which is sufficiently small relative to $\eps$, and we will assume $n$ to be sufficiently large with respect to $\gamma$ and $\eps$. Let $H$ be an $r$-edge-coloured $n$-vertex graph with minimum degree at least $\alpha n$.

Let $t := \min\set{2(\alpha -1)\frac{n}{2}, \frac{\alpha n}{r}, \frac{n}{r + 1}}$. Let $G$ be the $n$-vertex complete graph with the same vertex set as~$H$. Note that $G$ is $(\gamma, 1)$-pseudorandom, and $H$ is an $\alpha$-residual spanning subgraph of the complete graph $K_n$. Therefore, by \Cref{thm:large monochromatic linear forest}, $H$ contains a monochromatic linear forest of size $(1 - \eps)\cdot 2t$. Let $E$ be the edges of such a monochromatic linear forest. Clearly, the minimum degree of $H$ is at least $\alpha n = \frac{n}{2} + 2(\alpha -1)\frac{n}{2} \ge \frac{n}{2} + t$. Hence, by applying \Cref{lem:posa-force-hamilton} on $H$, we obtain a Hamilton cycle in $H$ which uses all the edges in $E$. Therefore, given that the set of edges $E$ is monochromatic, there is a Hamilton cycle in $H$ with one colour appearing at least $(1 - \eps)\cdot 2t$ times, as desired.
\end{proof}

\bigskip \section{Cleaning up}
\label{sec:clean-up}

In this section, we walk through Step~3 as described in \Cref{sec:proof ideas}. Going into this step, we have a monochromatic linear forest in the graph $H$ consisting of a constant number of paths spanning a large vertex set $S$. Let $\overline{S}$ be the complement of $S$. In order to adapt something analogous to the `resilience' result \cite[Theorem~1.6]{Montgomery2019resilient}, we would like it to be the case that for some $\delta > 0$, we have
\[
    d_{H}\left(v,\overline{S}\right) \ge \left(\frac{1}{2} + \delta\right)d_{G}(v, \overline{S})  \quad \text{for all } v \in \overline{S}.
\]
In addition, we would also like a somewhat large minimum degree condition in the graph $H[\overline{S}]$. Therefore, the first step of the clean-up process will be to move from $\overline{S }$  to $S$ all vertices $v$ with $d_H\left(v,\overline{S}\right)$ too small. The second step will be to amend and extend the linear forest to cover these extra vertices.

The following standard variant of Hall's Marriage Theorem will be handy in this section. The condition in the following lemma will sometimes be referred to as \emph{Hall's condition}.
\begin{lemma}[Bigamy Variant of Hall's Marriage Theorem]\label{thm:Hall_variant}
    Let $G$ be a bipartite graph with vertex classes $X$ and $Y$. Suppose for all $Z \subseteq Y$, we have $|N(Z)| \ge 2|Z|$. Then $G$ has two matchings $\mathcal{M}_1$ and $\mathcal{M}_2$ covering $Y$ such that the sets $V(\mathcal{M}_1)\cap X$ and $V(\mathcal{M}_2)\cap X$ are disjoint.
\end{lemma}
\begin{proof}
This follows immediately from the usual version of Hall's Marriage Theorem by considering an auxiliary graph with each vertex in $Y$ duplicated.
\end{proof}

The main result of this section is the following. 

\begin{theorem} \label{thm:clean-up_main}
    Let $r\ge 2$, $C\ge 0$, $\eps \in (0,\frac{1}{2}]$, and $\alpha \in \left[\frac{1}{2} + \frac{1}{2r}, 1\right]$.
    Then there exists $\delta >0$ such that the random graph $G \sim G(n,p)$ with $p \ge \frac{\log{n}}{2n}$ satisfies the following with probability $1-o(n^{-3})$.
    Let $H$ be an $\alpha$-residual spanning subgraph of $G$ such that the following hold. 
    \begin{enumerate}[leftmargin=*,label = {\bfseries A\arabic{enumi}}]
    \item Every vertex $v\in V(H)$ satisfies $d_H(v)\ge 2$.
    \label{item:mindeg2}
    \item No two vertices $u,v\in V(H)$ satisfying $d_H(u),d_H(v)\le \frac{np}{5000}$ are within distance $5$ from each other.
    \label{item:property low_deg_vertices_far_apart}
    \end{enumerate}
    Let $\mathcal{P}$ be a linear forest in $H$ spanning a vertex set $S$ where ${|S| = (1-\eps) \min\Set{(2\alpha - 1)n, \frac{2\alpha n}{r}, \frac{2n}{r+1}}}$. Suppose $\mathcal{P}$ contains at most $C$ paths.
    Then, there exists a set $T \supseteq S$ and a linear forest $\mathcal{P}^*$ in $H$ covering $T$ such that, setting $U = V(H)\setminus T$, the following conditions hold.
    
    \begin{enumerate}[leftmargin=*,label = {\bfseries B\arabic{enumi}}]
    \item $d_{H}\left(u,U\right) \ge \max\left(\frac{p|U|}{5000}, \; \left(\frac{1}{2}+\delta\right)d_{G}\left(u,U\right)\right)$ for all $u \in U$.
    \label{eq:degree_in_U}
    \item $|U| \ge n - |S| - \frac{300r}{\delta^2 p}$.
    \label{eq:size_of_U}    
    \item All but at most $\frac{2000r}{\delta^2 p}$ edges of $\mathcal{P}$ are included in $\mathcal{P}^*$.\label{eq:most_of_path_used}
    \item There are at most $\frac{2000r}{\delta^2 p}$ paths in $\mathcal{P}^*$.
    \label{eq:number_of_paths}
    \item All paths in $\mathcal{P}^*$ have at least $3$ vertices, with endpoints lying in $U$ and all other vertices lying in~$T$. \label{eq:matching}
    \item If $W$ is the set of all endpoints of paths in $\mathcal{P}^*$, then $d_H(u, W) \le (np)^{1/5}$ for all $u \in U$.
    \label{eq:deg_U_to_M}
    \end{enumerate}
\end{theorem}

As briefly described in \Cref{sec:prop-binom-rand}, we relax the bound on $p$ compared to \Cref{thm:main} and use the specific error probability $o(n^{-3})$ in the above result to facilitate the proof of our hitting time result \Cref{thm:hitting-time-main}. Since in this generality the properties \ref{item:mindeg2} and \ref{item:property low_deg_vertices_far_apart} do not hold, we separately assume them about the graph $H$. However, this will not be a problem when we come to prove \Cref{thm:hitting-time-main}, as there we will have~\ref{item:mindeg2} by assumption and~\ref{item:property low_deg_vertices_far_apart} \whp by~\Cref{lem:small_deg_vertex_separation_in_G_{n,M}}.

\begin{proof}[Proof of \Cref{thm:clean-up_main}]
Throughout the proof, we will assume the following to support our arguments. 
\begin{enumerate}[leftmargin=*,label = {\bfseries P0}]
\item Let $r, C, \eps, \alpha$ be as in \Cref{thm:clean-up_main}. We choose $\delta >0$ sufficiently small relative to $\eps$ and $r$. We also assume $n$ to be sufficiently large relative to all of these constants. Let $p \ge \frac{\log n}{2n}$. 
\label{item:parameter hierarchy}
\end{enumerate}
Then with probability $1-o(n^{-3})$, the random graph $G$ has the following properties. We provide brief justifications afterwards.
\begin{enumerate}[leftmargin=*,label = {\bfseries P\arabic{enumi}}]
\item Every vertex $v\in V(G)$ satisfies $d_G(v) \le 10np$. \label{item:property 1}
\item Let $A,B\subseteq V(G)$, not necessarily disjoint. Then \label{item:property 2 both}
    \begin{enumerate} 
    \item if $p|A||B|\ge \frac{100n}{\delta^2}$, then $e_G(A,B) = (1 \pm \delta)p|A||B|$, and
    \label{item:property 2}
    \item if $p|A||B|\ge 150n$, then $e_G(A,B) \ge \frac{p|A||B|}{2500}$.
    \label{item:property 2 alt version}
    \end{enumerate}
    \item For every set $S$ of size at most $\frac{6000r}{\delta^2 p}$, we have $e_G(S)\le |S|(np)^{1/8}$.
    \label{item:property 3}
    \item We have $|\{v \in V: d_G(v) \le \frac{np}{2500}\}|\le \frac{200}{p}$.
    \label{item:property 4}
\end{enumerate}
Indeed,~\ref{item:property 1} is immediate from~\Cref{lem:upper_bound_degrees}. \ref{item:property 2 both} is immediate by applying~\Cref{lem:large-sets-pseudorandom-property-Gnp} with $\lambda = \delta$ and $\lambda = 1 - \tfrac{1}{2500}$ respectively. 
To get \ref{item:property 3}, apply \Cref{lem:small-sets-cannot-contain-too-many-edges} with $a=\frac{15}{16}$ and $b=\frac{1}{8}$ and then notice that $(np)^{1/16} \ge (\frac{\log{n}}{2})^{1/16} \ge \frac{6000r}{\delta^2}$ using the assumptions in \ref{item:parameter hierarchy}, and so $n(np)^{-15/16} \ge \frac{6000r}{\delta^2 p}$. To see that \ref{item:property 4} holds, apply \Cref{lem:large-sets-pseudorandom-property-Gnp} with $A = \{v \in V: d_G(v) \le \frac{np}{2500}\}$ and $B = V(G)$. We have $e_G(A,V(G)) \le \frac{|A|np}{2500}$ by definition of $A$, and so by \Cref{lem:large-sets-pseudorandom-property-Gnp} we must have $p|A||V(G)| = p|A|n \le \left(\frac{2500}{2499}\right)^2 100n$, which gives $|A| < \frac{200}{p}$, as required.

To prove \Cref{thm:clean-up_main}, it is sufficient to prove \ref{eq:degree_in_U}--\ref{eq:deg_U_to_M} assuming that $G$ is an $n$-vertex graph satisfying the above properties. (In other words, \ref{item:property 1}--\ref{item:property 3} are the only properties of $G(n,p)$ we will use in our proof.)

Let $K =\{v \in V(H): d_H(v) \le \frac{np}{5000}\}$. For every $v\in K$, we have $d_G(v)\le \frac{d_H(v)}{\alpha}\le \frac{np}{2500}$, using $\alpha \ge \frac{1}{2}$. Thus, by \ref{item:property 4}, we deduce that $|K|\le \frac{200}{p}$. 

Now, using \ref{item:mindeg2}, for each vertex $v \in K$ we can pick two neighbours $x_1(v)$ and $x_2(v)$ of $v$ in $H$. For each $v \in K$, define $P_v$ to be the path $x_1(v),v,x_2(v)$. By \ref{item:property low_deg_vertices_far_apart}, all $v$ in $K$ are distance at least 5 apart in $H$ and therefore, the paths in $\{P_v: v \in K\}$ are pairwise vertex disjoint. Moreover, for any vertex $y \in V(H)$, $y$ cannot be adjacent to vertices in distinct $P_u,P_v$, $u,v \in K$, or $u,v$ would be distance at most 4 apart. In particular, for all $y \in V(H)$, $y$ has at most two neighbours in $\bigcup_{v\in K}\{x_1(v),x_2(v)\}$. Let $K' \coloneqq \bigcup_{v\in K}\{x_1(v),x_2(v)\}$.
We now record the following properties of $H$ for later use.
\begin{enumerate}[leftmargin=*, label = {\bfseries A\arabic{enumi}}, start=3]
    \item The set $K \coloneqq \{v \in V(H): d_H(v) \le \frac{np}{5000}\}$ contains at most $\frac{200}{p}$ vertices.
    \label{item:property K small}
    \item There is a set $K'\subseteq V(H)$ with a partition $\{K'_v: v\in K\}$ such that for every $v\in K$, the set $K'_v$ consists of exactly two neighbours of $v$ in $H$. (In particular, $|K'|=2|K|$.) Moreover, every vertex $u\in V(H)$ has at most two neighbors in $K'$.
    \label{item:property of K'}
\end{enumerate}
We will use the following two technical lemmas in our proof. We first prove \Cref{thm:clean-up_main} assuming these lemmas and later prove them at the end of this section.

\begin{lemma}\label{lem:combined}
Let $r,C,\eps,\alpha,\delta,n,p$ be as in \ref{item:parameter hierarchy}.
Let $G$ be an $n$-vertex graph satisfying \ref{item:property 2 both}--\ref{item:property 3}, and let $H$ be an $\alpha$-residual spanning subgraph of $G$ satisfying \ref{item:mindeg2}--\ref{item:property of K'}. Let $K$ and $K'$ be as in \ref{item:property K small} and \ref{item:property of K'}.
Let $S \subseteq V(H)$ be such that $|S| = (1-\eps) \min\Set{(2\alpha - 1)n, \frac{2\alpha n}{r}, \frac{2n}{r+1}}$.
Then, there exist sets $A$ and $T$ with $A\subseteq S \cup K \cup K' \subseteq T$ such that, setting $U = V(H)\setminus T$, the following hold. 
\begin{enumerate}[leftmargin=*,label = {\bfseries C\arabic{enumi}}]
    \item
    $d_{H}\left(u,U\right) \ge \max\left(\frac{p|U|}{5000}, \; \left(\frac{1}{2}+\delta\right)d_{G}\left(u,U\right)\right)$ for all $u \in U\cup A$. 
   \label{item:degree_in_unused_part}
    \item $|T\setminus A| \le \frac{200r}{\delta^2 p}.$ 
    \label{item:size_SminusA}
    \item There exists an ordering $y_1,y_2,y_3,\ldots $ on $T \setminus A$ such that for all sets $Z \subseteq T \setminus A$,
    \[\left|\bigcup_{y_i \in Z} (N_H(y_i) \cap (A \cup \{y_1,y_2,\ldots,y_{i-1}\}))\right| \ge 2|Z|.\] \label{item:Hall's condition}
\end{enumerate}
\end{lemma}

\begin{lemma}\label{lem:Hall_to_extend_paths}
    Let $H$ be a graph, and let $A$ and $T$ be sets such that $A\subseteq T\subseteq V(H)$ and \ref{item:Hall's condition} holds. Let $\mathcal{P}_1$ be a linear forest in $H$ spanning $A$.
    Then, there exists a linear forest $\mathcal{P}_2$ in $H$ spanning $T$ such that the following hold.
    \begin{enumerate}[leftmargin=*,label = {\bfseries D\arabic{enumi}}]
        \item The endpoints of all paths in $\mathcal{P}_2$ lie in $A$. \label{eq:endpoints_in_R}
        \item All but at most $4|T \setminus A|$ edges of $\mathcal{P}_1$ are used in $\mathcal{P}_2$.
        \label{item:few edges left}
        \item The number of paths in $\mathcal{P}_2$ minus the number of paths in $\mathcal{P}_1$ is at most $3|T\setminus A|$.
        \label{item:few paths added}
    \end{enumerate}
\end{lemma}
We now use these lemmas to conclude the proof of \Cref{thm:clean-up_main}. Recall that $\mathcal{P}[S \setminus (K \cup K')]$ denotes the subgraph of $\mathcal{P}$ induced by $S \setminus (K \cup K')$. Since $\mathcal{P}$ is a linear forest, $\mathcal{P}[S \setminus (K \cup K')]$ is also a linear forest. Then let $\mathcal{P}_0$ be the linear forest formed by taking the union of $\mathcal{P}[S \setminus (K \cup K')]$ and all the paths of $\{P_v:v\in K\}$. Note that $\mathcal{P}_0$ spans $S \cup K \cup K'$.

Take $A$ and $T$ with $A\subseteq S\cup K \cup K'\subseteq T$ as given by \Cref{lem:combined}. We will next construct a linear forest $\mathcal{P}_2$ spanning $T$, ensuring that the endpoints of the paths in $\mathcal{P}_2$ lie in $A$. This will allow us to extend to a desired linear forest $\mathcal{P}^*$ using \ref{item:degree_in_unused_part}. For this, let $\mathcal{P}_1$ be the subgraph of $\mathcal{P}_0$ induced by $A$. Since $\mathcal{P}_0$ is a linear forest, $\mathcal{P}_1$ is also a linear forest.
Finally, let $\mathcal{P}_2$ be the linear forest spanning $T$ given by~\Cref{lem:Hall_to_extend_paths} applied to the linear forest $\mathcal{P}_1$. 
Note that \ref{eq:degree_in_U} follows directly from \ref{item:degree_in_unused_part}.

We have
\begin{align*}
    |U| = n - |T| = n - |A| - |T\setminus A| \ge n - |S \cup K \cup K'| - |T\setminus A|
    &\ge n - |S| - 3|K| - |T\setminus A| \\
    &\ge n - |S| - \frac{300r}{\delta^2 p},
\end{align*}
where in the final step we have used \ref{item:property K small} and \ref{item:size_SminusA}. This proves~\ref{eq:size_of_U}.

By construction, the number of edges in $\mathcal{P}$ not in $\mathcal{P}_0$ is at most $2|K \cup K'|\le 6|K|$. Similarly, the number of edges in $\mathcal{P}_0$ not in $\mathcal{P}_1$ is at most $2|(S\cup K \cup K')\setminus A|$. By \ref{item:few edges left}, the number of edges in $\mathcal{P}_1$ not in $\mathcal{P}_2$ is at most $4|T\setminus A|$. Thus, the number of edges in $\mathcal{P}$ not in $\mathcal{P}_2$ is at most 
\[
6|K| + 2|(S\cup K \cup K')\setminus A| + 4|T\setminus A|\le 6|K| + 6|T\setminus A|\le \frac{2000r}{\delta^2 p},
\]
where the last inequality uses \ref{item:property K small} and \ref{item:size_SminusA}. This proves \ref{eq:most_of_path_used} for $\mathcal{P}_2$ (in place of $\mathcal{P}^*$). 

The number of paths in $\mathcal{P}[S \setminus (K \cup K')]$ is at most $|K| + |K'| = 3|K|$ more than the number of paths in $\mathcal{P}$. Thus the number of paths in $\mathcal{P}_0$ minus the number of paths in $\mathcal{P}$ is at most $3|K|+|K| = 4|K|$.
By construction, the number of paths in $\mathcal{P}_1$ minus the number of paths in $\mathcal{P}_0$ is at most $|(S\cup K \cup K')\setminus A|$. By \ref{item:few paths added}, the number of paths in $\mathcal{P}_2$ minus the number of paths in $\mathcal{P}_1$ is at most $3|T\setminus A|$. By assumption, the number of paths in $\mathcal{P}$ is at most $C \le \frac{200r}{\delta^2 p}$, where this inequality follows from the assumptions in \ref{item:parameter hierarchy}. Thus, the number of paths in $\mathcal{P}_2$ is at most 
\[
4|K| + |(S\cup K \cup K')\setminus A| + 3|T\setminus A| + \frac{200r}{\delta^2 p}\le 4|K| + 4|T\setminus A| + \frac{200r}{\delta^2 p}\le \frac{1000r}{\delta^2 p},
\]
where the last inequality uses \ref{item:property K small} and \ref{item:size_SminusA}.
This proves \ref{eq:number_of_paths} for $\mathcal{P}_2$ (in place of $\mathcal{P}^*$). 

Note that 
\begin{equation}\label{eq:lower bound on U}
    |U|\ge n-|S|-|T\setminus A| \;\overset{\text{\ref{item:size_SminusA}}}{\ge}\; n-\frac{2n}{r+1} -\frac{200r}{\delta^2 p}
    \quad\overset{\text{\ref{item:parameter hierarchy}}}{\ge} \quad\frac{n}{4}. 
\end{equation}
Let $Z$ be the set of endpoints of paths in $\mathcal{P}_2$, noting that $|Z| \le \frac{2000r}{\delta^2p}$. By \ref{eq:endpoints_in_R}, we have $Z \subseteq A$. Thus, by \ref{item:degree_in_unused_part}, each vertex $v \in U\cup Z$ satisfies $d_H(v,U) \ge \frac{p|U|}{5000} \ge \delta np$, where the last inequality uses \eqref{eq:lower bound on U} and \ref{item:parameter hierarchy}. By \ref{item:property 1}, for every $v\in V(H)$, we have $d_H(v,U) \le d_G(v) \le 10np$. Combining the last two inequalities, we get 
\begin{equation}\label{eq:bounds on degree to U}
    \delta np \le d_H(v,U) \le 10np \quad \text{ for every }  v \in U\cup Z.
\end{equation}

For every path $P$ in $\mathcal{P}_2$ with distinct endpoints $z_1,z_2\in Z$, we now wish to find two edges $z_1w_1, z_2w_2\in E(H)$ with $w_, w_2\in U$ so that appending these edges can extend $P$ to a path with endpoints in $U$. Although this will help us establish \ref{eq:matching}, we need to keep a few things in mind. First, we need to simultaneously ensure that \ref{eq:deg_U_to_M} holds. Second, in order for \ref{eq:matching} to hold, for any $z \in Z$ that is the only vertex of its path in $\mathcal{P}_2$, we will need two edges from $z$ to $U$ so that the path can be extended to have both endpoints in $U$. Third, all vertices in $U$ used to extend different paths in $\mathcal{P}_2$ must be distinct. This motivates us to construct two matchings $\mathcal{M}_1, \mathcal{M}_2$ between $Z$ and $U$ such that $\mathcal{M}_1, \mathcal{M}_2$ cover $Z$ and the sets $V(\mathcal{M}_1)\cap U$ and $V(\mathcal{M}_2)\cap U$ are disjoint. If we can do so, then we will construct a desired linear forest $\mathcal{P}^*$ from $\mathcal{P}_2$ in the following way. Consider the edge set $E\subseteq E(\mathcal{M}_1)$ formed by taking all edges in $\mathcal{M}_1$ incident to the vertices $z\in Z$ that are the only vertex of its path in $\mathcal{P}_2$. We now define $\mathcal{P}^*$ to be the graph obtained from adding the edges in $E\cup E(\mathcal{M}_2)$ to the linear forest $\mathcal{P}_2$. By construction, $\mathcal{P}^*$ is a linear forest and satisfies \ref{eq:matching}.
Moreover, note that as \ref{eq:most_of_path_used} and \ref{eq:number_of_paths} hold for $\mathcal{P}_2$, they will also hold for $\mathcal{P}^*$. We still need to ensure \ref{eq:deg_U_to_M}.

To this end, we will use the Lov\'asz Local Lemma to find a set $W'\subseteq U$ so that \ref{eq:deg_U_to_M} holds for $W'$ (in place of $W)$. We will then construct the matchings $\mathcal{M}_1, \mathcal{M}_2$ between $Z$ and $W'$ with the properties discussed in the last paragraph. This will ensure \ref{eq:deg_U_to_M}. Fix $q \coloneqq (np)^{-5/6}$. Select a random set $W' \subseteq U$ by including each vertex of $U$ independently with probability $q$. For every $v\in U\cup Z$, let $B_v$ denote the event that it does not hold that $d_H(v,W') = (1\pm \frac{1}{2})qd_H(v,U)$. Clearly, for every $v\in V(H)$, the random variable $d_H(v,W')$ is distributed as $\bin(d_H(v,U),q)$. Thus, we have $\Expect(d_H(v,W'))= qd_H(v,U)\ge \delta (np)^{1/6} $,  where the last inequality uses \eqref{eq:bounds on degree to U}. Therefore, by Chernoff's bound, for every $v\in U\cup Z$, we have the following.
\begin{align*}
    \Prob{B_v} \le 2\exp\left(-\frac{\Expect(d_H(v,W'))}{12}\right) \le 2\exp\left(-\frac{\delta(np)^{1/6}}{12}\right) \le (np)^{-3},
\end{align*}
where the last inequality uses the facts that $np \ge \frac{\log{n}}{2}$ and $n$ is sufficiently large relative to $\delta$.

Let $\mathcal{B} = \{B_v : v \in U \cup Z\}$. Observe that if two distinct vertices $u,v$ have no common neighbours in $H$ then the events $B_u$ and $B_v$ are independent. In particular, $B_u$ and $B_v$ are \emph{not} independent only if $u$ and $v$ are at most distance $2$ apart. By \ref{item:property 1}, for every vertex $v$, the number of vertices at distance at most $2$ from $v$ is at most $(10np)^2$. Thus, each event in $\mathcal{B}$ is independent from all but at most $(10np)^2$ others. 

We now apply \Cref{lem:LLL} with $\mathcal{B} = \{B_v : v \in U \cup Z\}$, $p = (np)^{-3}$, and $d = \left(10np\right)^2$. As $ep(d+1)<1$, there is a positive probability that no event in $\mathcal{B}$ occurs. In particular, this means that there must exist some set $W' \subseteq U$ such that 
\[
\text{every vertex} \;\; u \in U \;\; \text{satisfies} \;\; d_H(u,W') \le \frac{3}{2} q d_H(u,U) \overset{\eqref{eq:bounds on degree to U}}{\le} 15(np)^{1/6} \overset{\text{\ref{item:parameter hierarchy}}}{\le} (np)^{1/5}
\]
\[
\text{and every vertex} \;\; z \in Z \;\; \text{satisfies} \;\; d_H(z,W') \ge \frac{1}{2} q d_H(z,U) \overset{\eqref{eq:bounds on degree to U}}{\ge} \frac{\delta}{2}(np)^{1/6}.
\]
Fix such a set $W'$. This $W'$ clearly satisfies \ref{eq:deg_U_to_M}. To finish the proof, we will use \Cref{thm:Hall_variant} on the bipartite subgraph of $H$ induced by $Z,W'$ to obtain the matchings $\mathcal{M}_1, \mathcal{M}_2$ between $Z$ and $W'$ with the desired properties as described before. 
We will check Hall's condition that for all $Z' \subseteq Z$, we have $|N_H(Z') \cap W'| \ge 2|Z'|$. Suppose, for contradiction, this does not hold for some non-empty $Z'$. We know that $d_H(z,W') \ge \frac{\delta}{2}(np)^{1/6}$ for all $z \in Z$ and so
\begin{equation}\label{eq:lower bound on edges in Z' and its neighbor}
e_H(Z',N_H(Z')\cap W') \ge \sum_{z \in Z'} d_H(z,W') \ge \frac{\delta}{2}|Z'|(np)^{1/6}.
\end{equation}
Now, since $|N_H(Z') \cap W'| < 2|Z'|$, we have 
$|Z' \cup \left(N_H(Z') \cap W'\right)| < 3|Z'| \le 3|Z| \le \frac{6000r}{\delta^2 p}$. 
Thus, applying \ref{item:property 3} with $S=Z' \cup \left(N(Z') \cap W'\right)$, we have 
\[
e_H(Z',N(Z')\cap W') \le e_G(Z' \cup \left(N(Z') \cap W'\right)) \overset{\text{\ref{item:property 3}}}{\le} 3|Z'|(np)^{1/8},
\]
which contradicts \eqref{eq:lower bound on edges in Z' and its neighbor} because $np\ge \frac{\log{n}}{2}$ and $n$ is sufficiently large relative to $\delta$. Thus, Hall's condition holds, and we can apply~\Cref{thm:Hall_variant} to find the two matchings $\mathcal{M}_1, \mathcal{M}_2$ as required. 

If we denote by $W$ the set of all endpoints of paths in $\mathcal{P}^*$, then we have $W\subseteq W'$. Thus, it is now easy to see that as \ref{eq:deg_U_to_M} holds for $W'$, it also holds for $W$. This finishes the proof of \Cref{thm:clean-up_main}.
\end{proof}

We now prove the lemmas used in the proof of \Cref{thm:clean-up_main}.
\begin{proof}[\textbf{Proof of \Cref{lem:combined}}]
In this proof, let $V=V(H)$, and for a set $Y\subseteq V$, we write $\overline{Y}$ to denote $V\setminus Y$.

Let $S' = S \cup K \cup K'$.
We define a sequence of sets as follows. Let $S_0=S'$ and for all $i \ge 0$ do the following:
 if there exists some vertex $v \not\in S_i$ such that $d_H(v,\overline{S_i}) < \max\set{\frac{p|\overline{S_i}|}{5000}, \left(\tfrac{1}{2}+\delta\right)d_G(v,\overline{S_i})}$, set $x_{i+1}$ to be such a vertex and let $S_{i+1}= S_i \cup \{x_{i+1}\}$. 
Otherwise, if there are no such vertices, set $T=S_i$ and end the process.

Let $\ell = |T\setminus S'|$ denote the number of steps in this process before it ends, so in particular $T=S_{\ell}$. Let $U = \overline{T}$. By construction we have $d_H\left(u,U\right) \ge \max\set{\frac{p|U|}{5000}, \left(\tfrac{1}{2}+\delta\right)d_G(u,U)}$ for all $u \in U$ and so \ref{item:degree_in_unused_part} holds for vertices in $U$.
Let $A \coloneqq \{v \in S' : d_H(v,U) \ge \max\set{\frac{p|U|}{5000}, \left(\tfrac{1}{2}+\delta\right)d_G(v,U)}\}$. Then \ref{item:degree_in_unused_part} also holds for vertices in $A$.

Next, we will show that \ref{item:size_SminusA} holds. By definition of $A$, we see that for every $y \in S' \setminus A$ we have 
\[
d_H(y,U) < \max\left\{\frac{p|U|}{5000}, \left(\tfrac{1}{2}+\delta\right)d_G(y,U)\right\} \le \max\left\{\frac{p|\overline{S}|}{5000}, \left(\tfrac{1}{2}+\delta\right)d_G(y,\overline{S})\right\}
\]
and thus,
\begin{equation}\label{eq:S minus A has high edges in T}
d_H(y,T) = d_H(y) - d_H(y,U) > d_H(y) - \max\left\{\frac{p|\overline{S}|}{5000}, \left(\tfrac{1}{2}+\delta\right)d_G(y,\overline{S})\right\}.
\end{equation}
For every $i \in [\ell]$, using how the set $S_i$ was constructed, we have 
\[d_H(x_i,\overline{S_{i-1}}) < \max\left\{\frac{p|\overline{S_{i-1}}|}{5000}, \left(\tfrac{1}{2}+\delta\right)d_G(x_i,\overline{S_{i-1}})\right\} \le \max\left\{\frac{p|\overline{S}|}{5000}, \left(\tfrac{1}{2}+\delta\right)d_G(x_i,\overline{S})\right\}
\]
and thus,
\begin{equation}\label{eq:T minus S has high edges in S}
d_H(x_i,T) \ge d_H(x_i,S_{i-1}) = d_H(x_i) - d_H(x_i,\overline{S_{i-1}}) 
> d_H(x_i) - \max\left\{\frac{p|\overline{S}|}{5000}, \left(\tfrac{1}{2}+\delta\right)d_G(x_i,\overline{S})\right\}.
\end{equation}
In particular, by~\eqref{eq:S minus A has high edges in T} and~\eqref{eq:T minus S has high edges in S}, we see that for all $y \in T \setminus A$, we have
\begin{equation}\label{eq:T minus A has high edges in T}
d_H(y,T) > d_H(y) - \max\left\{\frac{p|\overline{S}|}{5000}, \left(\tfrac{1}{2}+\delta\right)d_G(y,\overline{S})\right\} \ge \alpha d_G(y) - \max\left\{\frac{p|\overline{S}|}{5000}, \left(\tfrac{1}{2}+\delta\right)d_G(y,\overline{S})\right\}.
\end{equation}
Recall that $\ell = |T\setminus S'|$. Let $k = \min\set{\ell,\frac{200r}{\delta^2 p}}$. Define the set 
\[
Z =
\begin{cases*} 
\{x_1,\dots,x_k\}, & if $k< \ell$
\\T\setminus A, & otherwise.
\end{cases*}
\]
For every $i\in [k]$, we have $d_G(x_i,S_k)\ge d_G(x_i,S_{i-1})\ge d_H(x_i,S_{i-1})$ and for every $y\in T\setminus A$, we have $d_G(y,S_{\ell})= d_G(y,T)\ge d_H(y,T)$. 
Thus, using \eqref{eq:T minus S has high edges in S}, \eqref{eq:T minus A has high edges in T}, and the definition of $Z$, we have the following (for both the cases $k< \ell$ and $k=\ell$).
\[
e_G(Z,S_k) = \sum_{y \in Z} d_G(y,S_k) > \sum_{y \in Z}\left(\alpha d_G(y) - \max\left\{\frac{p|\overline{S}|}{5000}, \left(\tfrac{1}{2}+\delta\right)d_G(y,\overline{S})\right\}\right).
\]
Define $L \coloneqq \left\{y \in V: \left(\tfrac{1}{2}+\delta\right)d_G(y,\overline{S})< \frac{p|\overline{S}|}{5000}\right\}$.
Then we have 
\begin{align}
    e_G(Z,S_k) &> \sum_{y \in Z}\alpha d_G(y) - \sum_{y \in Z \setminus L} \left(\tfrac{1}{2}+\delta\right)d_G(y,\overline{S}) - \sum_{y \in L\cap Z}\frac{p|\overline{S}|}{5000} \nonumber\\
    &= \alpha e_G(Z,V)  - \left(\tfrac{1}{2}+\delta\right)e_G(Z\setminus L,\overline{S}) - \frac{p|L\cap Z||\overline{S}|}{5000} \label{eq:T minus A to S lower bound}
\end{align}
Assume for contradiction that $|T\setminus A| > \frac{200r}{\delta^2 p}$.  Then, $|Z|\ge \frac{200r}{\delta^2 p}$ and so $p|Z|n >  \frac{100n}{\delta^2}$. In particular, by~\ref{item:property 2}, we have
\begin{equation}\label{eqn:bound_Z_V}
    e_G(Z,V) \ge (1-\delta)p|Z|n.
\end{equation}
Note that $e_G(L,\overline{S}) < \frac{p|L||\overline{S}|}{\left(\tfrac{1}{2}+\delta\right)5000} < \frac{p|L||\overline{S}|}{2500}$ and therefore, by \ref{item:property 2 alt version}, we must have $p|L||\overline{S}| <  150n$. Since $|\overline{S}| = n-|S| \ge \left(1-\frac{2}{r+1}\right)n \ge \frac{n}{3}$, we get in particular that $|L| \le \frac{450}{p}$. 
Thus, $|Z \setminus L| \ge \frac{200r}{\delta^2 p} - |L| \ge \frac{150r}{\delta^2 p}$, using the fact that $\delta$ is sufficiently small. Since $|\overline{S}| \ge \frac{n}{3}$, we have $p|Z\setminus L||\overline{S}| > p\cdot \frac{150r}{\delta^2 p} \cdot \frac{n}{3} \ge \frac{100n}{\delta^2}$. Thus by \ref{item:property 2}, we have
\begin{equation}\label{eqn:bound_ZL_Sbar}
e_G(Z\setminus L,\overline{S}) \le (1+\delta)p|Z\setminus L||\overline{S}|.
\end{equation}
Substituting~\eqref{eqn:bound_Z_V} and~\eqref{eqn:bound_ZL_Sbar} into~\eqref{eq:T minus A to S lower bound}, we obtain
\begin{align}
    e_G(Z,S_k) &> \alpha(1-\delta)p|Z|n  - \left(\tfrac{1}{2}+\delta\right)(1+\delta)p|Z\setminus L||\overline{S}| - \frac{p|L\cap Z||\overline{S}|}{5000}\nonumber \\
    &= p|Z|\left((1-\delta)\alpha n - \left(\tfrac{1}{2}+\delta\right)(1+\delta)|\overline{S}| \right) + p|L\cap Z||\overline{S}|\left(\left(\tfrac{1}{2}+\delta\right)(1+\delta) - \tfrac{1}{5000}\right) \nonumber \\
    &\ge p|Z|\left(\alpha n - \tfrac{1}{2}|\overline{S}| -3\delta n\right) + 0 \nonumber \\
    &= p|Z|\left(\alpha n - \tfrac{1}{2}n + \tfrac{1}{2}|S| -3\delta n\right).
 \label{eq:lower bound Z and S cup Z}
\end{align}
Note that $|S_k|\le |S'| + \frac{200r}{\delta^2 p} = |S| + 3|K| + \frac{200r}{\delta^2 p} \le |S| + \delta n$, using \ref{item:property K small} and \ref{item:parameter hierarchy}.
Note also that $p|Z||S_k|\ge p|Z||S| \ge  p\cdot \frac{200r}{\delta^2 p} \cdot |S| \ge \frac{100n}{\delta^2}$, where the last inequality uses the fact that $|S| \ge \frac{(1-\eps)n}{r}\ge \frac{n}{2r}$. Thus, using \ref{item:property 2}, we have 
\begin{equation}\label{eq:upper bound Z and S cup Z}
    e_G(Z,S_k) \le (1+\delta)p|Z||S_k|\le (1+\delta)p|Z|\left(|S| + \delta n\right) 
    \le p|Z|\left(|S| + 3\delta n\right).
\end{equation}
Combining \eqref{eq:lower bound Z and S cup Z} and \eqref{eq:upper bound Z and S cup Z}, we obtain 
\[
|S| + 3\delta n \;\; \ge \;\; \alpha n - \tfrac{1}{2}n + \tfrac{1}{2}|S| -3\delta n,
\]
which, upon rearranging, gives
\[
|S| \;\; \ge \;\; (2\alpha - 1) n -12\delta n,
\]
giving a contradiction to $|S| \le (1-\eps)(2\alpha - 1)n$ because $\delta$ is sufficiently small relative to $\eps$ and $r$. This finishes the proof of \ref{item:size_SminusA}.

Finally, we will prove \ref{item:Hall's condition}. We order the vertices of $T\setminus A$ as follows. For convenience, we define $T' \coloneqq (T \setminus A) \setminus K$. Let $y_1$ be a vertex of $T'$ such that $d_H(y_1, A\setminus K')$ is maximal. Suppose we have defined $y_1,y_2,\dots y_{i}$ for $1 \le i < |T'|$. We then define $y_{i+1}$ to be a vertex in $T'\setminus\{y_1,\ldots,y_{i}\}$ such that $d_H\big(y_{i+1}, (A\cup\{y_1,\ldots,y_{i}\})\setminus K'\big)$ is maximal. Finally, we let the final vertices in the ordering of $T\setminus A$ be the vertices in $(T \setminus A)\cap K$ in any order. We will show that \ref{item:Hall's condition} holds with respect to this ordering of the vertices in $T\setminus A$. Define $A_0 \coloneqq A$, and for every $i \in [|T\setminus A|]$, define $A_i \coloneqq A\cup\{y_1,\ldots,y_i\}$.

Now fix a set $Z \subseteq T\setminus A$. Our goal is therefore to show that $\left|\bigcup_{y_i \in Z} N_H(y_i) \cap A_{i-1} \right| \ge 2|Z|$.
We will divide $Z$ according to whether $y_i$ is in $K$ or not. By \ref{item:property of K'}, every $y \in K$ has two neighbours in $K'$ and these neighbours are distinct for distinct $y \in K$. Since $K' \subseteq T \setminus K \subseteq A_{i-1}$ for $i > |(T\setminus A)\setminus K|$, we therefore have 
\begin{equation}\label{eqn:Z_cap_K}
\left|\left(\bigcup_{y_i \in Z \cap K} N_H(y_i) \cap A_{i-1} \right)\cap K'\right| \ge 2|Z \cap K|.
\end{equation}

Next, we consider $Z \setminus K$. We first show that $d_H(y,T\setminus K') > \frac{np}{12000r}$ for all $y \in T'$ by splitting into cases according to whether $y$ is in $L$, recalling that $L = \left\{y \in V: \left(\tfrac{1}{2}+\delta\right)d_G(y,\overline{S})< \frac{p|\overline{S}|}{5000}\right\}$.

Consider $y \in T' \cap L $, so $(\tfrac{1}{2} + \delta)d_G(y,\overline{S}) < \frac{p|\overline{S}|}{5000}$ and $d_H(y) > \frac{pn}{5000}$ (the latter one follows by \ref{item:property K small}). For such $y$, using \eqref{eq:T minus A has high edges in T} and the fact that the fact that $|S| \ge \frac{(1-\eps)n}{r}\ge \frac{n}{2r}$, we have 
\begin{equation}
d_H(y,T) > d_H(y) - \frac{p|\overline{S}|}{5000} 
> \frac{pn}{5000} - \frac{p(n-|S|)}{5000}
\ge \frac{np}{10000r}.
\label{eq:lb_degree_L_to_T}
\end{equation}

Consider $y \in T'\setminus L $, so $(\tfrac{1}{2} + \delta)d_G(y,\overline{S}) \ge \frac{p|\overline{S}|}{5000}$ and $d_G(y)\ge d_H(y) > \frac{pn}{5000}$. For such $y$, using \eqref{eq:T minus A has high edges in T} and \ref{item:parameter hierarchy}, we have 
\begin{equation}
d_H(y,T) > \alpha d_G(y) - (\tfrac{1}{2} + \delta)d_G(y,\overline{S}) 
\ge \left(\alpha - \tfrac{1}{2} - \delta\right)d_G(y)
\ge \left(\frac{1}{2r} - \delta\right)d_G(y) \ge \frac{np}{11000r}.
\label{eq:lb_degree_T_minus_L_to_T}
\end{equation}
So by~\eqref{eq:lb_degree_L_to_T} and~\eqref{eq:lb_degree_T_minus_L_to_T}, we conclude that $d_H(y,T) > \frac{np}{11000r}$ for all $y \in T'$. Moreover, by \ref{item:property of K'}, we know that each $y \in V(H)$ has at most two neighbours in $K'$ and so for all $y \in T'$ we have
\begin{equation}\label{eq: lower bound to T minus K'}
    d_H(y,T\setminus K') > \frac{np}{11000r} - 2 > \frac{np}{12000r},
\end{equation}
where we use $np \ge \frac{\log{n}}{2}$ and the fact that $n$ is sufficiently large.

We will now show that for all $i \in [|T'|]$, we have $d_H\big(y_{i},A_{i-1}\setminus K' \big) \ge \frac{np}{24000r}$.
Suppose for contradiction that for some $i \in [|T'|]$, we have $d_H\big(y_{i},A_{i-1}\setminus K'\big) < \frac{np}{24000r}$. Then letting $B \coloneqq T' \setminus \{y_1,\ldots,y_{i-1}\}$, we have $d_H(y, A_{i-1}\setminus K') < \frac{np}{24000r}$ for every $y \in B$. Since $B \subseteq T'$, by~\eqref{eq: lower bound to T minus K'} we have $d_H(y,T\setminus K') > \frac{np}{12000r}$ for all $y \in B$. 
Therefore, for all $y \in B$, we have 
\begin{align*}
    d_H(y,B) = d_H(y,T'\setminus\{y_1,\ldots,y_{i-1}\}) 
    &= d_H(y,T\setminus K') - d_H(y, A_{i-1}\setminus K') \\
    &> \frac{np}{12000r} - \frac{np}{24000r} = \frac{np}{24000r}
\end{align*}
 From this, we conclude that 
 \begin{equation}\label{eqn:lower_bound_B}
 e_G(B) \ge e_H(B) > \frac{np}{48000r}|B|
 \end{equation}
 By~\ref{item:size_SminusA}, we know $|B|\le \frac{200r}{\delta^2 p}$. Thus, using \ref{item:property 3}, we have 
\begin{equation}\label{eqn:upper_bound_B}
 e_G(B)\le |B|(np)^{1/8}.
 \end{equation}
Since $n$ is sufficiently large and $np \ge \frac{\log{n}}{2}$, \eqref{eqn:lower_bound_B} and \eqref{eqn:upper_bound_B} give a contradiction. Therefore we must have for all $i \in [|T'|]$,
\begin{equation}\label{eq:lower_bound_degree_y_i}
     d_H\big(y_{i}, A_{i-1}\setminus K' \big) \ge \frac{np}{24000r}.
\end{equation}

We can now show that $\left|\left(\bigcup_{y_i \in Z \setminus K} N_H(y_i) \cap A_{i-1}\right)\setminus K'\right| \ge 2|Z \setminus K|$. Suppose for contradiction that this does not hold, i.e.~that $|\bigcup_{y_i \in Z \setminus K}  N_H(y_i) \cap (A_{i-1}\setminus K')| < 2|Z \setminus K|$. We have by~\eqref{eq:lower_bound_degree_y_i} that
\begin{align}\label{eq:lower bound on final contradiction}
e_G\left((Z \setminus K) \cup \left(\bigcup_{y_i \in Z \setminus K} N_H(y_i) \cap (A_{i-1}\setminus K')\right)\right) &\ge e_H\left((Z \setminus K) \cup \left(\bigcup_{y_i \in Z \setminus K} N_H(y_i) \cap (A_{i-1}\setminus K')\right)\right) \nonumber \\
&\ge \frac{1}{2} \sum_{y_i \in Z \setminus K}d_H\big(y_{i}, A_{i-1}\setminus K' \big) \nonumber \\
&\ge |Z \setminus K|\frac{np}{48000r}.
\end{align}
Since $Z \setminus K \subseteq T\setminus A$, we know by~\ref{item:size_SminusA} that $|Z \setminus K| \le \frac{200r}{\delta^2 p}$. Thus, by assumption, 
\[
\left|(Z \setminus K) \cup \left(\bigcup_{y_i \in Z \setminus K} N_H(y_i)\cap (A_{i-1}\setminus K')\right)\right| < |Z \setminus K| + 2|Z \setminus K| \le \frac{600r}{\delta^2 p}.
\]
Therefore, using~\ref{item:property 3} we have
\begin{align}\label{eq:upper bound on final contradiction}
e_G\left((Z \setminus K) \cup \left(\bigcup_{y_i \in Z \setminus K} N_H(y_i) \cap (A_{i-1}\setminus K')\right)\right) 
&\le \left|(Z \setminus K) \cup \left(\bigcup_{y_i \in Z \setminus K} N_H(y_i) \cap (A_{i-1}\setminus K')\right)\right|(np)^{1/8} \nonumber \\
&\le 3|Z \setminus K|(np)^{1/8}.
\end{align}
Now, \eqref{eq:lower bound on final contradiction} and \eqref{eq:upper bound on final contradiction} contradict each other since $n$ is sufficiently large and $np \ge \frac{\log{n}}{2}$. Therefore, our assumption must have been false, and we must, in fac,t have 
\[
    \left|\left(\bigcup_{y_i \in Z \setminus K } N_H(y_i) \cap A_{i-1}\right)\setminus K'\right| \ge 2|Z \setminus K|.
\]
Combining this with \eqref{eqn:Z_cap_K}, we can now conclude that
 \begin{align*}
 \left|\bigcup_{y_i \in Z} N_H(y_i) \cap A_{i-1} \right|  
 &= \left|\bigcup_{y_i \in Z \setminus K } (N_H(y_i) \cap A_{i-1})  \cup \bigcup_{y_i \in Z \cap K } (N_H(y_i) \cap A_{i-1})\right|
 \\
 &\ge \left|\Bigg(\bigcup_{y_i \in Z \setminus K } N_H(y_i) \cap A_{i-1}\Bigg)\setminus K'\right| +  \left|\Bigg(\bigcup_{y_i \in Z \cap K } N_H(y_i) \cap A_{i-1}\Bigg) \cap K'\right| \\
 &\ge 2|Z \setminus K| + 2|Z \cap K| = 2|Z|.
 \end{align*}
This establishes \ref{item:Hall's condition}, completing the proof of \Cref{thm:clean-up_main}.
\end{proof}

\begin{proof}[\textbf{Proof of \Cref{lem:Hall_to_extend_paths}}]
Construct an auxiliary bipartite graph $B$ with the vertex classes $X$ and $Y$, where part $X$ has vertices labelled by the vertices of $T$ and part $Y$ has vertices labelled by the vertices of $T\setminus A$. We fix an ordering $y_1,y_2,\ldots $ on the vertices of $T \setminus A$ satisfying \ref{item:Hall's condition}. Put an edge in $B$ between a vertex in $X$ labelled by $x$ and a vertex in $Y$ labelled by $y$ if both of the following conditions hold:
\begin{itemize}
    \item $x$ and $y$ are adjacent, and
    \item either $x \in A$, or if $x \in T \setminus A$ then $x$ precedes $y$ in the ordering on $T \setminus A$.
\end{itemize}
By~\ref{item:Hall's condition}, we see that for all $Z \subseteq Y$ we have $|N_B(Z)| = |\bigcup_{y_i \in Z} (N_H(y_i) \cap (A \cup \{y_1,y_2,\ldots,y_{i-1}\}))| \ge 2|Z|$ and thus Hall's condition holds. In particular, using \Cref{thm:Hall_variant} we find two matchings $\mathcal{M}_1, \mathcal{M}_2$ in $B$ such that $\mathcal{M}_1, \mathcal{M}_2$ cover $Y$ and the sets $V(\mathcal{M}_1)\cap X$ and $V(\mathcal{M}_2)\cap X$ are disjoint. The matchings $\mathcal{M}_1, \mathcal{M}_2$ correspond to two injective functions $\phi_1, \phi_2: T\setminus A \rightarrow T$, where $x$ and $\phi_i(x)$ are adjacent in $\mathcal{M}_i$ for every $i\in [2]$ and $x\in T\setminus A$, and the images of $\phi_1$ and $\phi_2$ are disjoint.

We now consider the directed graph $\mathcal{D}$ on vertex set $T$ where there is a directed edge from $x$ to $\phi_1(x)$ and from  $x$ to $\phi_2(x)$ for all $x \in T \setminus A$. Note that if the directions are ignored, $\mathcal{D}$ is a subgraph of $H$. Each $x \in T\setminus A$ has exactly two out-neighbours in $\mathcal{D}$, and each out-neighbour is either in $A$ or earlier than $x$ in the ordering on $T \setminus A$. Moreover, as the images of the injective functions $\phi_1$ and $\phi_2$ are disjoint, each vertex in $T$ has at most one in-neighbour. In particular, this means that $\mathcal{D}$ contains no cycles (directed or otherwise). More precisely, letting $O$ be the set of vertices in $T \setminus A$ that have no in-neighbour, $\mathcal{D}$ is a disjoint union of directed binary trees where $O$ is the set of root vertices and all edges in a tree are directed away from the root, ending in leaves in $A$.

We now inductively define a collection of paths which will form a linear forest $\mathcal{P}_{\phi}$ covering $T\setminus A$. We will construct a sequence of sets $T\setminus A = V_0 \supsetneq V_1 \supsetneq V_2 \supsetneq \ldots$ and a sequence of paths $P_1,P_2,P_3,\ldots$ such that for every $i\ge 0$, the set $V_i$ has the property that 
\begin{equation}\label{prop:all out neighbours}
\text{for every $x\in V_i$, we have $\phi_1(x), \phi_2(x)\in V_i$.}
\end{equation}
First, set $V_0 = T \setminus A$. Clearly, $V_0$ satisfies the property \eqref{prop:all out neighbours}. Now, let $i \ge 0$ and we have a set $V_i$ satisfying \eqref{prop:all out neighbours}. If $V_i$ is empty, then end the process. Otherwise, we do the following to obtain $V_{i+1}$ and $P_i$.  Take $x \in V_i$ such that $x$ has no in-neighbours in $\mathcal{D}[V_i]$. Let $r \ge 1$ be the smallest integer such that $\phi_1^r(x) \in A$ and let $s\ge 1$ be the smallest integer such that $\phi_2^s(x) \in A$ (these exist by the property \eqref{prop:all out neighbours} of $V_i$). (Here, $\phi_i^t$ denotes the $t$ times composition of the map $\phi_i$.) Let  $P_i$ be the path on edges $\{\phi_1^i(x)\phi_1^{i+1}(x): 0\le i < r\}\cup \{\phi_2^i(x)\phi_2^{i+1}(x) : 0\le i < s\} $. Let $V_{i+1} = V_i\setminus V(P_i)$. Now, continue with the $(i+1)$-st step, noting that $V_{i+1}$ satisfies the property \eqref{prop:all out neighbours}.

Note that the paths $P_1,P_2,P_3,\ldots$ created by this process are pairwise vertex-disjoint. Let $\mathcal{P}_{\phi}$ be the linear forest containing every path $P_i$. Note that by construction $\mathcal{P}_{\phi}$ covers $T\setminus A$. Also, for each path $P_i$, the endpoints lie in $A$ and all other vertices lie in $T\setminus A$. Let $A' = A\cap (\bigcup_{i\ge 0}V(P_i))$ be the set of all endpoints of the paths in $\mathcal{P}_{\phi}$. Note that $|A'| \le 2|\mathcal{P}_\phi| \le 2|T\setminus A|$. 

We define $\mathcal{P}_2$ to be the linear forest obtained from taking the union of $\mathcal{P}_{\phi}$ and the subgraph of $\mathcal{P}_1$ induced by $A \setminus A'$. Clearly, $\mathcal{P}_2$ spans $T$.
\ref{eq:endpoints_in_R} follows from the construction.
The number of edges of  $\mathcal{P}_1$ that are not used in $\mathcal{P}_2$ is at most $2|A'| \le 4 |T \setminus A|$. This proves \ref{item:few edges left}. Denote the number of paths in a linear forest $\mathcal{P}$ by $\#(\mathcal{P})$. Then, $\#(\mathcal{P}_2) \le \#(\mathcal{P}_1)+|A'| + \#(\mathcal{P}_{\phi}) \le \#(\mathcal{P}_1) + 3|T\setminus A|$, showing \ref{item:few paths added}. This finishes the proof of \Cref{lem:Hall_to_extend_paths}.
\end{proof}

\bigskip \section{Resilience of Hamiltonicity in random graphs under inclusion of a matching \texorpdfstring{$\mathcal{M}$}{M}}\label{sec:resilience of Hamiltonicity}

In this section, we execute Step~4 as described in \Cref{sec:proof ideas}. Going into this step, we have a linear-sized set $U$ such that the graph $H[U]$ has high minimum degree (as in \ref{eq:degree_in_U}). We also have an `almost monochromatic' linear forest ($\mathcal{P}^*$ in \cref{thm:clean-up_main}) covering $V(H)\setminus U$ such that the endpoints of the paths lie in $U$. Our final goal is to patch these paths together to a Hamilton cycle in $H$. To facilitate that, we represent by a matching $\mathcal{M}$ the set of all edges between two endpoints of a path in the linear forest. The main result of this section is the following (where the set~$V$ should be thought of as the set~$U$ from the above discussion).
\begin{theorem}
  \label{thm:main result of 7th section}
  Let $\eps, K >0$. If $p \ge \frac{\log n}{2n}$, then the random graph $G \sim G(n,p)$ satisfies the following with probability $1-o(n^{-3})$. Let $V \subseteq V(G)$ be a vertex set of size at least $\epsilon n$ and $H$ be a spanning subgraph of $G[V]$ satisfying $d_H(v)\ge \max\set{\frac{p|V|}{5000}, \left(\frac{1}{2} + \epsilon\right)d_{G[V]}(v)}$ for every $v\in V$. Let $\mathcal{M}$ be a matching on $V$ such that $|V(\mathcal{M})|\le \frac{K}{p}$ and every vertex $v\in V$ satisfies $d_H(v,V(\mathcal{M}))\le (np)^{1/5}$. Then there is a Hamilton cycle in $H + \mathcal{M}$ that contains all edges of $\mathcal{M}$.
\end{theorem}
We will prove the above result by splitting it into two parts (i.e., \Cref{thm:expander,thm:join_paths_to_cycle_main}), which will then be combined to conclude the proof at the end of this section.

\subsection{\texorpdfstring{$\mathcal{M}$-respecting $2$-expander in subgraphs of random graphs}{M-respecting 2-expander in subgraphs of random graphs}}
\label{sec:expander}

We start with a key definition in our argument, as briefly mentioned in the proof sketch in Section~\ref{sec:proof ideas}.
\begin{definition}
Let $H$ be a graph on a vertex set $V$, and let $\mathcal{M}$ be a matching on $V$. We say that $H$ is an \defin{$\mathcal{M}$-respecting $2$-expander} if $H$ is connected, and for any vertex set $U \subseteq V$ with $|U| \leq \frac{|V|}{8}$, we have $|N_H(U) \setminus V(\mathcal{M})| \geq 2|U|$.
\end{definition}
The main result of this subsection is the following.
\begin{theorem} \label{thm:expander}
  For every $\eps, K >0$ there exists $C > 0$ such that, if $p \ge \frac{C}{n}$, then the random graph $G \sim G(n,p)$ satisfies the following with probability $1-o(n^{-3})$. Let $V \subseteq V(G)$ be a vertex set of size at least $\epsilon n$ and $H$ be a spanning subgraph of $G[V]$ satisfying $d_H(v)\ge \max\set{\frac{p|V|}{5000}, \left(\frac{1}{2} + \epsilon\right)d_{G[V]}(v)}$ for every $v\in V$. Let $\mathcal{M}$ be a matching on $V$ such that $|V(\mathcal{M})|\le \frac{K}{p}$ and every vertex $v\in V$ satisfies $d_H(v,V(\mathcal{M}))\le (np)^{1/5}$. Then $H$ contains a spanning $\mathcal{M}$-respecting $2$-expander with at most $\eps p n^2$ edges.
\end{theorem}

\begin{proof}
Without loss of generality, we assume $\eps\le \frac{1}{5000}$. We choose $C$ to be sufficiently large with respect to $\eps$ and $K$. We also assume $n$ to be sufficiently large relative to all these constants to support our arguments whenever needed. We use the following standard definition in our proof.
\begin{definition}
A graph $G$ is a \defin{$(t,k)$-expander} if $|N(U)|\ge t|U|$ for every subset $U\subseteq V(G)$ with $|U|\le k$.
\end{definition}
\addtocounter{theorem}{-1}
\begin{claim}\label{clm:tk-expander}
The random graph $G$ satisfies the following with probability $1-o(n^{-3})$. Every subgraph $H$ of $G$ satisfying the assumptions in \Cref{thm:expander} contains a spanning $\left((np)^{1/4}, \frac{K}{p}\right)$-expander with at most $\frac{\eps p n^2}{2}$ edges. 
\end{claim}
\addtocounter{theorem}{1}
\begin{claimproof}
Applying~\Cref{lem:small-sets-cannot-contain-too-many-edges} with $a = b = \frac{2}{3}$, we have that with probability $1-o(n^{-3})$,
\begin{equation}\label{eq:smallsetssparse}
\text{all sets $S$ of size at most $\frac{2Kn}{(np)^{3/4}}$ satisfies $e_G(S)\le |S|(np)^{2/3}$ edges.}
\end{equation}
It now suffices to prove \Cref{clm:tk-expander} assuming that $G$ satisfies \Cref{eq:smallsetssparse}. For that, fix a subgraph $H$ of $G$ that satisfies the assumptions in \Cref{thm:expander}. Consider a minimal (all proper subgraphs of $H_1$ have lower minimum degree) spanning subgraph $H_1$ of $H$ whose minimum degree is $\lfloor\frac{\eps |V|p}{2}\rfloor \le \delta(H)$. By minimality, every edge in $H_1$ has an endpoint of degree $\lfloor\frac{\eps |V|p}{2}\rfloor$ and so the number of edges in $H_1$ is at most $|V|\lfloor\frac{\eps |V|p}{2}\rfloor \le \frac{\eps p n^2}{2}$. Suppose for contradiction that there is $A\subseteq V(H_1)$ with $|A|\le \frac{K}{p}$ such that $|N_{H_1}(A)|< |A|(np)^{1/4}$. This gives a set $S = A\cup N_{H_1}(A)$ of size at most $2|A|(np)^{1/4} \le \frac{2Kn}{(np)^{3/4}}$ with at least $|A|\cdot \lfloor\frac{\eps |V|p}{2}\rfloor\cdot \frac{1}{2} \ge |S| \cdot \frac{\eps^2 (np)^{3/4}}{8}$ edges in $H_1$, and thus also in $G$. However, this contradicts \eqref{eq:smallsetssparse}, completing the proof of \Cref{clm:tk-expander}.
\end{claimproof}

We are now back to proving \Cref{thm:expander}.
By \Cref{lem:large-sets-pseudorandom-property-Gnp}, with probability $1-o(n^{-3})$, we have that if $A,B\subseteq V(G)$ satisfy $|A|\ge \frac{K}{p}$ and $|B|\ge \frac{\eps n}{2}$, then 
\begin{equation}
\left(1-\frac{\eps}{2}\right)p|A||B|\le e_G(A,B)\le \left(1+\frac{\eps}{2}\right)p|A||B|
\label{eq:largesetsdense}
\end{equation}
We will now prove \Cref{thm:expander} assuming that $G$ satisfies the above property and \Cref{clm:tk-expander}. Fix a subgraph $H$ of $G$ satisfying the assumptions in \Cref{thm:expander}. Apply \Cref{clm:tk-expander} to find $H_1$, a subgraph of $H$, that is a spanning $\left((np)^{1/4}, \frac{K}{p}\right)$-expander with at most $\frac{\eps p n^2}{2}$ edges. Using \eqref{eq:largesetsdense}, for every disjoint $A,B\subseteq V$ with $|A|\ge \frac{K}{p}$ and $|B|\ge \frac{|V|}{2}$, we have
\begin{align*}
    e_H(A,B)
    = e_H(A,V) - e_H(A, V \setminus B) 
    &\ge (\tfrac{1}{2} + \eps)e_G(A,V) - e_G(A,V \setminus B) \\
    &= e_G(A,B) - \left(\tfrac{1}{2}-\epsilon\right)e_G(A,V) \\
    &\ge \left(1 - \tfrac{\epsilon}{2}\right)p|A||B| - \left(\tfrac{1}{2} - \epsilon\right) \left(1 + \tfrac{\epsilon}{2}\right)p|A||V| \\
    &\ge \left(1 - \tfrac{\epsilon}{2}\right)p|A|\frac{|V|}{2} - \left(\tfrac{1}{2} - \epsilon\right) \left(1 + \tfrac{\epsilon}{2}\right)p|A||V| \\
    &\ge \frac{\eps p|A||V|}{2}\ge \frac{K\eps^2 n}{2}.
\end{align*}
Let $H_2$ be a random subgraph of $H$ where each edge is included independently with probability $\frac{\eps}{3}$ (i.e., $H_2 \sim H(\frac{\eps}{3})$). Thus, the probability that there exist some disjoint $A,B\subseteq V$ with $|A|\ge \frac{K}{p}$, $|B|\ge \frac{|V|}{2}$, and $e_{H_2}(A,B)=0$ is at most 
\[
(2^n)^2 \left(1-\frac{\eps}{3}\right)^{K\eps^2 n/2} \le 4^n e^{-K\eps^3 n/6} = o(1).
\]
By \eqref{eq:largesetsdense}, we have $e(G)\le p n^2$ and thus by Chernoff's bound, \whp $e(H_2)\le \frac{\eps p n^2}{2}$. Thus, we can take a subgraph $H_2$ of $H$ such that 
\begin{equation}\label{eq:property of H2}
\text{$e(H_2)\le \frac{\eps p n^2}{2}$ and $e_{H_2}(A,B)>0$ for every disjoint $A,B\subseteq V$ with $|A|\ge \frac{K}{p}$, $|B|\ge \frac{|V|}{2}$.}
\end{equation}

Let $H_0 = H_1 + H_2$. We aim to show that $H_0$ is a desired expander. Clearly, $e(H_0)\le \eps pn^2$. Since $H_1$ is an $\left((np)^{1/4}, \frac{K}{p}\right)$-expander, for each $A\subseteq V$ with $|A|\le \frac{K}{p}$, we have 
\[
|N_{H_0}(A)\setminus V(\mathcal{M})|\ge |N_{H_1}(A)| - |N_{H}(A)\cap V(\mathcal{M})|\ge |A|(np)^{1/4} - |A|(np)^{1/5} \ge 2|A|.
\]
Next, for each $A\subseteq V$ with $\frac{K}{p} \le |A|\le \frac{|V|}{8}$, there are no edges in $H_2$ between $A$ and $V\setminus (A\cup N_{H_2}(A))$, and thus by~\eqref{eq:property of H2}, we must have $|A\cup N_{H_2}(A)|\ge \frac{|V|}{2}$. Hence, $|N_{H_2}(A)|\ge \frac{|V|}{2} - |A|\ge 3|A|$, where the last inequality uses $|A|\le \frac{|V|}{8}$. Consequently, $|N_{H_0}(A)\setminus V(\mathcal{M})|\ge |N_{H_2}(A)| - |V(\mathcal{M})|\ge 3|A| - \frac{K}{p}\ge 2|A|$, where we have used $\frac{K}{p} \le |A|$. 

The only thing remaining is to show that $H_0$ is connected. Indeed, suppose for a contradiction that $H_0$ has a connected component with vertex set $A$ such that $|A|\le \frac{|V|}{2}$. Then, since $N_{H_0}(A) = \emptyset$ and $H_1$ is an $\left((np)^{1/4}, \frac{K}{p} \right)$-expander, we have $|A|> \frac{K}{p}$. This, together with the fact that there are no edges in $H_2$ between $A$ and $V\setminus A$, gives a contradiction to~\eqref{eq:property of H2}.
\end{proof}

\subsection{\texorpdfstring{Utilising $\mathcal{M}$-respecting $2$-expander to construct Hamilton cycle containing $\mathcal{M}$}{Utilising M-respecting 2-expander to construct Hamilton cycle containing M}}
\label{sec:join_paths_to_cycle}

The main result of this subsection is the following modified version of Theorem~1.6 in \cite{Montgomery2019resilient}.
\begin{theorem}
  \label{thm:join_paths_to_cycle_main}
  For any $\eps, K >0$ there exist $\delta, C > 0$ such that, if $p \geq \frac{C}{n}$, then the random graph $G \sim G(n,p)$ satisfies the following with probability $1 - o(n^{-3})$.\ Let $V \subseteq V(G)$ be a vertex set of size at least $\epsilon n$ and $H$ be a $\paren{\frac{1}{2} + \eps}$-residual spanning subgraph of $G[V]$. Let $\mathcal{M}$ be a matching on $V$ of size at most $\frac{K n}{\log{n}}$ such that $H$ contains a spanning $\mathcal{M}$-respecting $2$-expander with at most $\delta p n^2$ edges. Then there is a Hamilton cycle in $H + \mathcal{M}$ that contains all edges of $\mathcal{M}$.
\end{theorem}

In order to prove this theorem, we will first need to prove several lemmas. The structure of which results rely on each other is shown in \Cref{fig:lemmas_find_hamilton_cycle} below.

\begin{figure}[h] 
  \centering
  \begin{tikzpicture}[
    node distance = 5mm and -10mm,
    start chain = going below,
    block/.style = {draw, rounded corners, thick, align=center, minimum height=2em,
      on chain},
    plate/.style={draw, rounded corners, thick, dotted, align=right, inner sep=1.5mm, label={[xshift=-1.8cm,yshift=.65cm]south east:#1}}
    ]
    \node (A) [block] { \cref{thm:join_paths_to_cycle_main}};
    \node (B) [block] [below left=of A]  {\cref{lem:find-sparse-helper-for-expander}};
    \node (C) [block] [below right=of A]  {\cref{lem:many-boosters-in-residual}};
    \node (D) [block] [below=of B] 
    {\cref{lem:find-linear-v-linear-u-many-e-boost}};
    \node (E) [block] [below=of D] 
    {\cref{lem:posa-expander}\\P\'osa rotation};

    \node[plate=Boosters, dotted, fit=(D) (B) (C)] (plate1) {};
    
    \draw[->, thick, to path={-- (\tikztotarget)}]
    (E) edge (D) (D) edge (B);
    \draw[->, thick,  to path={-- (\tikztotarget)}]
    (B) edge (A) (C) edge (A);
  \end{tikzpicture}
  \caption{Overview of the proof of~\cref{thm:join_paths_to_cycle_main}}
  \label{fig:lemmas_find_hamilton_cycle}
\end{figure}
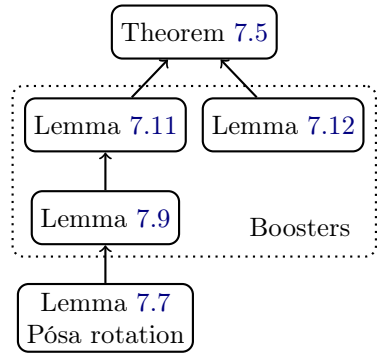

As mentioned earlier, \Cref{thm:join_paths_to_cycle_main} is a modification of \cite[Theorem~1.6]{Montgomery2019resilient}, and its proof is also very similar. The role of the matching $\mathcal{M}$ is new in \Cref{thm:join_paths_to_cycle_main} as compared to \cite[Theorem~1.6]{Montgomery2019resilient}. In order to obtain a Hamilton cycle containing all edges of $\mathcal{M}$, the main novelty we use here is to appropriately modify P\'osa's rotation and extension technique to an `$\mathcal{M}$-respecting' analogue. This `$\mathcal{M}$-respecting' concept is made precise in the following definitions.

\begin{definition}
  Let $H$ be a graph on a vertex set $V$, and let $\mathcal{M}$ be a matching on $V$.
  \begin{itemize}
  \item A path (respectively cycle) $P$ in $H + \mathcal{M}$ is an \defin{$\mathcal{M}$-respecting path} (respectively \defin{cycle}) if every pair $\set{x, y} \in E(\mathcal{M})$ either belongs to $E(P)$, or is disjoint with $V(P)$.
  \item Given an $\mathcal{M}$-respecting path $x_1, \dots, x_k$ in $H + \mathcal{M}$, an \defin{$\mathcal{M}$-respecting rotation with $x_1$ fixed} is made by breaking an edge $x_ix_{i + 1}$, for some $i\in [k-1]$, that satisfies $x_ix_k \in E(H)$ and $\set{x_ix_{i + 1}} \not\in \mathcal{M}$, and considering the new $\mathcal{M}$-respecting path $x_1, \dots, x_i, x_k, x_{k - 1}, \dots, x_{i + 1}$.
  \end{itemize}
\end{definition}

The following Lemma generalizes P\'{o}sa's rotation-expansion technique to $\mathcal{M}$-respecting paths. A path with endpoints $u$ and $v$ is often referred to as \defin{$uv$-path}. 
\begin{lemma}
  \label{lem:posa-expander}
  Let $H$ be a graph, and let $\mathcal{M}$ be a matching on $V(H)$. Assume that $H$ is
  an $\mathcal{M}$-respecting $2$-expander graph. Let $P$ be a longest $\mathcal{M}$-respecting path in $H + \mathcal{M}$ which has $v$ as an endpoint. Then, there are at least $\frac{|V(H)|}{8}$ vertices $u \in V(P)$ for which there is an $\mathcal{M}$-respecting $vu$-path in $H + \mathcal{M}$ with vertex set $V(P)$.
\end{lemma}
\begin{proof}
Let $U$ be the set of vertices $u \in V(P)$ such that one can obtain a $uv$-path by taking $P$ and iteratively performing $\mathcal{M}$-respecting rotations with $v$ fixed. In order to show that $|U| \geq \frac{|V(H)|}{8}$, it is enough to prove that $|N_H(U)| < 2|U| + |V(\mathcal{M}) \cap N_H(U)|$. Indeed, assuming this, and assuming also that $|U| < \frac{|V(H)|}{8}$, given that $H$ is an $\mathcal{M}$-respecting $2$-expander, we have $2|U| \leq |N_H(U) \setminus V(\mathcal{M})| = |N_H(U)| - |V(\mathcal{M}) \cap N_H(U)|$, a contradiction.

Observe first that $N_H(U) \subseteq V(P)$. Indeed, if there is an edge $ux \in E(H)$ with $u \in U$ and $x \not\in V(P)$, then, by defining $P_u$ to be an $\mathcal{M}$-respecting $vu$-path in $H + \mathcal{M}$ with vertex set $V(P)$, the path $P_u + ux$ is not $\mathcal{M}$-respecting, since otherwise it is a longer $\mathcal{M}$-respecting path than $P_u$. Therefore, $x$ belongs to an edge $xy$ of $\mathcal{M}$. We also have that $y \in V(P_u)$ since otherwise, $P_u + ux + xy$ is a longer $\mathcal{M}$-respecting path in $H + \mathcal{M}$ than $P_u$. However, the edge $xy \in E(\mathcal{M})$ with $x \not\in V(P_u)$ and $y \in V(P_u)$ contradicts that $P_u$ is an $\mathcal{M}$-respecting path in $H + \mathcal{M}$. Thus, $N_H(U) \subseteq V(P)$.

For each vertex $z \in V(P)$ that is not an endpoint of $P$, we denote its furthest and closest neighbours in $P$ from $v$ as $z^+$ and $z^-$ respectively. For $z = v$, or $z$ being the endpoint of $P$ other than $v$, we use $z^+$ or $z^-$ respectively to denote the neighbour of $z$ in $P$. We now make the following claim, recalling that $N_H(U) \subseteq V(P)$.
\begin{claim}\label{claim:neigh_in_U_or_I_in_M}
    For every $z \in N_H(U)$, either $\{z^-, z^+\} \cap U \neq \emptyset$, or $z\in V(\mathcal{M})$.
\end{claim}
\begin{claimproof}
Suppose that neither $z^-$ nor $z^+$ belongs to $U$. We also have that $z \not\in U$ given that $z \in N_H(U)$. Since $z \in N_H(U)$, there exists $y \in U\cap N_H(z)$ such that there is an $\mathcal{M}$-respecting $vy$-path $P_y$ in $H + \mathcal{M}$ with vertex set $V(P)$, which can be achieved by taking $P$ and iteratively performing $\mathcal{M}$-respecting rotations with $v$ fixed. Note that if an edge $ab$ is broken when performing an $\mathcal{M}$-respecting rotation with $v$ fixed, then either $a$ or $b$ is an endpoint of the resulting $\mathcal{M}$-respecting path from such rotation, and hence, either $a$ or $b$ belongs to $U$. Therefore, given that $z, z^-, z^+ \not\in U$, the edges $z^-z, zz^+$ are never broken when performing $\mathcal{M}$-respecting rotations starting from $P$ to obtain $P_y$. Thus, $z^-z, zz^+ \in E(P_y)$. Assume without loss of generality that, among $z^+$ and $z^-$, $z^+$ is the vertex furthest from $v$ in $P_y$. Clearly $z^+ \neq y$ since $z^+ \not\in U$. Observe that $zz^+ \in E(\mathcal{M})$, since otherwise we can perform an $\mathcal{M}$-respecting rotation with $v$ fixed using the edge $zy \in E(H) \setminus E(P_y)$, and breaking the edge $zz^+$ to get an $\mathcal{M}$-respecting path in $H + \mathcal{M}$ with $z^+$ as an end-vertex, which contradicts that $z^+ \not\in U$. We then have that $z\in V(\mathcal{M})$, as desired.
\end{claimproof}
  
Let $u$ be the endpoint of $P$ different from $v$. Trivially, for every vertex in $U \setminus \{u\}$, at most two of its neighbours in $P$ belong to $N_H(U)$, and at most one neighbour of $u$ in $P$ belongs to $N_H(U)$. Also, at most $|V(\mathcal{M}) \cap N_H(U)|$ vertices of $V(\mathcal{M})$ belong to $N_H(U)$. By \Cref{claim:neigh_in_U_or_I_in_M}, all the vertices of $N_H(U)$ can be obtained in one of these two ways. We then conclude that $|N_H(U)| \leq 2|U| - 1 + |V(\mathcal{M}) \cap N_H(U)|$, as desired. This finishes the proof of \Cref{lem:posa-expander}
\end{proof}

The following lemma corresponds to Lemma 3.5 of~\cite{Montgomery2019resilient}. The proof is essentially the same, though we need to modify the classical definition of `boosters' as follows and adapt the proof accordingly.
\begin{definition}
Let $H$ be a graph on a vertex set $V$, and let $\mathcal{M}$ be a matching on $V$. A set $B \subseteq \binom{V}{2}$ of pairs of vertices is called an \defin{$\mathcal{M}$-respecting-booster for $H$} if the graph $H + \mathcal{M} + B$ contains either a longer $\mathcal{M}$-respecting path than a longest $\mathcal{M}$-respecting path in $H + \mathcal{M}$, or an $\mathcal{M}$-respecting Hamilton~cycle.
\end{definition}
\begin{lemma}
  \label{lem:find-linear-v-linear-u-many-e-boost}
  For every $ \eps, K > 0$ there exist $\delta, C > 0$ such that, if $p \geq \frac{C}{n}$, then $G \sim G(n, p)$ has the following property with probability $1 - o(n^{-3})$. For every set $V \subseteq V(G)$ of size $|V| \geq \eps n$, every matching $\mathcal{M}$ on $V$ of size at most $\frac{K n}{\log{n}}$, every spanning subgraph $H_0$ of $G[V]$ that is an $\mathcal{M}$-respecting $2$-expander graph with at most $2\delta p n^2$ edges, and every $(\frac{1}{2} + \eps)$-residual spanning subgraph $H$ of $G[V]$, the following happens. For at least $\frac{|V|}{8}$ vertices $v \in V$, there is some set $U_v \subseteq V$ of size $|U_v| \geq \left(\frac{1}{2} + \frac{\eps}{8}\right)|V|$ and disjoint subset  $E_{v, u} \subseteq E(H - H_0)$, for $u\in U_v$, of size $|E_{v, u}| \geq \frac{50}{\eps \delta}$ such that, for each $u \in U_v$ and $e \in E_{v, u}$, $\{uv, e\}$ is an $\mathcal{M}$-respecting booster for $H_0$.
\end{lemma}
\begin{proof}
Let $\eps, K >0$. Choose $\delta >0$ sufficiently small relative to $\eps$, and then choose $C$ sufficiently large relative to $\delta$, $\eps$, and $K$. We will assume $n$ to be sufficiently large with respect to all these constants to support our arguments.
Let $\cH$ be the set of all pairs $(H_0, \mathcal{M})$ such that $H_0$ is a graph with $V(H_0)\subseteq V(G)$, $|V(H_0)|\ge \eps n$, and $e(H_0)\le 2\delta p n^2$, $\mathcal{M}$ is a matching on $V(H_0)$ of size at most $\frac{K n}{\log{n}}$, and $H_0$ is an $\mathcal{M}$-respecting $2$-expander. Let $(H_0, \mathcal{M}) \in \cH$, and denote $V = V(H_0)$. By~\Cref{lem:posa-expander}, there is a set $W_0\subseteq V$ of at least $\frac{|V|}{8}$ vertices that appear at the end of a longest $\mathcal{M}$-respecting path in $H_0 + \mathcal{M}$.

For each $v \in W_0$, let $F(H_0, \mathcal{M}, v)$ be the event that, for every $(\frac{1}{2} + \eps)$-residual spanning subgraph $H$ of $G[V]$, there is some set $U \subseteq V$ with $|U| \geq (\frac{1}{2} + \frac{\eps}{8})|V|$, and disjoint subset  $E_u \subseteq E(H - H_0)$, for $u \in U$, so that $|E_u| \geq \frac{50}{\eps \delta}$ and, for each $u \in U$ and $e \in E_u$, the set $\{uv, e\}$ is an $\mathcal{M}$-respecting booster for $H_0$.

\begin{claim}\label{claim:prob-not-FHv}
    For every $(H_0, \mathcal{M}) \in \cH$ and every $v \in W_0$, we have $\Prob{\overline{F(H_0, \mathcal{M}, v)}} \leq \exp\left(-\frac{\eps^5 pn^2}{10^6}\right)$.
\end{claim}
\begin{claimproof}
Let $(H_0, \mathcal{M}) \in \cH$ and $v \in W_0$ be fixed. Denote $V=V(H_0)$. Let $P \subseteq H_0$ be a longest $\mathcal{M}$-respecting path in $H_0 + \mathcal{M}$ that has $v$ as an end-vertex. We apply \Cref{lem:posa-expander} to choose a set $A \subseteq V(P)$ of size $|A| = \frac{\eps |V|}{30}$ such that for every $a\in A$, there is an $\mathcal{M}$-respecting $va$-path $P_a$ in $H_0 + \mathcal{M}$ with vertex set $V(P)$. Let $B = V \setminus A$. 
    
For $u \in B \cap V(P)$, let $X_u$ be the set of pairs $\{a, b\}$ with $a \in A$ and $b \in V(P) \setminus A$ such that we can take $P_a$ and perform an $\mathcal{M}$-respecting rotation on $H_0 + ab$ with $v$ fixed using the edge $ab$, to get an $\mathcal{M}$-respecting path in $H_0 + \mathcal{M}$ with $u$ as an end-vertex. Given a pair $\{a, b\}$, if this pair belongs to a set $X_u$, then $u$ is the neighbour of $b$ closest to $a$ in $P_a$. Therefore, for every such pair, the vertex $u$ such that $\{a, b\} \in X_u$ is uniquely determined. Hence, the sets $X_u$ are disjoint and satisfy $|X_u|\le |A|$. For a pair $\{a, b\} \in X_u$, we can also see that the set $\{uv, ab\}$ is an $\mathcal{M}$-respecting booster for $H_0$. Indeed, after performing the above rotation in the graph $H_0 + ab$, we get an $\mathcal{M}$-respecting $vu$-path $P_u$ in $H_0 + \mathcal{M}$. Then, the graph $\mathcal{C} = P_u + uv$ is an $\mathcal{M}$-respecting cycle. If $\mathcal{C}$ is not a Hamilton cycle in $H_0 + \mathcal{M} + \{uv, ab\}$, there are vertices $w \in V \setminus V(\mathcal{C})$ and $x\in V(\mathcal{C})$ such that $xw\in E(H_0)$, since $H_0$ is connected. Given that $\mathcal{M}$ is a matching, one of the neighbours of $x$ in $\mathcal{C}$, say $y$, is such that $xy \not\in E(\mathcal{M})$. Then, the path $\mathcal{C} - xy + xw$ has length strictly larger than $P_a$. If this path is not $\mathcal{M}$-respecting, there is a vertex $z \not\in V(\mathcal{C}) \cup \{w\}$ with $wz \in E(\mathcal{M})$, and the path $\mathcal{C} - xy + xw + wz$ is $\mathcal{M}$-respecting and longer than $P_a$.

For each $a \in A$, there are at least $|V(P)| - |A| - e(\mathcal{M})$ values of $b \in V(P) \setminus A$ for which $P_a$ can be rotated using $ab$ to produce an $\mathcal{M}$-respecting path. Each such rotation will give a different endpoint, at most $|A|$ of which can be in $A$. Thus, there are at least $|V(P)| - 2|A| - e(\mathcal{M})$ vertices $u \in V(P) \setminus A = B \cap V(P)$ for which $a$ is a member of a set of $X_u$. Therefore, counting over all $a \in A$, we have 
\[
\left|\cup_{u \in B \cap V(P)} X_u\right| \geq |A|\left( |V(P)| - 2|A| - e(\mathcal{M}) \right).
\]
    
For each $u \in B \setminus V(P)$, let $X_u$ be the set of pairs $\{u, a\}$ with $a \in A$. Clearly, the sets $X_u$ are disjoint from each other, as well as from each set $X_u$, $u \in B \cap V(P)$. Also, for each $u \in B \setminus V(P)$, $|X_u| = |A|$. The set $\{ua\}$ is an $\mathcal{M}$-respecting booster for $H_0$. Indeed, if the path $P_a + ua$ is not $\mathcal{M}$-respecting, then there is a vertex $z \not\in V(P_a) \cup \{u\}$ such that $uz \in E(\mathcal{M})$, and the path $P_a + ua + uw$ is an $\mathcal{M}$-respecting path longer than $P_a$. Trivially, the set $\{uv, ua\}$ is also an $\mathcal{M}$-respecting booster for $H_0$. Hence,
\begin{equation}\label{eq:X-is-a-set-of-boosters}
    \text{for every~}e \in X_u \text{,~} u\in B \text{, the set~} \set{uv, e} \text{~is an $\mathcal{M}$-respecting booster for~} H_0.
\end{equation}
We also have the following for every $u \in B$,
\begin{equation}\label{eq:x_u-at-most-A}
    |X_u| \leq |A|.
\end{equation}
Given that $|\cup_{u \in B\setminus V(P)} X_u| = |A|\left( |V| - |V(P)| \right)$, we have
\begin{align}\label{eq:bound_X_u}
    |\cup_{u \in B} X_u| &\geq |A| \left( |V| - |V(P)| \right) + |A| \left( |V(P)| - 2|A| - e(\mathcal{M}) \right) \nonumber \\
    &= |A||V| - 2|A|^2 - |A|e(\mathcal{M}).
\end{align}
Note that since $\delta$ is sufficiently small relative to $\eps$, we have $\frac{\eps |A||V|}{14} \geq 2|A|^2 + 2\delta p n^2$, and therefore
\begin{equation}\label{eq:bound_simplify_AH}
    |A||V| - 2|A|^2 - 2\delta p n^2 \geq \left( 1 - \frac{\eps}{14} \right) |A||V|.
\end{equation}
For each $u \in B$, let $Y_u \coloneqq X_u \setminus E(H_0)$. Then, since $e(\mathcal{M}) \leq \frac{Kn}{\log n} \leq \frac{\eps |V|}{20}$, we have
\begin{equation}\label{eq:bound_Y_u}
    |\cup_{u \in B} Y_u| \stackrel{\eqref{eq:bound_X_u}}{\geq} |A| |V| - 2|A|^2 - |A|e(\mathcal{M}) - 2\delta p n^2 \stackrel{\eqref{eq:bound_simplify_AH}}{\geq} \left(1 - \frac{\eps}{8}\right)|A||V|.
\end{equation}
For each $a \in A$, let $Z_a$ be the members of $\cup_{u \in B} Y_u$ that contain $a$, and let $Z_A = \cup_{u \in A} Z_a$ so that $Z_A = \cup_{u \in B} Y_u$, and hence,
\begin{equation}\label{eq:size_ZA}
    |Z_A| = |\cup_{u \in B} Y_u| \stackrel{\eqref{eq:bound_Y_u}}{\geq} \left(1 - \frac{\eps}{8}\right) |A| |V|.
\end{equation}
We now define the following events:
\begin{enumerate}[label = {$F_{\arabic*}$:}, leftmargin= \widthof{S000000}]
    \item\label{item:f1}
      $\left|Z_A \cap E(G)\right| \geq \paren{1 - \frac{\eps}{4}} p |A||V|$.
    \item\label{item:f2}
      $\sum_{a \in A} d_{G[V] - H_0}(a) \leq \paren{1 + \frac{\eps}{8}} p
      |A||V|$.
    \item\label{item:f3} Every set $U \subseteq B$ of size $|U| \leq (\frac{1}{2} + \frac{\eps}{8}) |V|$ satisfies $|\cup_{u \in U} \left( Y_u \cap E(G) \right)| < \left(\frac{1}{2} + \frac{\eps}{4} \right) p |A| |V|$.
\end{enumerate}
We next bound the probability of the above events $F_1, F_2, F_3$. For $F_1$, note that, by~\eqref{eq:size_ZA}, we have $\Expect{|Z_A \cap E(G)|} \geq (1 - \frac{\eps}{8}) p |A||V|$. Then, by Chernoff's bound and the fact that $|A||V| \geq \frac{\eps^3 n^2}{30}$, we have
\begin{equation*}
    \Prob{\overline{F_1}} \leq \exp\paren{-\frac{\eps^5 p n^2}{10^5}}.
\end{equation*}
For $F_2$, given that $|A||V| \geq \frac{\eps^3 n^2}{30}$, and using \Cref{prop:weird-chernoff} we have
\begin{equation*}
    \Prob{\overline{F_2}} \leq \Prob{\sum_{a \in A} d_{G[V] - H_0}(a) > \paren{1 + \frac{\eps}{8}} p |A| |V|} \leq \exp\paren{-\frac{\eps^5 p n^2}{10^5}}.
\end{equation*}
For $F_3$, let $U \subseteq B$ with $|U| \leq \left(\frac{1}{2} + \frac{\eps}{8}\right)|V|$. By~\eqref{eq:x_u-at-most-A}, we have that $|\cup_{u \in U} Y_u| \leq |U| |A| \leq \paren{\frac{1}{2} + \frac{\eps}{8}} |A| |V|$.
Then, by Chernoff's bound and the fact that $|A||V| \geq \frac{\eps^3 n^2}{30}$, we have
\begin{equation*}
    \Prob{\left| \cup_{u \in U} \paren{Y_u \cap E(G)} \right| \geq \paren{\frac{1}{2} + \frac{\eps}{4}} p |A| |V|} \leq \exp\paren{-\frac{\eps^5 p n^2}{10^5}}.
\end{equation*}
Taking a union bound over all sets $U \subseteq B$, and using $pn^2\ge Cn$ where $C$ is sufficiently large, we obtain
\begin{equation*}
    \Prob{\overline{F_3}} \leq 2^n \cdot 2 \exp\paren{-\frac{\eps^5 p n^2}{10^5}} \leq \exp\paren{-\frac{\eps^5 pn^2}{2 \cdot 10^5}}.
\end{equation*}
Therefore, $\Prob{F_1 \land F_2 \land F_3} \geq 1 - \exp\left(-\frac{\eps^5 p n^2}{10^6}\right)$. Suppose now that the events $F_1$, $F_2$, and $F_3$ hold. We will show that $F(H_0, \mathcal{M}, v)$ also holds, which will complete the proof of the claim.

Let $H$ be a $(\frac{1}{2} + \eps)$-residual spanning subgraph of $G[V]$. For each $u \in B$, let $E_u = Y_u \cap E(H)$. By~\eqref{eq:X-is-a-set-of-boosters}, for each $e \in E_u$, the set $\{uv, e\}$ is an $\mathcal{M}$-respecting booster for $H_0$. Let $U \subseteq B$ be the set of vertices for which $|E_u| \geq \frac{50}{\eps\delta}$. We will prove that $|U| \geq (\frac{1}{2} + \frac{\eps}{8}) |V|$. For this, it is enough to show that $|\cup_{u \in U}(Y_u \cap E(G))| \geq (\frac{1}{2} + \frac{\eps}{4}) p |A| |V|$. Then, the lower bound for $|U|$ follows immediately given that $F_3$ holds. Since $\delta$ is sufficiently small relative to $\eps$, and then $C$ is sufficiently large relative to $\delta$, we have
\begin{align*}
    |\cup_{u \in U} \paren{Y_u \cap E(G)}| \geq& |\cup_{u \in B} \paren{Y_u \cap E(G)}| - |\cup_{u \in B \setminus U} \paren{Y_u\cap E(G - H)}| - \sum_{u \in B \setminus U} |E_u| \\
    \geq& |Z_A \cap E(G)| - |Z_A \cap E(G - H)| - \frac{50}{\eps\delta} |B \setminus U|\\
    \stackrel{F_1}{\geq}& \paren{1 - \frac{\eps}{4}} p |A| |V| - \sum_{a \in A} \paren{\frac{1}{2} - \eps} d_{G[V]}(a) - \frac{50|V|}{\eps \delta} \\
    \stackrel{F_2}{\geq}& \paren{1 - \frac{\eps}{4}} p |A| |V| - \paren{\frac{1}{2} - \eps} \paren{1 + \frac{\eps}{8}} p|A| |V| - 2e(H_0) - \frac{10^4 p |A| |V|}{\eps^3 \delta C} \\
    \geq& \paren{\frac{1}{2} - \frac{\eps}{4} + \eps - \frac{\eps}{8} - \frac{120 \delta}{\eps^3} - \frac{10^4}{\eps^3 \delta C}} p |A| |V| \\
    \geq& \paren{\frac{1}{2} + \frac{\eps}{4}} p |A| |V|.
\end{align*}
This finishes the proof of \Cref{claim:prob-not-FHv}.
\end{claimproof}
  
Note that the events $F(H_0, \mathcal{M}, v)$ and $H_0 \subseteq G$ are independent. Therefore, the probability that there exists a pair $(H_0, \mathcal{M}) \in \cH$ and a vertex $v \in W_0$ such that $F(H_0, \mathcal{M}, v)$ does not hold and $H_0 \subseteq G$ is at most
  
\begin{align*}
    \sum_{\left(H_0, \mathcal{M}\right) \in \cH} & \sum_{v \in W_0} \Prob{\overline{F(H_0, \mathcal{M}, v)}} \cdot \Prob{H_0\subseteq G} \\
    &\leq \exp\left(-\frac{\eps^5 pn^2}{10^6}\right) \cdot\sum_{\left(H_0, \mathcal{M}\right) \in \cH} \sum_{v \in W_0} \Prob{H_0 \subseteq G} &\text{(by \Cref{claim:prob-not-FHv})}\\
    &\leq \exp\left(-\frac{\eps^5 pn^2}{10^6}\right) n^{\frac{2Kn}{\log{n}}} \cdot n \cdot \sum_{H_0 \subseteq K_n, e(H_0) \leq 2\delta pn^2}  \Prob{H_0 \subseteq G} \\
    &\leq \exp\left(3Kn - \left(\frac{\eps^5}{10^6} - 2\delta \log\left(\frac{\e}{2\delta}\right) \right) pn^2 \right), &\text{(by \Cref{prop:sparse-containment-Gnp})}
\end{align*}
where we have naively bounded the number of matchings $\mathcal{M}$ of size at most $\frac{Kn}{\log n}$ by $n^{\frac{2Kn}{\log{n}}}$. Thus, since $\delta$ is sufficiently small relative to $\eps$, and $C$ is sufficiently large relative to $\eps$ and $K$, the probability that $F(H_0, \mathcal{M}, v)$ holds for all pairs $(H_0, \mathcal{M}) \in \cH$ with $H_0 \subseteq G$ and all vertices $v \in W_0$ is $1 - o(n^{-3})$, as desired. This completes the proof of \Cref{lem:find-linear-v-linear-u-many-e-boost}.
\end{proof}

The next two Lemmas correspond to Lemmas 3.4 and 3.7 of \cite{Montgomery2019resilient} respectively, and their proofs are almost identical -- the only distinction being that we rely on $\mathcal{M}$-respecting boosters, rather than classically defined boosters. We include the proofs for completeness. The following definition will be helpful.
\begin{definition}\label{def:M respecting booster}
  Given two edge-disjoint graphs $H_0$ and $H'$ with $V(H_0) = V(H') \eqqcolon V$ and a matching $\mathcal{M}$ on $V$, we say $H_0$ has \defin{$\eps$-many $\mathcal{M}$-respecting boosters with help from $H'$} if there are at least $\eps|V|$ vertices $v \in V$ for which there are at least $(\frac{1}{2} + \eps)|V|$ many vertices $u \in V \setminus \{v\}$ for which there exists an edge $e \in E(H_0) \cup E(H')$ such that $\{uv, e\}$ is an $\mathcal{M}$-respecting booster for $H_0$.
\end{definition}

\begin{lemma}
  \label{lem:find-sparse-helper-for-expander}
  For each $\eps, K > 0$ there exist $\delta, C > 0$ such that, if $p \geq \frac{C}{n}$, then $G \sim G(n, p)$ has the following property with probability $1 - o(n^{-3})$. For every set $V \subseteq V(G)$ of size $|V| \geq \eps n$, every matching $\mathcal{M}$ on $V$ of size at most $\frac{K n}{\log{n}}$, every spanning subgraph $H_0$ of $G[V]$ that is an $\mathcal{M}$-respecting $2$-expander with $e(H_0) \leq 2\delta pn^2$, and every $(\frac{1}{2} + \eps)$-residual spanning subgraph $H$ of $G[V]$, there is some graph $H' \subseteq H - H_0$ with $e(H') \leq 2\delta pn^2$ such that $H_0$ has $\paren{\frac{\eps}{16}}$-many $\mathcal{M}$-respecting boosters with help from $H'$.
\end{lemma}
\begin{proof}
Let $\delta, C > 0$ be constants for which \Cref{lem:find-linear-v-linear-u-many-e-boost} holds with $\eps, K$. Then with probability $1 - o(n^{-3})$, $G$ has the property from \Cref{lem:find-linear-v-linear-u-many-e-boost} and we have $e(G)\le pn^2$ by an easy application of Chernoff's bound. Let $V$, $\mathcal{M}$, $H_0$, and $H$ be as in the statemetn of \Cref{lem:find-sparse-helper-for-expander}. By the conclusion of \Cref{lem:find-linear-v-linear-u-many-e-boost}, we find a set $W \subseteq V$ of $\frac{|V|}{8}$ vertices $v$ for which there is some set $U_v \subseteq V$ of size $|U_v| \geq \paren{\frac{1}{2} + \frac{\eps}{8}}|V|$ and disjoint subsets $E_{uv} \subseteq E(H - H_0)$, for $u \in U_v$, of size $|E_{uv}| \geq \frac{50}{\eps \delta}$ so that, for each $u \in U_v$ and $e \in E_{uv}$, $\set{uv, e}$ is an $\mathcal{M}$-respecting booster for $H_0$.

Now, consider a random subgraph $H'$ of $H - H_0$ by retaining each edge independently with probability~$\delta$. By Chernoff's bound, \whp we have $e(H') \leq 2\delta pn^2$. For every $v \in W$, let $U_v' \subseteq U_v$ be such that for every $u\in U'_v$, we have $E_{uv} \cap E(H') \neq \emptyset$, so that $\Prob{u \not\in U_v'} \leq \paren{1 - \delta}^{50 /\eps\delta} \leq \exp{\paren{-\frac{50}{\eps}}} \leq \frac{\eps}{64}$. Since the sets $E_{uv}$, for $u \in U_v$, are disjoint, the events that $u \not\in U_v'$, for $u \in U_v$, are independent. Thus, by another application of Chernoff's bound and using the fact that $|U_v| \geq \paren{\frac{1}{2} + \frac{\eps}{8}}|V|$, we have
\begin{equation*}
    \mathbb{P}\left(|U_v'| \geq \paren{\frac{1}{2} + \frac{\eps}{16}}|V|\right) \geq \mathbb{P}\left(|U_v'| \geq \paren{1 - \frac{\eps}{32}} |U_v|\right) = 1 - o(n^{-1}).
\end{equation*}
Hence, taking a union bound over vertices $v\in W$, \whp some graph $H' \subseteq H - H_0$ exists such that $e(H') \leq 2\delta pn^2$ and $|U'_v| \geq \paren{\frac{1}{2} + \frac{\eps}{16}} |V|$ for every $v \in W$. This finishes the proof of \Cref{lem:find-sparse-helper-for-expander}.
\end{proof}

\begin{lemma}
  \label{lem:many-boosters-in-residual}
  For every $\eps, K > 0$, there exist $\delta, C > 0$ such that, if $p \geq \frac{C}{n}$, then $G \sim G(n, p)$ has the following property with probability $1 - o(n^{-3})$. For every set $V \subseteq V(G)$ of size $|V|\ge \eps n$, every matching $\mathcal{M}$ on $V$ of size at most $\frac{K n}{\log{n}}$, every pair $H_0, H'$ of edge-disjoint spanning subgraphs of $G[V]$ such that $e(H_0), e(H') \leq 2\delta pn^2$ and $H_0$ has $\eps$-many $\mathcal{M}$-respecting boosters with help from $H'$, and every $\frac{1}{2}$-residual spanning subgraph $H$ of $G[V]$, there exist edges $e_1, e_2 \in E(H) \cup E(H')$ such that $\set{e_1, e_2}$ is an $\mathcal{M}$-respecting booster for $H_0$.
\end{lemma}

\begin{proof}
Choose constants $\delta, C > 0$ such that $\delta$ is sufficiently small relative to $\eps$, and $C$ is sufficiently large relative to $\eps$ and $K$. Let $\cH$ be the set of triples $\paren{H_0, H', \mathcal{M}}$ where $H_0$ and $H'$ are edge-disjoint graphs such that $V(H_0) = V(H') \subseteq V(G)$, $|V(H_0)| \geq \eps n$, $e(H_0), e(H') \leq 2\delta pn^2$, $\mathcal{M}$ is a matching on $V(H_0)$ of size at most $\frac{K n}{\log{n}}$, and $H_0$ has $\eps$-many $\mathcal{M}$-respecting boosters with help from $H'$.

For a triple $\paren{H_0, H',\mathcal{M}} \in \cH$, and a vertex $x \in V(H_0)$, let $V \coloneqq V(H_0)$, and let $V_x\subseteq V \setminus \set{x}$ be the set of vertices $v$ for which there exists an edge $e \in E(H')$ so that $\set{xv, e}$ is an $\mathcal{M}$-respecting booster for $H'$. Let $X = \set{x \in V : |V_x| \geq \paren{\frac{1}{2} + \eps} |V|}$. By our choice of $H_0$ and $H'$, and using \Cref{def:M respecting booster}, we have $|X| \geq \eps |V| \geq \eps^2 n$. Let $G \sim G(n, p)$ with $p \geq \frac{C}{n}$, and let $F(H_0, H',\mathcal{M})$ be the event that ${\sum_{x \in X} d_{G - H_0 - H'}(x, V_x) \geq (\frac{1}{2} + \frac{\eps}{4}) p |X| |V|}$. Since $\delta$ is sufficiently small relative to $\eps$, we have
\begin{align*}
    \Expect\paren{\sum_{x \in X} d_{G - H_0 - H'}(x, V_x)} 
    &\geq \paren{\frac{1}{2} + \eps} p |X||V| - 2e(H_0) - 2e(H_1) \\
    &\geq \paren{\frac{1}{2} + \frac{\eps}{2}} p |X| |V| \geq \frac{\eps^3 p n^2}{2}.
\end{align*}
Then, by \Cref{prop:weird-chernoff}, we have
\begin{equation*}
    \Prob{\overline{F\paren{H_0, H',\mathcal{M}}}} \leq 4 \exp\paren{-\paren{\frac{\eps}{4}}^2 \cdot \frac{\eps^3 p n^2}{18}} \leq \exp\paren{-\frac{\eps^5 p n^2}{10^3}}.
\end{equation*}
Observe that, since the graphs $H_0$ and $H'$ are edge-disjoint, the three events $F(H_0, H',\mathcal{M})$, $H_0 \subseteq G$, and $H' \subseteq G$ are independent. Then,
\begin{align*}
    \sum_{\mathclap{\paren{H_0, H',\mathcal{M}} \in \cH}} \Prob{\overline{F\paren{H_0, H',\mathcal{M}}} \land \set{H_0, H'\subseteq G}} 
    &\leq \sum_{\mathclap{\paren{H_0, H',\mathcal{M}} \in \cH}} \Prob{\overline{F\paren{H_0, H',\mathcal{M}}}} \Prob{H_0\subseteq G} \Prob{H' \subseteq G} \\
    &\leq n^{\frac{2Kn}{\log{n}}}\exp\paren{-\frac{\eps^5 p n^2}{10^3}} \paren{\quad\quad \underset{\mathclap{\substack{H_0 \subseteq K_n
    \\ e(H_0) \leq 2\delta p n^2}}}{\sum} \ \Prob{H_0 \subseteq G}}^2 \\
    &\leq \exp\paren{2Kn - \paren{\frac{\eps^5}{10^3} - 4\delta \log\paren{\frac{\e}{2\delta}}} p n^2},
\end{align*}
where we have bounded the number of matchings $\mathcal{M}$ of size at most $\frac{Kn}{\log{n}}$ by $n^{\frac{2Kn}{\log{n}}}$ and used \Cref{prop:sparse-containment-Gnp}. Hence, since $\delta$ is sufficiently small relative to $\eps$, and $C$ is sufficiently large with respect to $\eps$ and $K$, with probability $1 - o(n^{-3})$ the event $F(H_0, H',\mathcal{M})$ holds for every $\paren{H_0, H',\mathcal{M}} \in \cH$ with $H_0, H' \subseteq G$. By \Cref{lem:large-sets-pseudorandom-property-Gnp}, the graph $G$ satisfies the following with with probability $1-o(n^{-3})$. 
\begin{equation}\label{eq:property to prove 7.12}
  \text{If $X, U \subseteq V(G)$ are sets of size $|X| \geq \eps^2 n$ and $|U| \geq \eps n$, then $e_G(X, U) \leq \paren{1 + \frac{\eps}{3}} p |X| |U|$.} 
\end{equation}
Thus, to finish our proof, it is enough to establish the conclusion of \Cref{lem:many-boosters-in-residual} assuming the two aforementioned properties that hold with probability $1-o(n^{-3})$.

To this end, assume \eqref{eq:property to prove 7.12} holds and fix a triple $\paren{H_0, H',\mathcal{M}} \in \cH$ with $H_0, H' \subseteq G$, and assume that $F(H_0, H',\mathcal{M})$ holds. We will use the above notation for one such fixed pair. Then, we have
\begin{equation*}
    \sum_{x \in X} \paren{d_{G - H_0 - H'}\paren{x, V_x} - \frac{d_{G[V]}(x)}{2}} \geq \paren{\frac{1}{2} + \frac{\eps}{4}} p|X||V| - \frac{e_G(X, V)}{2} \stackrel{\eqref{eq:property to prove 7.12}}{>} 0.
\end{equation*}
Then, there exists a vertex $x \in X$ with $d_G(x, V_x) > \frac{d_{G[V]}(x)}{2}$. Thus, for any $\frac{1}{2}$-residual subgraph $H \subseteq G$ on $V$, we have $d_H(x, V_x) \geq d_G(x, V_x) - \frac{d_{G[V]}(x)}{2} > 0$, which implies that $x$ has a neighbour $v \in V_x$ in the graph $H$. Therefore, by the definition of the set $V_x$, there is an edge $e \in E(H')$ such that $\set{e, xv}$ is an $\mathcal{M}$-respecting booster for $H_0$. This holds for each $(H_0, H', \mathcal{M}) \in \cH$ with $H_0, H' \subseteq G$ and each $\frac{1}{2}$-residual subgraph $H \subseteq G$ on $V(H_0)$. This finishes the proof of \Cref{lem:many-boosters-in-residual}.
\end{proof}

We are now ready to prove \Cref{thm:join_paths_to_cycle_main}.

\begin{proof}[Proof of \Cref{thm:join_paths_to_cycle_main}]
Let $\delta, C > 0$ be such that \Cref{lem:find-sparse-helper-for-expander} holds for $\eps, K$, \Cref{lem:many-boosters-in-residual} holds for $\frac{\eps}{16}, K$, and such that $C \delta \geq 2$. Let $p \geq \frac{C}{n}$, and let $G \sim G(n, p)$. Then, with probability $1-o(n^{-3})$, the graph $G$ has the property in \Cref{lem:find-sparse-helper-for-expander} for $\eps, K$ and the property in \Cref{lem:many-boosters-in-residual} for $\frac{\eps}{16},K$. Let $V \subseteq V(G)$ be a vertex set of size $|V| \geq \eps n$, and let $H$ be a $(\frac{1}{2} + \eps)$-residual spanning subgraph of $G[V]$. Let $\mathcal{M}$ be a matching on $V$ of size at most $\frac{K n}{\log{n}}$ such that $H$ contains a spanning $\mathcal{M}$-respecting $2$-expander $H_0$ with $e(H_0) \leq \delta p n^2$.

For each $i = 1,\ldots,n$, we inductively find $e_{i, 1}, e_{i, 2} \in E(H)$ such that $\set{e_{i, 1}, e_{i, 2}}$ is an $\mathcal{M}$-respecting booster for $H_{i - 1}$, and we let $H_i \coloneqq H_{i - 1} + e_{i, 1} + e_{i, 2}$. Indeed, given that $e(H_{i - 1}) \leq \delta p n^2 + 2n \leq 2\delta p n^2$ (since $C\delta \ge 2$), by the property from \Cref{lem:find-sparse-helper-for-expander}, there is some graph $H' \subseteq H - H_{i - 1}$ with $e(H') \leq 2\delta pn^2$ so that $H_{i - 1}$ has $\paren{\frac{\eps}{16}}$-many $\mathcal{M}$-respecting boosters with help from $H'$. Then, by \Cref{lem:many-boosters-in-residual}, there are $e_{i, 1}, e_{i, 2} \in E(H)$ such that $\set{e_{i, 1}, e_{i, 2}}$ is an $\mathcal{M}$-respecting booster for $H_{i - 1}$, as desired.

Therefore, we have added $n$ many $\mathcal{M}$-respecting boosters to $H_0$ to obtain $H_n \subseteq H$, which contains an $\mathcal{M}$-respecting Hamilton cycle, that is, a Hamilton cycle in $H + \mathcal{M}$ that contains all edges in $\mathcal{M}$.
\end{proof}

We now combine \Cref{thm:expander,thm:join_paths_to_cycle_main} to prove \Cref{thm:main result of 7th section}, which will be readily useful in the next section to prove \Cref{thm:main,thm:hitting-time-main}.
\begin{proof}[Proof of \Cref{thm:main result of 7th section}]
Let $\eps, K >0$. Choose $\delta, C >0$ such that \Cref{thm:join_paths_to_cycle_main} holds. Let ${\eps'=\min\set{\eps, \delta}}$, and let $C'$ be such that \Cref{thm:expander} holds with $(\eps, K) = (\eps', K)$. We will assume $n$ to be sufficiently large relative to all these constants to support our arguments, and let $p\ge \frac{\log n}{2n}$. Then, the random graph $G\sim G(n,p)$ satisfies the following with probability $1-o(n^{-3})$. Let $V$, $H$, and $\mathcal{M}$ be as in the assumptions of \Cref{thm:main result of 7th section}. Since $\eps'\le \eps, \delta$, by the conclusion of \Cref{thm:expander}, the graph $H$ contains a spanning $\mathcal{M}$-respecting $2$-expander with at most $\delta pn^2$ edges. Now, this, together with the assumptions on $V$, $H$, and $\mathcal{M}$, satisfies the assumptions of \Cref{thm:join_paths_to_cycle_main}, noting that $e(\mathcal{M}) = \frac{|V(\mathcal{M})|}{2} \le \frac{K}{2p} \le \frac{Kn}{\log n}$. Thus, the conclusion of \Cref{thm:join_paths_to_cycle_main} holds, i.e., there is a Hamilton cycle in $H + \mathcal{M}$ that contains all edges of $\mathcal{M}$, as desired.
\end{proof}

\bigskip \section{Wrapping up the proof of the main results}
\label{sec:wrap up}

We can now use all the components from previous sections to prove our main results. 
\main*
\begin{proof}[Proof of~\Cref{thm:main}]
Let $r \geq 2$, $\alpha \in [\frac{1}{2} + \frac{1}{2r}, 1]$. Without loss of generality, we assume that $\eps \in (0,1]$. Let $\eps_0 = \frac{\eps}{2}$, and we choose a constant $C$ sufficiently large relative to $\eps_0$. Choose constants $\eps_1, K >0$ such that $\eps_1$ is sufficiently small relative to $\eps_0$ and $C$, and $K$ is sufficiently large with respect to $\eps_1$. We will assume $n$ to be sufficiently large with respect to all these constants to support our arguments. Let $p \geq \frac{\log{n} + \log{\log{n}} + \omega\paren{1}}{n}$, and $G \sim G(n, p)$. Then, \whp the graph $G$ satisfies the conclusions of the following statements: (i)~\Cref{lem:lower_bound_degrees}, (ii)~\Cref{cor:small_deg_vertices}, (iii)~\Cref{cor:random graph large monochromatic linear forest} with $(\eps, C) = (\eps_0, C)$, (iv)~\Cref{thm:clean-up_main} with $(\eps, C, \delta) = (\eps_0, C, \eps_1)$, and (v)~\Cref{thm:main result of 7th section} with $(\eps, K) = (\eps_1, K)$. In the remainder of the proof, we will assume these \whp properties of $G$ to establish the conclusion of \Cref{thm:main}. Doing so is clearly sufficient to prove \Cref{thm:main}.
To this end, let $H$ be an $\alpha$-residual spanning subgraph of $G$ whose edges are coloured with $r$ colours. 
Denote $k\coloneqq \min\set{(2\alpha - 1)n, \frac{2\alpha n}{r}, \frac{2n}{r+1}}$.

By the conclusion of \Cref{cor:random graph large monochromatic linear forest} with $(\eps, C) = (\eps_0, C)$, we obtain a monochromatic linear forest $\cP$ in $H$ of size $(1 - \eps_0)\min\left\{ \frac{2\alpha n}{r}, \frac{2n}{r+1} \right\}$ consisting of at most $C$ paths. Since $\min\left\{ \frac{2\alpha n}{r}, \frac{2n}{r+1} \right\}\geq k$, by possibly removing some edges from $\cP$, we may (and do) assume that the monochromatic linear forest $\cP$ spans a set $S\subseteq V(H)$ of size $|S| = (1 - \eps_0) k$ and $\cP$ still has at most $C$ paths. Thus, $e(\cP)\ge (1 - \eps_0) k - C$.

We now check the hypotheses of \Cref{thm:clean-up_main} about $H$, namely \ref{item:mindeg2} and \ref{item:property low_deg_vertices_far_apart}. To see \ref{item:mindeg2}, we use the conclusion of \Cref{lem:lower_bound_degrees} to deduce that every vertex $v\in V(H)$ satisfies $d_H(v) \geq \alpha d_G(v) \geq 2\alpha > 1$. To see \ref{item:property low_deg_vertices_far_apart}, first note that every vertex $v\in V(H)$ with $d_H(v)\le \frac{np}{5000}$ satisfies $d_G(v)\le \frac{d_H(v)}{\alpha}\le \frac{np}{2500}$. This observation, together with the conclusion of \Cref{cor:small_deg_vertices}, implies the validity of \ref{item:property low_deg_vertices_far_apart}.

Now, by the conclusion of \Cref{thm:clean-up_main} with $(\eps, C, \delta) = (\eps_0, C, \eps_1)$, we obtain a linear forest $\mathcal{P}^*$ covering a set $T \supseteq S$ such that, by letting $U = V(H) \setminus T$, the following hold.
\begin{enumerate}[leftmargin=*,label = {\bfseries E\arabic{enumi}}]
  \item $d_{H[U]}(v) \geq \max\set{\frac{p|U|}{5000}, \paren{\frac{1}{2} + \eps_1}d_{G[U]}(v)}$ for all $v \in U$.
  \label{item:E1}
  \item $|U| \geq n - |S| - \frac{300r}{\eps_1^2 p} \geq \eps_1 n$.
  \label{item:E2}
  \item All but at most $\frac{2000r}{\eps_1^2 p} \le \eps_1 n$ edges of $\mathcal{P}$ are included in $\mathcal{P}^*$.
  \label{item:E3}
  \item There are at most $\frac{2000r}{\eps_1^2 p}\le \frac{K}{2p}$ paths in $\mathcal{P}^*$.
  \label{item:E4}
  \item All paths in $\mathcal{P}^*$ have at least $3$ vertices, with endpoints lying in $U$ and all other vertices lying in $T$.
  \label{item:E5}
  \item If $W$ is the set of endpoints of the paths in $\mathcal{P}^*$, then $d_H(u, W) \leq (np)^{1 / 5}$ for all $u \in U$.
  \label{item:E6}
\end{enumerate}
  
We now define a matching $\mathcal{M}$ on $U$ by adding an edge between two vertices in $U$ if they are endpoints of a path in $\mathcal{P}^*$. Observe that, since \ref{item:E5} holds and $\mathcal{P}^*$ covers $T$, if we find a Hamilton cycle in the graph $H[U] + \mathcal{M}$ that contains all edges of $\mathcal{M}$, we can then replace each edge of the cycle that is also an edge of $\mathcal{M}$ by its corresponding monochromatic path in $\mathcal{P}^*$ to obtain a Hamilton cycle in $H$ containing all the paths in $\mathcal{P}^*$. Then, by \ref{item:E3}, all but at most $\eps_1 n$ edges of $\mathcal{P}$ are included in this Hamilton cycle, implying that the number of edges of $\mathcal{P}$ in the cycle is at least $e(\cP) - \eps_1 n \geq (1 - \eps_0)k - \eps_1 n \ge (1 - \eps) k$, and all of these edges are of the same colour. Therefore, to complete the proof of \Cref{thm:main}, it is enough to find a Hamilton cycle in the graph $H[U] + \mathcal{M}$ that contains all edges in $\mathcal{M}$.

We now wish to apply the conclusion of \Cref{thm:main result of 7th section} with $(\eps, K) = (\eps_1, K)$ on the set $U$ in place of~$V$ and the graph $H[U]$ in place of $H$. The required hypotheses follows from \ref{item:E1}, \ref{item:E2}, \ref{item:E4} with the fact that $|V(\mathcal{M})|$ is at most twice the number of paths in $\mathcal{P}^*$, and finally \ref{item:E6} with the fact that $W=V(\mathcal{M})$. Thus, the conclusion of \Cref{thm:main result of 7th section} obtains a Hamilton cycle in the graph $H[U] + \mathcal{M}$ that contains all edges of~$\mathcal{M}$, as desired. This completes the proof of \Cref{thm:main}.
\end{proof}

We now provide a proof of \Cref{thm:hitting-time-main}, which is very similar to the above proof of \Cref{thm:main}. 
\hittingtime*
\begin{proof}[Proof of \Cref{thm:hitting-time-main}]
Let $r \geq 2$, $\alpha \in [\frac{1}{2} + \frac{1}{2r}, 1]$. Without loss of generality, we assume that $\eps \in (0,1]$. Let $\eps_0 = \frac{\eps}{2}$, and we choose a constant $C$ sufficiently large relative to $\eps_0$. Choose constants $\eps_1, K >0$ such that $\eps_1$ is sufficiently small relative to $\eps_0$ and $C$, and $K$ is sufficiently large with respect to $\eps_1$. We will assume $n$ to be sufficiently large with respect to all these constants to support our arguments. By \Cref{lem:few_edges_GnM_degree_less_2}, \whp the random graph process $\set{G_{n,M}}$ satisfies: (i) for every $M < \frac{19n \log n}{40}$, we have $\delta(G_{n, M}) < 2$. Thus, it suffices to prove \Cref{thm:hitting-time-main} only in the range $M\ge \frac{19n \log n}{40}$. By \Cref{lem:small_deg_vertex_separation_in_G_{n,M}}, \whp the random graph process $\{G_{n,M}\}$ satisfies: (ii) for every $M \ge \frac{19n\log{n}}{40}$, no two vertices with degree at most $\frac{M}{1000n}$ in $G_{n,M}$ are within distance $5$ of each other.

Now, by \Cref{lemma:model-switching}, for every $M\ge \frac{19n\log n}{40}$, the graph $G\sim G_{n,M}$ satisfies, with probability $1-o(n^{-2})$, the conclusions of the following statements with $p= M/\binom{n}{2}$: (iii)~\Cref{cor:random graph large monochromatic linear forest} with $(\eps, C) = (\eps_0, C)$, (iv)~\Cref{thm:clean-up_main} with $(\eps, C, \delta) = (\eps_0, C, \eps_1)$, and (v)~\Cref{thm:main result of 7th section} with $(\eps, K) = (\eps_1, K)$. If for every $M\ge \frac{19n\log n}{40}$, we can establish the conclusion of \Cref{thm:hitting-time-main} using the conclusions of (ii)--(v), then the proof of \Cref{thm:hitting-time-main} follows from a simple union bound over all choices of $M\ge \frac{19n\log n}{40}$. (Here, the error probability $o(n^{-2})$ for each $M$ is crucial.) Thus, in the remainder of the proof, we focus on showing the conclusion of \Cref{thm:hitting-time-main} assuming the conclusions of (ii)--(v) for every $M\ge \frac{19n\log n}{40}$.
To do this, we fix $M\ge \frac{19n\log n}{40}$ and $G\sim G_{n,M}$ and $p=M/\binom{n}{2}$. We assume $\delta(G)\ge 2$ and let $H$ be an $\alpha$-residual spanning subgraph of $G$ whose edges are coloured with $r$ colours. 
Denote $k\coloneqq \min\set{(2\alpha - 1)n, \frac{2\alpha n}{r}, \frac{2n}{r+1}}$.

As before, by the conclusion of \Cref{cor:random graph large monochromatic linear forest} with $(\eps, C) = (\eps_0, C)$, we obtain a monochromatic linear forest $\cP$ in $H$ of size $(1 - \eps_0)\min\left\{ \frac{2\alpha n}{r}, \frac{2n}{r+1} \right\}$ consisting of at most $C$ paths. Since $\min\left\{ \frac{2\alpha n}{r}, \frac{2n}{r+1} \right\}\geq k$, by possibly removing some edges from $\cP$, we may (and do) assume that the monochromatic linear forest $\cP$ spans a set $S\subseteq V(H)$ of size $|S| = (1 - \eps_0) k$ and $\cP$ still has at most $C$ paths. Thus, $e(\cP)\ge (1 - \eps_0) k - C$.

We now check the hypotheses of \Cref{thm:clean-up_main} about $H$, namely \ref{item:mindeg2} and \ref{item:property low_deg_vertices_far_apart}. To see \ref{item:mindeg2}, we use our assumption that $\delta(G)\ge 2$ to deduce that every vertex $v\in V(H)$ satisfies $d_H(v) \geq \alpha d_G(v) \geq 2\alpha > 1$. To see \ref{item:property low_deg_vertices_far_apart}, first note that every vertex $v\in V(H)$ with $d_H(v)\le \frac{np}{5000}$ satisfies $d_G(v)\le \frac{d_H(v)}{\alpha}\le \frac{np}{2500}\le \frac{M}{1000n}$. This observation, together with the assumption of (ii), implies the validity of \ref{item:property low_deg_vertices_far_apart}.

While the remainder of the proof is identical to that of \Cref{thm:main}, we include the details again for completeness. By the conclusion of \Cref{thm:clean-up_main} with $(\eps, C, \delta) = (\eps_0, C, \eps_1)$, we obtain a linear forest $\mathcal{P}^*$ covering a set $T \supseteq S$ such that, by letting $U = V(H) \setminus T$, the properties in \ref{item:E1}--\ref{item:E6} hold.
  
We now define a matching $\mathcal{M}$ on $U$ by adding an edge between two vertices in $U$ if they are endpoints of a path in $\mathcal{P}^*$. Observe that, since \ref{item:E5} holds and $\mathcal{P}^*$ covers $T$, if we find a Hamilton cycle in the graph $H[U] + \mathcal{M}$ that contains all edges of $\mathcal{M}$, we can then replace each edge of the cycle that is also an edge of $\mathcal{M}$ by its corresponding monochromatic path in $\mathcal{P}^*$ to obtain a Hamilton cycle in $H$ containing all the paths in $\mathcal{P}^*$. Then, by \ref{item:E3}, all but at most $\eps_1 n$ edges of $\mathcal{P}$ are included in this Hamilton cycle, implying that the number of edges of $\mathcal{P}$ in the cycle is at least $e(\cP) - \eps_1 n \geq (1 - \eps_0)k - \eps_1 n \ge (1 - \eps) k$, and all of these edges are of the same colour. Therefore, it is enough to find a Hamilton cycle in the graph $H[U] + \mathcal{M}$ that contains all edges in $\mathcal{M}$.

We now wish to apply the conclusion of \Cref{thm:main result of 7th section} with $(\eps, K) = (\eps_1, K)$ on the set $U$ in place of~$V$ and the graph $H[U]$ in place of $H$. The required hypotheses follows from \ref{item:E1}, \ref{item:E2}, \ref{item:E4} with the fact that $|V(\mathcal{M})|$ is at most twice the number of paths in $\mathcal{P}^*$, and finally \ref{item:E6} with the fact that $W=V(\mathcal{M})$. Thus, the conclusion of \Cref{thm:main result of 7th section} obtains a Hamilton cycle in the graph $H[U] + \mathcal{M}$ that contains all edges of~$\mathcal{M}$, as desired. This completes the proof of \Cref{thm:hitting-time-main}.
\end{proof}

To finish, we prove the following corollary of \Cref{thm:main}.
\corGp*
\begin{proof}[Proof of~\Cref{cor:G(p)}]
Let $r, \alpha$ be as in \Cref{cor:G(p)}. Without loss of generality, we choose $\eps$ to be sufficiently small, and then choose $C$ sufficiently large relative to $\eps$. Let $\eps' = \eps^2$. Let $n, p$ also be as in \Cref{cor:G(p)}. Let $G$ be an $n$-vertex graph with minimum degree at least $\alpha n$. We will use the following claims, which follow from standard applications of Chernoff's bound and union bound.
\begin{claim}
    With high probability, the graph $G(p)$ satisfies $d_{G(p)}(v) \ge \left(1 - \eps' \right)\alpha np$ for all vertices $v$.
\end{claim}
\begin{claim}
    With high probability, the graph $G(n,p)$ satisfies $d_{G(n,p)}(v) \le \left(1 + \eps' \right) np$ for all vertices~$v$.
\end{claim}
Now, we think of $G(p)$ as a subgraph of $G(n,p)$ in the natural way. To be more precise, consider a graph $G'\sim G(n,p)$ and consider the spanning subgraph $H'$ of $G'$ whose edge set is $E(G')\cap E(G)$. Clearly, $H'\sim G(p)$. Using the above claims, \whp $H'$ is an $\alpha'$-residual spanning subgraph of $G'$ for $\alpha' \coloneqq (1-2\eps')\alpha$. If $\alpha' \ge \frac{1}{2} + \frac{1}{2r}$, \Cref{cor:G(p)} follows immediately from \Cref{thm:main} with $(\alpha, \eps) = (\alpha', \frac{\eps}{2})$, since 
\[
\left(1-\frac{\eps}{2}\right)\min\left\{2\alpha' - 1, \frac{\alpha'}{r}, \frac{2}{r+1} \right\} \ge (1-\eps)\min\left\{2\alpha - 1, \frac{\alpha}{r}, \frac{2}{r+1} \right\}.
\]
Note that if $\alpha' < \frac{1}{2} + \frac{1}{2r}$, then it is easy to check that $(1-\eps)(2\alpha - 1) < \frac{1}{r}$. Also, by \Cref{thm:Dirac theorem in random graphs}, \whp the graph $H'$ contains a Hamilton cycle. Thus, the trivial lower bound that any Hamilton cycle in $H'$ with any $r$-edge colouring has some colour appearing at least $\frac{n}{r}$ times suffices to finish the proof of \Cref{cor:G(p)}.
\end{proof}

\bibliographystyle{plainurl}
\bibliography{mybib}

\end{document}